\documentclass[11pt]{amsart}
\usepackage[T1]{fontenc}
\usepackage{amsmath}
\usepackage{amsfonts}
\usepackage{amssymb}
\usepackage{amsopn}
\usepackage{amscd}
\usepackage[applemac]{inputenc}
\newtheorem{thm}{Theorem}[section]
\newtheorem{prop}[thm]{Proposition}
\newtheorem{lem}[thm]{Lemma}
\newtheorem{cor}[thm]{Corollary}

\theoremstyle{definition}
\newtheorem{definition}[thm]{Definition}

\theoremstyle{remark}

\newtheorem{remark}[thm]{Remark}

\numberwithin{equation}{section}

\newcommand{\R}{\mathbb{R}}  
\DeclareMathOperator{\dist}{dist} 
\DeclareMathOperator{\diam}{diam} 

\DeclareMathOperator{\supp}{supp} 
\newcommand{\N}{\mathbb N}
\newcommand{\Z}{\mathbb Z}
\DeclareMathOperator{\fr}{fr} 
\begin{document}

\title{Local Tb theorems and Hardy inequalities}

\author{P. Auscher}

\address{Pascal Auscher, Université Paris-Sud, laboratoire de Math\'ematiques, UMR 8628, Orsay F-91405; CNRS, Orsay, F-91405; Centre for Mathematics and its Applications, Australian National University, Canberra ACT 0200, Australia}
\email{pascal.auscher@math.u-psud.fr}

\author{E. Routin}
\address{Eddy Routin, Université Paris-Sud, laboratoire de Math\'ematiques, UMR 8628, Orsay F-91405; CNRS, Orsay, F-91405; Centre for Mathematics and its Applications, Australian National University, Canberra ACT 0200, Australia}
\email{eddy.routin@math.u-psud.fr}

\begin{abstract}
In the setting of spaces of homogeneous type, we give a direct proof of the local $Tb$ theorem for singular integral operators. Motivated by questions of S. Hofmann, we extend it to the case when the integrability conditions are lower than $2$, with an additional weak boundedness type hypothesis, which incorporates some Hardy type inequalities. The latter can be obtained from some geometric conditions on the homogeneous space. For example, we prove that the monotone geodesic property of Tessera suffices.
\end{abstract}

\subjclass[2010]{42B20}

\keywords{Local Tb theorem, singular integral operators, space of homogeneous type, Hardy type inequalities}

\maketitle

\tableofcontents

\section{Introduction}

\medskip

The goal of this paper is to study a question raised by S. Hofmann in \cite{Hofmann2} (question $3.3.1$). The question is whether in existing local $Tb$ theorems, one can weaken the integrability conditions on the accretive systems to any exponent greater than $1$. So far, the known arguments for, say, antisymmetric kernels, \cite{Christ}, \cite{AHMTT}, \cite{Hofmann}, \cite{AY}, \cite{TanYan}, handle exponents greater than or equal to $2$ and no less. Motivations in \cite{HM} are towards obtaining uniform rectifiability of $n-$dimensional Ahlfors-David regular sets which are seen as boundaries of domains satisfying some interior access condition to the boundary, and whose Poisson kernel satisfies some scale invariant $L^p$ estimate for $p$ larger than $1$ and $p$ close to $1$. Here, we work on spaces of homogeneous type with scalar operators. We do not consider non-doubling spaces or Banach space valued theory, where $Tb$ theorems are proved under $L^\infty$ or $BMO$ control on $b$, and we refer to the work of Nazarov, Treil and Volberg \cite{NTV1}, \cite{NTV2}, and T. Hytönen \cite{Hyt1}, \cite{Hyt3}. The interest of local $Tb$ theorems over the "global" $Tb$ theorem of David, Journé and Semmes \cite{DJS} (and M$^c$Intosh, Meyer \cite{Mc-M} where it was introduced for the first time) is that there is no need for producing para-accretive functions which are unavoidable in such a formulation (see the work of Han and Sawyer \cite{HS}). And in application, the local statement is much easier to use. An occurrence of this is in the work of the first author with Alfonseca, Axelsson, Hofmann, Kim \cite{AAAHK}. 

We partially answer Hofmann's question with our Theorem \ref{AR}. That is, we provide an argument that works for all integrability exponents. This argument is based on the Beylkin-Coifman-Rokhlin algorithm \cite{BCR} (see also the work of T. Figiel \cite{Fig}), but in adapted Haar wavelets rather than the normal ones. However, as we shall see, our proof works at the expense of a supplementary weak boundedness property assumption when the integrability exponents are close to $1$. As a matter of fact, we think that the statement proposed in \cite{Hofmann2} is not correct for exponents close to $1$, due to the inapplicability of Hardy inequalities, without a further hypothesis (which appears in the analysis of one term, which is not susceptible to cancel with other controlled terms). It is thus interesting to isolate this hypothesis but it could be hard to check. The understanding of this issue might require some extra efforts.

Historically, local $T(b)$ theorems were introduced and proved on spaces of homogeneous type by M. Christ \cite{Christ} with $L^{\infty}$ bounds on the accretive systems. The first author, Hofmann, Muscalu, Tao and Thiele \cite{AHMTT} found a generalization for all exponents but for model singular integral operators called perfect dyadic. This argument was carried out to standard singular integral operators by S. Hofmann \cite{Hofmann} but a restriction on the exponents appeared, to be able to use Hardy inequalities. Then the first author and Yang \cite{AY} were able to find a different argument improving on exponents, still with a restriction though. This argument was carried out to spaces of homogeneous type by Tan and Yan \cite{TanYan}, we mention however some gaps there, as Hardy inequalities need to be proved or assumed on such spaces.

This gives us the opportunity to study, on a space of homogeneous type, inequalities (which we call of Hardy type) 
\[ \left | \int_I \int_J {\frac{f(y)g(x)}{\mathrm{Vol}(B(x,\dist(x,y)))} d\mu(x) d\mu(y)} \right | \leq C \|f\|_{L^p(I,d\mu)} \|g\|_{L^{p'}(J,d\mu)}  \]
for $I\cap J = \varnothing,$ and $1<p<+\infty$, $1/p + 1/{p'} = 1$. Our terminology comes from the fact that in Euclidean spaces, such inequalities follow from the well known $1-$dimensional Hardy inequality. It turns out that they hold without any restriction if $I,J$ are Christ's dyadic cubes (see Section $2$). This seems not to have been noticed in the literature for dual pairs of exponents and we prove it in Section $9$. It depends in particular on the small layers for dyadic cubes. However, if $I$ is a ball $B$ and $J$, say, $2B \backslash B$, then it is not clear in general. It clearly depends on how $B$ and $2B \backslash B$ see each other through their boundary. We prove that some small boundary hypothesis (called the relative layer decay property) suffices. We also show that this property holds in all doubling complete Riemannian manifolds, length spaces and more generally in any monotone geodesic space of homogeneous type. The latter notion arose in geometric measure theory from the work of R. Tessera \cite{Tess}, and was recently proved by Lin, Nakai and Yang \cite{LNY} to be equivalent to a chain ball notion introduced by S. Buckley \cite{Buck}.

The paper is organized as follows. In Section $2$ we recall some basic definitions and results in spaces of homogeneous type, such as the existence of dyadic cubes and the definition of singular integral operators. In Section $3$ we state our main results, that is Theorem \ref{AR} (main $Tb$ theorem with hypotheses on dyadic cubes), Theorem \ref{ARBalls} ($Tb$ theorem with hypotheses on balls), and Theorem \ref{ARjut} (relaxing support conditions on accretive systems), and comment our hypotheses. We prove Theorem \ref{ARBalls} and Theorem \ref{ARjut} in Section $4$, while the proof of Theorem \ref{AR} is developed over the three subsequent \mbox{sections:} we introduce some notations and give some preliminary results in Section $5$, notably the existence of adapted Haar wavelets in a space of homogeneous type, explain important reductions in Section $6$, and finally detail the algorithm  and estimate all the terms involved in Section $7$. We then devote Section $8$ to the study of two particular cases: the case of perfect dyadic operators and the case when the exponents are not too small, that is $1/p + 1/q \leq 1$, for which the proof is much easier. Finally, we have a closer look at Hardy type inequalities in Section $9$. 

This work is part of a doctorate dissertation of the second author. It was conducted at Université Paris-Sud and the Center for Mathematics and its Applications (CMA) at ANU. The authors are particularly grateful to CMA for their warm hospitality  and financial support during their visit. The authors also thank S. Hofmann, T. Hytönen, J.-M. Martell for discussions related to this work.

\medskip

\section{Notations and preliminaries}

\medskip

Throughout this work, we assume that $(X,\rho,\mu)$ is a space of homogeneous type, that is, $X$ is a set equipped with a metric $\rho$ and a non-negative Borel measure $\mu$ on $X$ for which there exists a constant $C_D <+\infty$ such that all the associated balls $B(x,r)=\{y \in X ; \rho(x,y) <r \}$ satisfy the doubling property
\[ 0 < \mu (B(x,2r)) \leq C_D \mu (B(x,r)) < \infty  \]
for every $x \in X$ and $r>0$. Contrary to Christ's algorithm, ours allows the presence of atoms, that is points $x$ with $\mu(\{x\}) \neq 0$. In fact, all points could be atoms if one wishes. We will use the notation $A\lesssim B$ (resp. $A \approx B$) to denote the estimate $A\leq CB$ (resp. $(1/C) B \leq A \leq C B$) for some absolute constant $C$ which may vary from line to line.

\subsection{Dyadic cubes}The following result, due to M. Christ (see \cite{Christ}), states the existence of sets analogous to the dyadic cubes of $\R^n$ in a space of homogeneous type.
\begin{lem}\label{cubes}
There exist a collection of open subsets $\{ Q_{\alpha}^j \subset X: j \in \mathbb{Z}, \alpha \in I_j \}$, where $I_j$ denotes some (possibly finite) index set depending on $j$, and constants $0 < \delta < 1$, $a_0 > 0$, $\eta > 0$, and $C_1,C_2 < +\infty$ such that
\begin{enumerate}
\item For all $j \in \mathbb{Z}$, $\mu(\{ X \backslash \bigcup_{\alpha \in
I_j}{Q_{\alpha}^j} \} ) = 0$.
\item If $j < j'$ , then either $Q_{\beta}^{j'} \subset
Q_{\alpha}^{j}$, or $Q_{\beta}^{j'} \cap Q_{\alpha}^{j} = \varnothing$.
\item For each $(j,\alpha)$ and each $j'<j$ there is a unique $\beta$ such that $Q_{\alpha}^j \subset Q_{\beta}^{j'}$.
\item For each $(j,\alpha)$, we have  $\mathrm{diam}(Q_{\alpha}^j)
\leq C_1 {\delta}^j$.
\item Each $Q_{\alpha}^j$ contains some ball $B(z_{\alpha}^j,a_0
{\delta}^j)$. We say that $z_{\alpha}^j$ is the center of the cube $Q_{\alpha}^j.$
\item Small boundary condition:
\[ \mu \left( \left\{ x \in Q_{\alpha}^j: \rho (x,X \backslash Q_{\alpha}^j) \leq t {\delta}^j \right\} \right) \leq C_2 t^{\eta} \mu(Q_{\alpha}^j) \quad \forall j,\alpha, \quad \forall t>0. \]
\end{enumerate}
\end{lem}

We will call those open sets dyadic cubes of the space of homogeneous type $X$. For a cube $Q=Q^j_{\alpha}$, $j$ is called the generation of $Q$, and we set $l(Q)= \delta^j$. By $(4)$ and $(5)$, $l(Q)$ is comparable to the diameter of $Q$, and we call it, in analogy with $\R^n$, the length of $Q$.  Whenever $Q^{j+1}_{\alpha} \subset Q^j_{\beta}$, we will say that $ Q^{j+1}_{\alpha}$ is a child of $Q^{j}_{\beta}$, and $Q^{j}_{\beta}$ the parent of $Q^{j+1}_{\alpha}$. For every dyadic cube $Q$, the notation $\widetilde{Q}$ denotes the collection of all the children of $Q$. It is easy to check that each dyadic cube has a number of children uniformly bounded. We remark that all atoms must be isolated points by the doubling condition. In case $x$ is an atom, there exists a smallest integer $j$ such that $\{ x \}$ is a dyadic cube $Q_{\alpha}^j$ of generation $j$. In that case, $Q^j_{\alpha}$ has only one child $Q_{\beta}^{j+1} = Q^j_{\alpha}$ as set, $Q_{\beta}^{j+1}$ has only one child, and so on.

A neighbor of $Q$ is any dyadic cube $Q'$ of the same generation with $\rho(Q,Q') < l(Q)$. The notation $\widehat{Q}$ will denote the union of $Q$ and all its neighbors. It is clear that $Q$ and $\widehat{Q}$ have comparable measures. It is easy to check that a cube $Q$ has a number of neighbors that is uniformly bounded. 

\medskip

\subsection{Singular integral operators}For $1 \leq p \leq \infty $, the space of $p$-integrable complex valued functions on $X$ with respect to $\mu$ is denoted by $L^p(X)$, the norm of a function $f \in L^p(X)$ by $\|f\|_p$, the duality bracket given by $\langle f,g\rangle = \int_X{fg d\mu}$ (we do mean the bilinear form), and the mean of a function $f$ on a set $E$ denoted by $[f]_E = \mu(E)^{-1} \int_E{f d\mu}.$ For any $x,y \in X$, we set
\[ \lambda (x,y) = \mu (B(x, \rho (x,y))). \]
It is easy to see that $\lambda(x,y)$ is comparable to $\lambda(y,x)$, uniformly in $x, y.$

\begin{definition} 
A standard Calder\'on-Zygmund kernel on $X$ is a function $ K:  X \times X \backslash \{ x=y \} \rightarrow \mathbb{C} $ such that there exists a constant $\alpha >0$ for which
\begin{equation}\label{standard1}
|K(x,y)| \lesssim \frac{1}{\lambda(x,y)} 
\end{equation}
and
\begin{equation}\label{standard2}
| K(x,y) - K(x',y) | + | K(y,x) - K(y,x') | \lesssim \left ( \frac{\rho (x,x')}{\rho (x,y)} \right )^{\alpha} \frac{1}{\lambda(x,y)}
\end{equation}
whenever $\rho (x,x') \leq \rho (x,y) /2$, and $\rho(x,y)>0$.
\end{definition}
Denote by $D_{\alpha}$ the space of all Hölder continuous functions of order $\alpha \in (0,1]$ with compact support and $D'_{\alpha}$ the dual space of $D_{\alpha}$ (we refer to \cite{Christ} for the detail). A singular integral operator (sio) $T$ on $X$ is a continuous mapping from $D_{\alpha}$ to $D'_{\alpha}$ which is associated to a standard kernel $K(x,y)$, in the sense that
\[ \langle Tf,g \rangle = \iint K(x,y)f(y)g(x) d\mu(x) d\mu(y) \]
whenever $f,g \in D_{\alpha}$ with disjoint supports. Standard computations and a density argument show that one can extend $T$ from $L^p(K)$ into $L^{\infty}_{loc}(K^c)$ for any compact $K$, so that
\begin{equation} \label{etoile}
Tf(x) = \int {K(x,y)f(y)d\mu(y)} 
\end{equation}
for all $f \in L^p(X)$ with $\supp f \subset K$ and almost all $x \notin K$. In the following, $T^{\ast}$ will denote the operator adjoint to $T$. Let us state the well known standard estimates for singular integral operators.

\begin{prop}[Standard Calder\'on-Zygmund estimates]
Let $T$ be a singular integral operator on $X$. Let $r>0, x_B \in X$ and let $B$ be the ball of radius $r$ centered in $x_B$. Let $g$ be a function of mean $0$ supported on $B$, and $f$ a function supported on the complement of $2B$. Then we have the standard estimate
\begin{equation} \label{standard estimate}
|\langle g , Tf \rangle | \lesssim \|g\|_1 \int_{(2B)^c}{\left ( \frac{r}{\rho(x_B,y)} \right )^{\alpha} \lambda(x_B,y)^{-1} |f(y)| d\mu(y)}.
\end{equation}
Similarly in the dyadic setting, if $Q$ is a dyadic cube in $X$ of center $z_Q$, $g$ a function of mean $0$ supported on $Q$ and $f$ a function supported on the complement of $\widehat{Q}$, we have
\begin{equation}\label{standard estimate dyadic}
|\langle g , Tf \rangle | \lesssim \|g\|_1 \int_{{\widehat{Q}}^c}{\left ( \frac{l(Q)}{\rho(z_Q,y)} \right )^{\alpha} \lambda(x_Q,y)^{-1} |f(y)| d\mu(y)}.
\end{equation}
\end{prop}

Over the course of this article, we will often use what we call Hardy type inequalities. The prototype of Hardy type inequalities is 
\[ \left |  \int_I \int _J {\frac{f(y)g(x)}{x-y}dydx}  \right | \leq C \left ( \int_I{|f(y)|^{\nu} dy} \right )^ {\frac{1}{\nu}} \left ( \int_J{|g(x)|^{\nu'} dx} \right )^ {\frac{1}{\nu'}},   \]
when $I,J$ are adjacent intervals, $\supp f \subset I$, $\supp g \subset J$, and $1<\nu<\infty$, $\nu' = \frac{\nu}{\nu-1}$. This easily follows from the boundedness of the Hardy operator $Hf(x) = \frac{1}{x} \int_0^x {f(t)dt}$ (hence our terminology), and does not use regularity of $\frac{1}{x-y}$. Thus one can hope to extend such inequalities on spaces of homogeneous type with kernel $1/{\lambda(x,y)}$ and measure $\mu$ when $I,J$ are reasonable disjoint subsets of $X$. Let us state such a result in the dyadic setting.

\begin{lem}\label{Hardydyadic} 
Let $Q, Q'$ be two disjoint dyadic cubes in $X$. Let $1< \nu < +\infty$, with dual exponent $\nu'$. There exists $C <+\infty$ such that for all function $f$ supported on $Q$, $f\in L^{\nu}(Q)$, and all function $g$ supported on $Q'$, $g \in L^{\nu'}(Q')$, we have
\begin{equation}  \label{2etoile}
\int_{Q'} \int _Q {\frac{|f(y)g(x)|}{\lambda(x,y)}d\mu(x)d\mu(y)}  \leq C \|f\|_{\nu}  \|g\|_{\nu'}. 
\end{equation}
The constant $C$ only depends on $C_D$ and $\nu$.
\end{lem}

We refer to Section $9$ for the proof of Lemma \ref{Hardydyadic}. Observe that this result immediately yields the following corollary regarding singular integral operators.

\begin{lem} \label{Hardy}
Let $T$ be a singular integral operator on $X$. Let $1< \nu < +\infty$ with dual exponent $\nu'$. 
\begin{itemize}
\item There exists $C <+\infty$ such that for every disjoint dyadic cubes $Q$,$Q'$, and every functions $f,g$ respectively supported on $Q,Q'$, with $f \in L^{\nu}(Q)$, $g \in L^{\nu'}(Q'),$ we have
\begin{equation} \label{Hardy1}
| \langle Tf , g \rangle | \leq C \|f\|_{L^{\nu}(Q)}  \|g\|_{L^{\nu'}(Q')}.
\end{equation}
\item There exists $C <+\infty$ such that for every dyadic cube $Q$, and every function $f$ supported on $Q$ with $f \in L^{\nu}(Q)$, we have
\begin{equation}  \label{Hardy2}
\| Tf \|_{L^{\nu}(\widehat{Q} \backslash Q)} \leq C \|f\|_{L^{\nu}(Q)}. 
\end{equation}
\end{itemize}
\end{lem}

\begin{proof} 
Use \eqref{standard1}, Lemma \ref{Hardydyadic}, and for the second part the fact that the number of neighbors of any given cube is uniformly bounded.
\end{proof}

\medskip

\section{Main results and comments}

\medskip

Our main result is the following.

 \begin{thm}\label{AR}
Let $ 1< p,q < + \infty$. Let $T$ be a singular integral operator with locally bounded kernel. Assume that there exists a $(p,q)$ dyadic pseudo-accretive system adapted to $T$. Then $T$ extends to a bounded operator on $L^2(X)$, with bounds independent of $\| K \|_{\infty,loc} \,$. 
\end{thm}

Let us explain what is a $(p,q)$ dyadic pseudo-accretive system adapted to $T$.

\begin{definition}{$(p,q)$ dyadic pseudo-accretive system.} \\
Let $ 1< p,q < + \infty$ with dual exponents $p', q'$, and let $T$ be a singular integral operator on $X$. We say that a collection of functions $(\{b^1_Q\}_Q,\{b^2_Q\}_Q)$ is a $(p,q)$ dyadic pseudo-accretive system adapted to $T$ if there exists a constant $C_A  <+\infty$, such that for each dyadic cube Q, $b^1_Q$, $b^2_Q$ are supported on Q with
 \begin{equation}\label{accretive1}
 \int_Q{b^1_Q d\mu} = \mu(Q) = \int_Q{b^2_Q d\mu},
   \end{equation}
   \begin{equation}\label{accretive2}
 \int_Q{ \left (|b^1_Q|^p + |b^2_Q|^q \right ) d\mu} \leq C_A \, \mu(Q),
  \end{equation}
   \begin{equation}\label{accretive3}
 \int_{\widehat{Q}}{\left ( |T(b^1_Q)|^{q'} + |T^{\ast}(b^2_Q)|^{p'} \right ) d\mu} \leq C_A \, \mu(Q),
  \end{equation}
Furthermore, the functions $b^i_Q$ are required to satisfy the following properties: for all $C<+\infty$, there exists $C_H <+\infty$ and $1<\nu<+\infty$ such that for every $(Q,Q')$ dyadic cubes, for every dyadic cubes $R' \subset Q'$,  $R_n \subset (\widehat{R'}\backslash R') \cap Q$, $R_n$ mutually disjoint, with $\rho(R',R_n) < l(R_n)$, $[|b^1_{Q'}|^p]_{R'} \leq C$, $[|b^2_{Q}|^q ]_{R_n} \leq C$, and for every set of coefficients $(\alpha_{n})_n$, we have
\begin{equation}\label{WBP1}
\left | \left \langle b^2_Q \left ( \sum_n{\alpha_n 1_{R_n}} \right ), T(b^1_{Q'} 1_{R'})  \right \rangle \right | \leq C_H \left \| \sum_n{\alpha_n 1_{R_n}}  \right  \|_{\nu} \mu(R')^{\frac{1}{\nu'}}.
\end{equation}
We also need a control of the diagonal terms: for all $C<+\infty$, there exists $C_{WBP} <+\infty$ such that for every $(Q,Q')$ dyadic cubes, for every dyadic cube $R \subset Q \cap Q'$ with $[|b^1_{Q'}|^p]_{R} \leq C$, $[|b^2_{Q}|^q ]_{R} \leq C$, we have
\begin{equation}\label{WBP2}
\left | \left \langle b^2_Q 1_{R}  , T(b^1_{Q'} 1_{R})  \right \rangle \right | \leq C_{WBP} \, \mu(R).
\end{equation}
We require the $b^i_Q$ to satisfy also the symmetric properties, with respectively $b^1$ instead of $b^2$, $p$ instead of $q$, $q$ instead of $p$, and $T^{\ast}$ instead of $T$.
\end{definition}

We remark that the statement has a converse. If $T$ is bounded, then the collection $(\{1_Q\}_Q,\{1_Q\}_Q)$ is a $(p,q)$ accretive system adapted to $T$ for any exponents. Several comments are in order. Let us begin with the case $1/p +1/q >1.$
\begin{enumerate}
\item In this case, we cannot use the Hardy inequality \eqref{2etoile} with exponents $p$ and $q$ replacing $\nu$ and $\nu'$. So our hypotheses \eqref{WBP1} and \eqref{WBP2} are a substitute for the missing \eqref{2etoile}. In practice, they could very well hold due to specific relations or cancellations between the $b^i_Q$ and $T$.
\item Observe that, while being a rather unsatisfactory condition, \eqref{WBP1} and \eqref{WBP2} imply the following weaker statement: for every dyadic subcubes $R,R'$ of cubes $Q,Q'$, such that $R,R'$ are neighbors, and $b^1_{Q'}$, $b^2_Q$ satisfy the same size estimates as above on $R$, $R'$, we have
\begin{equation} \label{WBPC}
|\langle b^2_Q 1_R, T(b^1_{Q'} 1_{R'})  \rangle | \lesssim \mu(R').
\end{equation}
This weaker property is more satisfactory as it is a lot closer to what we are used to calling a weak boundedness property. Unfortunately, \eqref{WBPC} suffices for all but one term we could not estimate otherwise than assuming the stronger \mbox{property \eqref{WBP1}.} 
\item Note that it is not clear whether for any systems $(b^1_Q),(b^2_Q)$ satisfying \eqref{accretive1} and \eqref{accretive2}, if $T$ is bounded on $L^2(X)$ then \eqref{accretive3}, \eqref{WBP1}, \eqref{WBP2} hold. But a $(p,q)$ accretive system adapted to $T$ is, as its name indicates, not any pair of systems.
\end{enumerate} 
Let us now assume $1/p +1/q \leq 1$. First, we show that \eqref{WBP1} is necessarily satisfied as a consequence of \eqref{accretive2} and \eqref{accretive3}, and it is an application of Lemma \ref{Hardy}. This is stated in the following proposition.

\begin{prop} \label{particular}
Let $ 1< p,q < + \infty$ with dual exponents $p', q'$, and such that $1/p + 1/q \leq 1$. Let $T$ be a singular integral operator on $X$. Suppose that $(b^1_Q), (b^2_Q)$ constitute a collection of functions supported on $Q$, satisfying \eqref{accretive1}, \eqref{accretive2}, and the following weaker form of \eqref{accretive3}, 
\begin{equation} \label{accretive3bis} 
\int_{Q}{\left ( |T(b^1_Q)|^{q'} + |T^{\ast}(b^2_Q)|^{p'} \right ) d\mu} \leq C_A \, \mu(Q).
\end{equation}
Then the functions $b^i_Q$ satisfy \eqref{accretive3} and \eqref{WBP1}.
\end{prop}
\begin{proof}
We first prove \eqref{accretive3}. By \eqref{Hardy2}, we have
\begin{align*} 
\int_{\widehat{Q}\backslash Q}{\left ( |T(b^1_Q)|^{q'} + |T^{\ast}(b^2_Q)|^{p'} \right ) d\mu} & \lesssim \int_Q{(|b^1_Q|^{q'} + |b^2_Q|^{p'}  )d\mu}\\
& \lesssim \left ( \int_Q{|b^1_Q|^p d\mu} \right)^{\frac{q'}{p}} \mu(Q)^{1- \frac{q'}{p}} \ \ + \ \ \left ( \int_Q{|b^2_Q|^q d\mu} \right)^{\frac{p'}{q}} \mu(Q)^{1- \frac{p'}{q}}\\
& \lesssim \mu(Q), 
\end{align*}
where we have applied the Hölder inequality to get the second inequality, which is made possible only because $1/p + 1/q \leq 1$ implies that $p/{q'}, q/{p'} \geq 1$. The inequality \eqref{accretive3} follows. Now we prove \eqref{WBP1} with $\nu = q$ (and the symmetrical inequality will hold for $\nu =p$). Write 
\begin{align*}
\left | \left \langle b^2_Q \left ( \sum_n{\alpha_n 1_{R_n}} \right ), T(b^1_{Q'} 1_{R'})  \right \rangle \right | & \lesssim \left \|  b^2_Q \left ( \sum_n{\alpha_n 1_{R_n}}   \right ) \right \|_q \| T(b^1_{Q'} 1_{R'})  \|_{L^{q'}(\widehat{R'} \backslash R')}\\
\end{align*}
By \eqref{Hardy2}, and using again the fact that $1/p +1/q \leq 1$,
\[ \| T(b^1_{Q'} 1_{R'})  \|_{L^{q'}(\widehat{R'}\backslash R')} \lesssim \|b^1_{Q'}  \|_{L^{q'}(R')} \lesssim \mu(R')^{\frac{1}{q'}}. \]
Moreover, because the $R_n$ are disjoint,
\begin{align*}
\left \|  b^2_Q \left ( \sum_n{\alpha_n 1_{R_n}}   \right ) \right \|_q & \lesssim \left (  \sum_n {|\alpha_n |^q \int_{R_n}{|b^2_Q|^q d\mu  }}  \right )^{\frac{1}{q}} \lesssim \left (  \sum_n {|\alpha_n |^q \mu(R_n)}  \right )^{\frac{1}{q}} \lesssim  \left \| \sum_n{\alpha_n 1_{R_n}}  \right  \|_{q}.\\
\end{align*}
\end{proof}
As formulated, the inequality \eqref{WBP2} is not a direct consequence of \eqref{accretive1}, \eqref{accretive2} and \eqref{accretive3bis}. For this, one needs further control for $b^1_{Q'},T(b^1_{Q'}),b^2_Q, T^{\ast}(b^2_Q)$ on $R$ than the one written in \eqref{accretive3bis}. This can be achieved (see Section $8.2$). In other words, when $1/p+1/q \leq 1$, a possible definition of a $(p,q)$ accretive system adapted to $T$ to prove Theorem \ref{AR} is \eqref{accretive1}, \eqref{accretive2} and \eqref{accretive3bis}. In particular, this covers $p=q=2$.

Our argument to prove Theorem \ref{AR} involves using the BCR algorithm introduced in \cite{BCR}, applied with Haar wavelets adapted to the dyadic pseudo-accretive system $(b^i_Q)$, which allows us to obtain a direct proof without having to use the decomposition of a singular integral operator as the sum of a bounded  operator and a perfect dyadic singular integral operator used in \cite{AY} and \cite{TanYan}. We will develop this in the following sections.

Since dyadic cubes can be ugly sets in practice, on which checking $\eqref{accretive3}$, \eqref{WBP1} or \eqref{WBP2} might be difficult, a natural question is whether or not one can switch from dyadic hypotheses in Theorem \ref{AR} to hypotheses made on balls. The following Theorem \ref{ARBalls} deals with this concern, but first let us give one more definition.

\begin{definition}{Hardy property.}\\ \label{HardyBalls}
Let $(X,\rho,\mu)$ be a space of homogeneous type. We say that $X$ has the Hardy property if for every $1<\nu< +\infty$, with dual exponent $\nu'$, there exists $C < +\infty$ such that for every ball $B$ in $X$, with $2B$ denoting the concentric ball with double radius, and all functions $f$ supported on $B$, $f\in L^{\nu}(B)$, $g$ supported on $2B \backslash B$, $g \in L^{\nu'}(2B \backslash B)$, we have
\begin{equation}  \label{3etoile}
\int_{B} \int _{2B\backslash B} {\frac{|f(y)g(x)|}{\lambda(x,y)}d\mu(x)d\mu(y)}  \leq C \|f\|_{\nu}  \|g\|_{\nu'}.
\end{equation}
\end{definition}
Obviously, if $X$ has the Hardy property, \eqref{3etoile} will remain true with $2B$ replaced by $cB$ for any $c>1$, with a different constant $C$. We refer to Section $9.2$ for geometric conditions ensuring that a space of homogeneous type has the Hardy property.
\medskip

\begin{thm}\label{ARBalls}
Let $ 1< p,q < + \infty$ with dual exponents $p', q'$, such that $1/p + 1/q \leq 1$. Let $T$ be a singular integral operator with locally bounded kernel. Assume that there exists a collection of functions $(\{b^1_B\}_B,\{b^2_B\}_B)$, such that there exists a constant $C< +\infty$ such that for every ball $B$ in $X$, $b^i_B$ is supported on $B$, and we have
 \begin{equation}\label{accretive1balls}
 \int_B{b^1_B d\mu} = \mu(B) = \int_B{b^2_B d\mu},
   \end{equation}
   \begin{equation}\label{accretive2balls}
 \int_B{ \left (|b^1_B|^p + |b^2_B|^q \right ) d\mu} \leq C \, \mu(B),
  \end{equation}
   \begin{equation}\label{accretive3balls}
 \int_{X}{\left ( |T(b^1_B)|^{q'} + |T^{\ast}(b^2_B)|^{p'} \right ) d\mu} \leq C \, \mu(B),
  \end{equation}
Then $T$ extends to a bounded operator on $L^2(X)$, with bounds independent of $\| K \|_{\infty,loc}$.\\
Furthermore, if \eqref{accretive3balls} is replaced by the weaker uniform bound
\begin{equation}  \label{accretive3ballsbis}
   \int_{B}{\left ( |T(b^1_B)|^{q'} + |T^{\ast}(b^2_B)|^{p'} \right ) d\mu} \leq C \, \mu(B),
\end{equation}
then the conclusion still holds provided $X$ has the Hardy property.
\end{thm}

\medskip
\begin{remark}
\begin{enumerate}
\item In \cite{TanYan}, the authors state Theorem \ref{ARBalls}, with hypothesis \eqref{accretive3ballsbis}, but they do not assume $X$ has the Hardy property. They justify their statement by reducing to a dyadic pseudo-accretive system as in our proof of Theorem \ref{ARBalls} (see the following section). It might have been a bit over-optimistic, as we do not see how to obtain \eqref{accretive3} from \eqref{accretive3ballsbis} without using a Hardy type inequality, and the Hardy property \eqref{3etoile} might not always be satisfied in a general space of homogeneous type.
\item We had to assume that $1/p +1/q \leq1$. Indeed, when $1/p+1/q>1$, the incompatibility of exponents $p,q$ makes things tricky, and we cannot see a way of adapting \eqref{WBP1} and \eqref{WBP2} to the balls setting. 
\item We will prove in the next section that integrability over $X$ is equivalent to integrability over $CB$ for any $C>1$, and it obviously implies integrability over $B$. Conversely though, integrability over $B$ does not necessarily imply integrability over $X$. It does when a Hardy type inequality on balls is satisfied in the space $X$. Some particular relation between $T$ and the $b^i_B$ could also be a substitute.
\end{enumerate}
\end{remark}

\medskip

Finally, a natural question is whether one can relax the support condition on the accretive system and impose that the $b^i_Q$ are only supported in a slight enlargment, say $\widehat{Q}$, of $Q$. We answer this question with following Theorem \ref{ARjut}\footnote{We thank T. Hytönen for his suggestion which led us to the formulation of this theorem.}. 

\begin{thm}\label{ARjut}
Let $1< p,q < + \infty$ with dual exponents $p', q'$. Let $T$ be a singular integral operator with locally bounded kernel. Assume that there exists a collection of functions $(\{b^1_Q\}_Q,\{b^2_Q\}_Q)$, and a constant $C< +\infty$ such that for every dyadic cube $Q$ in $X$, $b^i_Q$ is supported on $\widehat{Q}$, and we have
 \begin{equation}\label{accretive1jut}
 \int_{\widehat{Q}}{b^1_Q d\mu} = \mu(Q) = \int_{\widehat{Q}}{b^2_Q d\mu},
   \end{equation}
   \begin{equation}\label{accretive2jut}
 \int_{\widehat{Q}}{ \left (|b^1_Q|^p + |b^2_Q|^q \right ) d\mu} \leq C \, \mu(Q),
  \end{equation}
   \begin{equation}\label{accretive3jut}
 \int_{X}{\left ( |T(b^1_Q)|^{q'} + |T^{\ast}(b^2_Q)|^{p'} \right ) d\mu} \leq C \, \mu(Q).
  \end{equation}
Suppose as well that for all $C<+\infty$, there exists $C_H <+\infty$ and $1<\nu<+\infty$ such that for every $(Q,Q')$ dyadic cubes, for every dyadic cubes $R' \subset \widehat{Q'}$, $R_n \subset (\widehat{R'}\backslash R') \cap \widehat{Q}$, $R_n$ mutually disjoint, with $\rho(R',R_n) < l(R_n)$, $[|b^1_{Q'}|^p]_{R'} \leq C$, $[|b^2_{Q}|^q ]_{R_n} \leq C$, and for every set of coefficients $(\alpha_{n})_n$, we have
\begin{equation}\label{WBP1jut}
\left | \left \langle b^2_Q \left ( \sum_n{\alpha_n 1_{R_n}} \right ), T(b^1_{Q'} 1_{R'})  \right \rangle \right | \leq C_H \left \| \sum_n{\alpha_n 1_{R_n}}  \right  \|_{\nu} \mu(R')^{\frac{1}{\nu'}}. 
\end{equation}
Suppose also that for all $C<+\infty$, there exists $C_{WBP} <+\infty$ such that for every $(Q,Q')$ dyadic cubes, for every dyadic cube $R \subset \widehat{Q} \cap \widehat{Q'}$ with $[|b^1_{Q'}|^p]_{R} \leq C$, $[|b^2_{Q}|^q ]_{R} \leq C$, we have
\begin{equation}\label{WBP2jut}
\left | \left \langle b^2_Q 1_{R}  , T(b^1_{Q'} 1_{R})  \right \rangle \right | \leq C_{WBP} \, \mu(R).
\end{equation}
We naturally require the $b^i_Q$ to satisfy the symmetric properties, with respectively $b^1$ instead of $b^2$, $p$ instead of $q$, $q$ instead of $p$, and $T^{\ast}$ instead of $T$.

\noindent Then $T$ extends to a bounded operator on $L^2(X)$, with bounds independent of $\| K \|_{\infty,loc}$.\\

\end{thm}

\begin{remark}
As for Theorem \ref{AR}, when $1/p +1/q \leq 1$, we have a simpler formulation because one needs not \eqref{WBP1jut} and \eqref{WBP2jut} and the result holds only assuming \eqref{accretive1jut}, \eqref{accretive2jut} and \eqref{accretive3jut}.
\end{remark}
\medskip

\section{Proofs of Theorem \ref{ARBalls} and Theorem \ref{ARjut}}

\medskip

\begin{proof}[Proof of Theorem \ref{ARBalls}]
Theorem \ref{ARBalls} is a direct consequence of Theorem \ref{AR} in the particular case when $1/p+1/q \leq 1$. In this case, one needs not check hypotheses \eqref{WBP1} and \eqref{WBP2} as we remarked earlier (see Section $8.2$ for the detail). We begin by proving the first part of Theorem \ref{ARBalls}. Assuming there exists a pseudo-accretive system $(\{b^1_B\},\{b^2_B\})$ on the balls $B$ of $X$ satisfying \eqref{accretive1balls}, \eqref{accretive2balls} and \eqref{accretive3balls}, for $i=1,2$, and $Q$ a dyadic cube, let us consider the functions $b^i_{B_Q}$ where $B_Q$ is a ball contained in $Q$ of radius comparable to the diameter of $Q$. The existence of such a ball is given by property $(5)$ of Lemma \ref{cubes}. Then, normalizing the $b^i_{B_Q}$ so that they have mean $1$ on $Q$, we obtain a collection of functions $b^i_Q$, supported on $Q$, and that satisfy
\[ [b^i_Q  ]_Q =1, \]
\[ \int_Q{(|b^1_Q|^p +|b^2_Q|^q)d\mu} \lesssim \int_B{(|b^1_{B_Q}|^p +|b^2_{B_Q}|^q)d\mu} \lesssim \mu(Q), \]
\[  \int_{\widehat{Q}}{\left ( |T(b^1_Q)|^{q'} + |T^{\ast}(b^2_Q)|^{p'} \right ) d\mu} \lesssim \int_X{\left ( |T(b^1_{B_Q})|^{q'} + |T^{\ast}(b^2_{B_Q})|^{p'} \right ) d\mu} \lesssim \mu(Q). \]
Thus, applying Theorem \ref{AR} in the particular case when $1/p + 1/q \leq 1$, we obtain the boundedness of $T$ on $L^2(X)$. For the second part of Theorem \ref{ARBalls}, we prove that if $X$ has the Hardy property then \eqref{accretive3ballsbis} implies \eqref{accretive3balls}. Observe first that as a consequence of \eqref{3etoile} and of the fact that $1/p+1/q \leq1$, we have
\[ \int_{2B\backslash B} {\left ( |T(b^1_B)|^{q'} + |T^{\ast}(b^2_B)|^{p'} \right ) d\mu} \lesssim \mu(B).\]
The result is then a direct application of the following lemma.
\begin{lem} \label{lemmesio}
Let $T$ be a singular integral operator on $X$. Let $\alpha >1$, let $B$ be a ball in $X$, and $f_B$ a function supported on $B$ such that $\|f_B\|_{L^1(B)} \lesssim \mu(B)$. Then $T(f_B) \in L^{\nu}(X\backslash \alpha B)$ and $\|T(f_B)\|^{\nu}_{L^{\nu}(X \backslash \alpha B)} \lesssim \mu(B)$ for all $1 < \nu < +\infty$.
\end{lem}
Admitting Lemma \ref{lemmesio}, one only has to apply it to the functions $b^i_B$, with $\alpha=2$, and $\nu_1 = q', \nu_2 = p'$, to obtain
\[ \int_{X\backslash 2B} {\left ( |T(b^1_B)|^{q'} + |T^{\ast}(b^2_B)|^{p'} \right ) d\mu} \lesssim \mu(B). \]
Summing up, we obtain \eqref{accretive3balls} as desired. 
\end{proof}

\begin{proof}[Proof of Lemma \ref{lemmesio}]
For $x \in X \backslash 2B$, applying \eqref{etoile} and \eqref{standard1}, write
\[ |T(f_B) (x)| \lesssim \int_{B}{ \frac{|f_B(y)|}{\lambda(x,y)} d\mu(y)} \lesssim \frac{\mu(B)}{\mu(B(x,\rho(x,B)))}. \]
Now, for $j \geq 0$, set $B_j = 2^{j+1} \alpha B$ and $C_j = B_j \backslash B_{j-1}$, and split the integral over $X \backslash \alpha B$ as
\[ \int_{X \backslash \alpha B}{ |T(f_B) |^{\nu} d\mu}  \lesssim \mu(B)^{\nu} \sum_{j \geq 0}\frac{ \mu(C_j)}{{\mu(B_j)}^{\nu}}. \]
Splitting this sum into bundles for which $\mu(B_j)$ is comparable, we have
\begin{align*}
\sum_{j \geq 0}\frac{ \mu(C_j)}{{\mu(B_j)}^{\nu}} & = \sum_{k\geq 0} \quad \sum_{j:  2^k \mu(B) \leq \mu(B_j) < 2^{k+1} \mu(B)}{\frac{\mu(B_j) - \mu(B_{j-1})}{{\mu(B_j)}^{\nu}}}\\
& \leq  \sum_{k\geq 0} \frac{1}{(2^k \mu(B))^{\nu}}  \sum_{j:  2^k \mu(B) \leq \mu(B_j) < 2^{k+1} \mu(B)}{(\mu(B_j) - \mu(B_{j-1}))}\\
& \leq \sum_{k \geq 0} {\frac{2^{k+1}\mu(B)}{(2^k \mu(B))^{\nu}}} \lesssim \mu(B)^{1 -\nu}.\\
\end{align*}
Finally, we obtain as desired
\[ \int_{X \backslash \alpha B}{ |T(f_B) |^{\nu} d\mu} \lesssim \mu(B),  \]
and Lemma \ref{lemmesio} follows.
\end{proof}


\medskip

\begin{proof}[Proof of Theorem \ref{ARjut}]
The proof uses the same idea as in the previous argument and can be skipped in a first reading of the paper. We define new systems from the given ones and check the hypotheses of Theorem \ref{AR}. These verifications can be long and technical to write but not difficult. They only use basic Calder\'on-Zygmund estimates for singular integrals \eqref{standard estimate dyadic} and Hardy inequalities on dyadic cubes \eqref{Hardy1}. Details are as follows.\\ 

Let $0<C<+\infty$, and suppose that there exists a dyadic pseudo-accretive system $(\{b^1_Q\},\{b^2_Q\})$ on $X$ satisfying \eqref{accretive1jut}, \eqref{accretive2jut}, \eqref{accretive3jut}, \eqref{WBP2jut} and \eqref{WBP1jut} for some $1 < \sigma < +\infty$. For $i=1,2$, and $Q$ a dyadic cube, set $\beta^i_Q = \lambda_Q b^i_{Q^k}$,  where $Q^k$ is a dyadic subcube of $Q$ such that $l(Q^k) = \delta^k l(Q)$, $\rho(Q^k,Q^c) \geq (1+C_1) l(Q^k)$ where $C_1$ is the constant intervening in property $(4)$ of Lemma \ref{cubes} (that is, $Q^k$ is far enough from the border of $Q$), and $\lambda_Q$ is such that $[\beta^i_Q]_Q = 1$ for every cube $Q$. That $Q^k$ can be so chosen follows from the small boundary condition $(6)$ of Lemma \ref{cubes} because such cubes occupy a set of measure larger than $\frac{1}{2} \mu(Q)$ if $k$ is taken large enough (independently of $Q$). We prove that $\{\beta^1_Q \}_Q, \{\beta^2_Q \}_Q$ form a $(p,q)$ dyadic pseudo-accretive system adapted to $T$. Remark first that because $Q^k$ is taken sufficiently far from the boundary of $Q$, the $\beta^i_Q$ are always supported inside $Q$, so that \eqref{accretive1} is satisfied. Next, observe that the $(\lambda_Q)_Q$ form a set of uniformly bounded coefficients : as a consequence of \eqref{accretive1jut}, we have for every dyadic cube $Q$, 
\[ \lambda_Q = \frac{ \mu(Q)}{\mu(Q^k)} \lesssim 1, \] 
by the doubling property. We have by \eqref{accretive2jut}
\[ \int_Q{(|{\beta}^1_Q|^p + |{\beta}^2_Q|^q)d\mu} \lesssim \int_{\widehat{Q^k}}{(|b^1_{Q^k}|^p + |b^2_{Q^k}|^q)d\mu} \lesssim \mu(Q^k) \lesssim \mu(Q),  \]
so that \eqref{accretive2} is satisfied. Similarly, for \eqref{accretive3}, we have by \eqref{accretive3jut} 
\[ \int_{\widehat{Q}}{(|T(\beta^1_Q)|^{q'} +|T^{\ast}(\beta^2_Q)|^{p'})d\mu} \lesssim \int_X{(|T(b^1_{Q^k})|^{q'}+|T^{\ast}(b^2_{Q^k})|^{p'})d\mu} \lesssim \mu(Q^k) \lesssim \mu(Q). \]
Let us now prove \eqref{WBP2}. Let $Q_1, Q_2$ be two dyadic cubes and let $R \subset Q_1 \cap Q_2$ be such that $[|\beta^1_{Q_1}|^p]_{R} \leq C$, $[|\beta^2_{Q_2}|^q ]_{R} \leq C$. First, if $R \subset \widehat{Q_1^k} \cap \widehat{Q_2^k}$, \eqref{WBP2} is a direct consequence of \eqref{WBP2jut}. \\
\noindent Suppose now that for example $R \subset \widehat{Q^k_2}$, $R \not\subset \widehat{Q^k_1}$. Write
\[ \langle \beta^2_{Q_2} 1_R, T(\beta^1_{Q_1} 1_R) \rangle = \langle b^2_{Q_2^k} 1_R, T(b^1_{Q_1^k} 1_{R \cap \widehat{Q_1^k}}) \rangle, \]
because $b^1_{Q_1^k}$ is supported on $\widehat{Q^k_1}$. If $R \cap \widehat{Q^k_1} \neq \varnothing$, then we must have $l(R) > l(Q^k_1)$. 
Remark that since $l(R) \leq l(Q_1)$ this implies that $R$ and $Q^k_1$ have comparable measure. Write
\[ R \cap \widehat{Q^k_1} = \bigcup_j{Q^k_{1,j}}, \]
where the $Q^k_{1,j}$ are dyadic neighbors of $Q^k_1$, and each one of them is strictly contained inside $R$. Remark that there is a finite number (uniformly bounded) of such cubes, and write
\begin{equation}  \label{proofjut}
 \langle b^2_{Q_2^k} 1_R, T(b^1_{Q_1^k} 1_{Q_{1,j}^k}) \rangle   =   \langle b^1_{Q_1^k} 1_{Q^k_{1,j}}, T^{\ast}(b^2_{Q_2^k} 1_R)  \rangle
= \Sigma_1 + \Sigma_2 + \Sigma_3.
\end{equation} 
where we have applied the decomposition $1_R = 1_{Q^k_{1,j}} + 1_{R \cap (\widehat{Q^k_{1,j}}\backslash Q^k_{1,j})} + 1_{R \backslash \widehat{Q^k_{1,j}}}$.
Applying \eqref{WBP2jut} on $Q^k_{1,j} \subset \widehat{Q^k_2} \cap \widehat{Q^k_1}$, one gets $|\Sigma_1| \lesssim \mu(Q^k_{1,j})$. Decompose further $\Sigma_2$ in a sum on dyadic neighbors of $Q^k_{1,j}$ (they have their measure comparable to that of $Q^k_{1,j}$), and apply \eqref{WBP1jut} in the particular case when there is only one cube $R_n$ to each one of those terms to get $|\Sigma_2| \lesssim \mu(Q^k_{1,j})^{\frac{1}{\sigma}} \mu(Q^k_{1,j})^{\frac{1}{\sigma'}}\lesssim \mu(Q^k_{1,j}).$ To bound the last term, it remains to apply the standard \mbox{estimates :} write 
\[ b^1_{Q^k_1} 1_{Q^k_{1,j}} = \left ( b^1_{Q^k_1} 1_{Q^k_{1,j}} - [b^1_{Q^k_1}]_{Q^k_{1,j}} 1_{Q^k_{1,j}} \right ) + [b^1_{Q^k_1}]_{Q^k_{1,j}} 1_{Q^k_{1,j}} = f + g,  \]
where both functions are supported on $Q^k_{1,j}$ and $f$ has mean zero. Remark that by \eqref{accretive2jut}
\[ | [b^1_{Q^k_1}]_{Q^k_{1,j}} | \lesssim \frac{1}{\mu(Q^k_1)} \int_{\widehat{Q^k_1}}{|b^1_{Q^k_1}|d\mu} \lesssim \frac{\mu(\widehat{Q^k_1})^{\frac{1}{p'}}}{\mu(Q^k_1)} \left (\int_{\widehat{Q^k_1}}{|b^1_{Q^k_1}|^p d\mu} \right )^{\frac{1}{p}} \lesssim C^{\frac{1}{p}}. \]
Thus $\int_{Q^k_{1,j}}{|f|d\mu} \lesssim \mu(Q^k_{1,j})$. Remark also that $| [b^2_{Q^k_2}]_R|  \lesssim C^{\frac{1}{q}}$ because of the size estimate of $\beta^2_{Q_2}$ on $R$. For $x \in Q^k_{1,j}$ and $y \in R \backslash \widehat{Q^k_{1,j}}$, we have $\lambda(x,y) \gtrsim \mu(Q^k_1) \gtrsim \mu(R)$ because as we have already remarked these two cubes have comparable measure. By \eqref{standard estimate dyadic}, we thus have
\[| \langle b^2_{Q^k_2} 1_{R \backslash \widehat{Q^k_{1,j}}}  , T f  \rangle | \lesssim \int_{Q^k_{1,j}}{|f|d\mu} \int_{R}{|b^2_{Q^k_2} 1_{R \backslash \widehat{Q^k_{1,j}}}|d\mu}  \left ( \frac{l(Q^k_1) }{\rho(Q^k_{1,j}, R \backslash \widehat{Q^k_{1,j}})}\right )^{\alpha} \frac{1}{\mu(R)} \lesssim \mu(Q^k_{1,j}). \]
For the second term, we have 
\[| \langle b^2_{Q^k_2} 1_{R \backslash \widehat{Q^k_{1,j}}}  , T g  \rangle | \lesssim | \langle b^2_{Q^k_2} 1_{R \backslash \widehat{Q^k_{1,j}}} , T(1_{Q^k_{1,j}})  \rangle | \lesssim \mu(R)^{\frac{1}{q}} \mu(Q^k_{1,j})^{\frac{1}{q'}}  \lesssim \mu(R), \]
by the Hardy inequality \eqref{Hardy1}. Thus $|\Sigma_3| \lesssim \mu(R)$. As we have a uniformly bounded number of $j$, we have shown in this case 
\[  | \langle \beta^2_{Q_2} 1_R, T(\beta^1_{Q_1} 1_R) \rangle| \lesssim \mu(R).  \]
Suppose now that $R \not\subset \widehat{Q^k_1}$, $R \not\subset \widehat{Q^k_2}$. Write again 
\[ \langle \beta^2_{Q_2} 1_R, T(\beta^1_{Q_1} 1_R) \rangle = \langle b^2_{Q_2^k} 1_{R \cap \widehat{Q_2^k}}, T(b^1_{Q_1^k} 1_{R \cap \widehat{Q_1^k}}) \rangle. \]
For this to be different from zero, we must have $l(R) > \max(l(Q^k_1),l(Q^k_2))$. Write then as before
\[ R \cap \widehat{Q^k_1} = \bigcup_j{Q^k_{1,j}}, \quad \quad R \cap \widehat{Q^k_2} = \bigcup_i{Q^k_{2,i}},   \]
with $Q^k_{1,j}$ (resp $Q^k_{2,i}$) dyadic neighbors of $Q^k_1$ (resp $Q^k_2$), and decompose
\[ \langle b^2_{Q_2^k} 1_{R \cap \widehat{Q_2^k}}, T(b^1_{Q_1^k} 1_{R \cap \widehat{Q_1^k}}) \rangle  = \sum_{i,j}  \langle b^2_{Q_2^k} 1_{Q^k_{2,i}}, T(b^1_{Q_1^k} 1_{Q_{1,j}^k}) \rangle. \]
Fix $i$ and $j$. If $Q^k_{1,j} \cap Q^k_{2,i} \neq \varnothing$, then we estimate the above term by decomposing it as in \eqref{proofjut}. If $Q^k_{1,j} \cap Q^k_{2,i} = \varnothing$ and $\rho(Q^k_{1,j}, Q^k_{2,i}) < \min(l(Q^k_{1}),l(Q^k_2))$, then apply \eqref{WBP1jut} in the particular case when there is only one cube $R_n$ to get the expected bound. Finally, if $Q^k_{1,j} \cap Q^k_{2,i} = \varnothing$ and $\rho(Q^k_{1,j}, Q^k_{2,i}) \geq  \min(l(Q^k_{1}),l(Q^k_2))$, an standard computation using \eqref{standard estimate dyadic} also gives the expected bound, that is
\[  | \langle \beta^2_{Q_2} 1_R, T(\beta^1_{Q_1} 1_R) \rangle| \lesssim \mu(R),  \]
and \eqref{WBP2} is proved.

\medskip

\noindent We now prove \eqref{WBP1}. Let $Q_1, Q_2$ be two dyadic cubes, $N$ a set of integers, and $R' \subset Q_1$, for every $n \in N$, $R_n \subset (\widehat{R'}\backslash R') \cap Q_2$, $R_n$ mutually disjoint, with $\rho(R',R_n) < l(R_n)$, $[|\beta^1_{Q_1}|^p]_{R'} \leq C$, $[|\beta^2_{Q_2}|^q ]_{R_n} \leq C$, and let $(\alpha_{n})_{n \in N}$ be a set of coefficients. We distinguish two cases. Suppose first that $R' \subset \widehat{Q^k_1}$. If for all $n \in N$, $R_n \subset \widehat{Q^k_2}$, then \eqref{WBP1} is a direct consequence of \eqref{WBP1jut}. Else, for the dyadic cubes $R_n$ which are not contained inside $\widehat{Q^k_2}$ but intersect $\widehat{Q^k_2}$ (and thus satisfy $l(R_n) > l(Q^k_2)$), write as before $R_n \cap \widehat{Q^k_2} = \cup_i{R_{n,i}}$, where the $R_{n,i}$ are dyadic neighbors of $Q^k_2$. Denote by $N_0$ the set of integers $n$ such that $R_n$ satisfy the previous property. Remark that $N_0$ is a finite set, and that its cardinal is uniformly bounded : indeed, for $n \in N_0$, the dyadic cubes $R_n$ satisfy $l(Q^k_2) < l(R_n) \leq l(Q_2) $ and they are disjoint with respect to one another. Remark also that we have for every $n \in N_0$ and every $i$, $[|b^2_{Q_2}|^q ]_{R_{n,i}} \lesssim C$. As a matter of fact, since $R_{n,i}$ is a dyadic neighbor of $Q^k_2$, its measure is comparable to the measure of $Q^k_2$, so that since $R_n \subset Q_2$, we have $\frac{\mu(R_n)}{\mu(R_{n,i})} \lesssim \frac{\mu(Q_2)}{\mu(Q^k_2)} \lesssim 1$ and
\[[|b^2_{Q_2}|^q ]_{R_{n,i}} \leq [|b^2_{Q_2}|^q ]_{R_n} \times \frac{\mu(R_n)}{\mu(R_{n,i})} \lesssim C. \]
For every $n \in N_0$, and every $i$ such that $\rho(R_{n,i}, R') \geq l(R_{n,i})$, apply as before the standard Calder\'on-Zygmund estimates \eqref{standard estimate dyadic} and the Hardy inequality \eqref{Hardy1} to get
\begin{align*} 
| \langle b^2_{Q^k_2} 1_{R_{n,i}} , T(b^1_{Q^k_1} 1_{R'})  \rangle | & \leq  |\langle T^{\ast} (b^2_{Q^k_2} 1_{R_{n,i}}\!\! - \!\! [b^2_{Q^k_2}]_{R_{n,i}} 1_{R_{n,i}}  ) , b^1_{Q^k_1} 1_{R'} \rangle | +  |[b^2_{Q^k_2}]_{R_{n,i}}|\!\! | \langle  T^{\ast}( 1_{R_{n,i}}) , b^1_{Q^k_1} 1_{R'} \rangle|  \\
& \lesssim \mu(R_{n,i})l(R_{n,i})^{\alpha} \int_{R'}{\frac{|b^1_{Q^k_1}(y) |}{\rho(y,z_{n,i})^{\alpha}}  \frac{d\mu(y)}{\mu(B(z_{n,i},\rho(y,z_{n,i}) ))}}    +  \mu(R_{n,i})^{\frac{1}{p'}} \mu(R')^{\frac{1}{p}}, \\
\end{align*}
where $z_{n,i}$ is the center of the dyadic cube $R_{n,i}$.
To estimate this integral, split it onto coronae for which $2^l l(R_{n,i}) \leq \rho(y, z_{n,i}) < 2^{l+1}l(R_{n,i})$ : let $B_l =B(z_{n,i}, 2^{l+1} l(R_{n,i}))$ and applying Hölder's inequality, write
\begin{align*}
\int_{R'}{\frac{|b^1_{Q^k_1}(y) |}{\rho(z_{n,i},y)^{\alpha}}  \frac{d\mu(y)}{\mu(B(y,\rho(z_{n,i},y) ))}} & \lesssim \sum_{l \geq l_0}\frac{1}{2^{l\alpha} l(R_{n,i})^{\alpha}} \frac{1}{\mu(B_l)} \int_{R' \cap B_l}{|b^1_{Q^k_1}|d\mu}\\
& \lesssim \frac{1}{l(R_{n,i})^{\alpha}} \sum_{l\geq l_0} \frac{1}{2^{l \alpha}} \frac{\mu(B_l)^{\frac{1}{p'}}}{\mu(B_l)} \left ( \int_{R'}{|b^1_{Q^k_1}|^p d\mu} \right )^{\frac{1}{p}}\\
& \lesssim  \frac{\mu(R')^{\frac{1}{p}}}{l(R_{n,i})^{\alpha} \mu(R_{n,i})^{\frac{1}{p}}} \sum_{l\geq l_0} \frac{1}{2^{l \alpha}},
\end{align*}
where $l_0$ is a fixed integer which only depends on the dimensional constant $a_0$ of Lemma \ref{cubes}. We have applied the Hölder inequality to get the second line, and the fact that for every $l \geq l_0$, $\mu(B_l) \gtrsim \mu(R_{n,i})$ to get the last line. Consequently, since $\mu(R_{n,i}) \leq \mu(R_n)$, we obtain
\[ | \langle b^2_{Q^k_2} 1_{R_{n,i}} , T(b^1_{Q^k_1} 1_{R'})  \rangle | \lesssim \mu(R_{n})^{\frac{1}{p'}} \mu(R')^{\frac{1}{p}}.\]
For every $n \in N_0$ and every $i$ such that $\rho(R_{n,i}, R') < l(R_{n,i})$, apply \eqref{WBP1jut} in the particular case when there is only one cube involved to get
\[ | \langle b^2_{Q^k_2} 1_{R_{n,i}} , T(b^1_{Q^k_1} 1_{R'})  \rangle |  \lesssim \mu(R_{n,i})^{\frac{1}{\sigma}} \mu(R')^{\frac{1}{\sigma'}} \lesssim \mu(R_{n})^{\frac{1}{\sigma}} \mu(R')^{\frac{1}{\sigma'}}. \]
Since $\mu (R_n) \lesssim \mu(R')$ and since there is a finite uniformly bounded number of $i$, we have
\[ | \langle b^2_{Q^k_2} 1_{R_{n}} , T(b^1_{Q^k_1} 1_{R'})  \rangle |  \lesssim \mu(R_{n})^{\frac{1}{\nu}} \mu(R')^{\frac{1}{\nu'}}, \]
with $\nu = \max(\sigma,p')$. Now, applying \eqref{WBP1jut} to the terms for which $n \in N\backslash N_0$, and the fact that the cardinal of $N_0$ is uniformly bounded, write 
\begin{align*}
\left | \left \langle \beta^2_{Q_2} \left ( \sum_{n\in N}{\alpha_n 1_{R_n}} \right ), T(\beta^1_{Q_1} 1_{R'})  \right \rangle \right | \! & = \left | \left \langle b^2_{Q^k_2} \left ( \sum_{n\in N}{\alpha_n 1_{R_n}} \right ), T(b^1_{Q^k_1} 1_{R'})  \right \rangle \right | \!   \\
& \lesssim \! \left ( \sum_{n \in N\backslash N_0}{\!\! |\alpha_n|^{\sigma} \mu(R_n)}  \right  )^{\! \frac{1}{\sigma}}\!\!\!\! \mu(R')^{\frac{1}{\sigma'}} \! + \! \sum_{n \in N_0}{|\alpha_n| \mu(R_{n})^{\frac{1}{\nu}} \mu(R')^{\frac{1}{\nu'}} }\\
& \lesssim \! \left ( \sum_{n \in N\backslash N_0}{\! \!\!\!\! |\alpha_n|^{\sigma} \mu(R_n)} \!  \right  )^{\! \frac{1}{\sigma}}\!\!\!\! \mu(R')^{\frac{1}{\sigma'}} \! +  \!\! \left( \sum_{n \in N_0} {\!|\alpha_n|^{\nu} \mu(R_n)}  \right  )^{\! \frac{1}{\nu}}\!\!\!\! \mu(R')^{\frac{1}{\nu'}}\\
\end{align*}
Remark that using Hölder's inequality with $\sum{\mu(R_n) \lesssim \mu(R')}$ we can raise $\sigma$ to $\nu$ if $\sigma < p'$. Thus, we have
\[ \left | \left \langle \beta^2_{Q_2} \left ( \sum_{n\in N}{\alpha_n 1_{R_n}} \right ), T(\beta^1_{Q_1} 1_{R'})  \right \rangle \right | \! \lesssim \left \| \sum_{n \in N}{\alpha_n 1_{R_n}}  \right  \|_{\nu} \mu(R')^{\frac{1}{\nu'}}.  \]
\noindent It remains only to see what happens when $R'$ is not contained inside $\widehat{Q^k_1}$ and is not disjoint with $\widehat{Q^k_1}$. In this case, write again $R' \cap \widehat{Q^k_1} = \cup_{j\in J}{R'_j}$, where the $R'_j$ are dyadic neighbors of $Q^k_1$. Distinguish the $j$ for which $\rho(R'_j, {R'}^c) \geq l(R'_j )$ (let $J_0$ be the set composed of such $j$) and those for which $\rho(R'_j, {R'}^c) < l(R'_j )$ (remark that there can very well be no such $j$). For the latter, apply the same argument as above : the only difference is that there can be a finite (uniformly bounded) number of $n$ for which $l(R_n) > l(R'_j)$ (and then since $l(R_n)\leq l(R')$ they have comparable measure) but one only has to apply either the dual estimate of \eqref{WBP1jut} if $\rho(R_n,R'_j) < l(R'_j)$, or the usual decomposition and \eqref{standard estimate dyadic} if $\rho(R_n,R'_j) \geq l(R'_j)$, to get the expected bound for each one of those terms. We obtain for every $j \in J \backslash J_0$
\[  \left | \left \langle b^2_{Q^k_2} \left ( \sum_{n}{\alpha_n 1_{R_n}} \right ), T(b^1_{Q'} 1_{R'_j})  \right \rangle \right | \! \lesssim \left \| \sum_{n}{\alpha_n 1_{R_n}}  \right  \|_{\nu_1} \mu(R')^{\frac{1}{\nu'_1}},    \]
with $\nu_1 = \max(\sigma,p')$. For $j \in J_0$, apply once again the same kind of decomposition as \mbox{before :} setting $f = \sum_{n}{\alpha_n 1_{R_n}}$, we have
\begin{align*}
\left \langle b^2_{Q^k_2}f , T(b^1_{Q^k_1} 1_{R'_j})  \right \rangle  & = \left \langle b^2_{Q^k_2} f, T(b^1_{Q^k_1} 1_{R'_j} - [b^1_{Q^k_1}]_{R'_j} 1_{R'_j})  \right \rangle +  [b^1_{Q^k_1}]_{R'_j}  \left \langle b^2_{Q^k_2} f, T( 1_{R'_j})  \right \rangle\\
& = \Sigma_1 + \Sigma_2\\
\end{align*}
Remark that $|[b^1_{Q^k_1}]_{R'_j}| \leq \frac{\mu(R')}{\mu(R'_j)} |[b^1_{Q^k_1}]_{R'}| \lesssim \frac{\mu(Q_1)}{\mu(Q^k_1)} \lesssim1$. Apply \eqref{standard estimate dyadic} and the disjointness of the cubes $R_n$ to get 
\begin{align*}
|\Sigma_1| & \lesssim \int_{\cup{R_n}}{|b^2_{Q^k_2} f|d\mu} \int_{R'_j}{|b^1_{Q^k_1} 1_{R'_j} - [b^1_{Q^k_1}]_{R'_j} 1_{R'_j}|d\mu} \frac{1}{\mu(R'_j)}\\
& \lesssim \| f \|_{q'} \left ( \int_{\cup{R_n}}{|b^2_{Q^k_2}|^q d\mu} \right )^{\frac{1}{q}} \lesssim  \| f \|_{q'} \ \ \mu(R')^{\frac{1}{q}}.\\
\end{align*}
Apply the Hardy inequality \eqref{Hardy1} to $\Sigma_2$ for some $1 < r < q$ to get
\begin{align*}
|\Sigma_2| & \lesssim \left( \int_{\cup{R_n}}{|f b^2_{Q^k_2}|^r d\mu} \right )^{\frac{1}{r}} \mu(R'_j)^{\frac{1}{r'}} \lesssim \left( \int_{\cup{R_n}}{|f |^{r\theta'} d\mu} \right )^{\frac{1}{r \theta'}} \left( \int_{\cup{R_n}}{|b^2_{Q^k_2} |^{q} d\mu} \right )^{\frac{1}{r\theta}} \mu(R')^{\frac{1}{r'}},
\end{align*}
where the last inequality is obtained by applying Hölder's inequality with the exponent $\theta = \frac{q}{r}$. Since $\sum{\mu(R_n)} \lesssim \mu(R')$, we have $\int_{\cup{R_n}}{|b^2_{Q^k_2} |^{q} d\mu} \lesssim C \mu(R')$, and thus
\[ |\Sigma_2| \lesssim \| f \|_{\nu_2} \mu(R')^{\frac{1}{\nu'_2}},   \] 
with $\nu_2 = r \theta' = \frac{rq}{q-r}.$ Finally, as there is a finite uniformly bounded number of $j$ in $J$, we obtain as expected the bound
\[ \left | \left \langle \beta^2_{Q_2} \left ( \sum_{n}{\alpha_n 1_{R_n}} \right ), T(\beta^1_{Q_1} 1_{R'})  \right \rangle \right | \! \lesssim \left \| \sum_{n }{\alpha_n 1_{R_n}}  \right  \|_{\nu} \mu(R')^{\frac{1}{\nu'}},  \]
with $\nu = \max(\nu_1, \nu_2, q')$, and thus \eqref{WBP1} is proved. The dual estimates obviously follow by symmetry of our hypotheses. It remains only to apply Theorem \ref{AR} to get the boundedness of $T$.
\end{proof}
\medskip

\section{Preliminaries to the proof of Theorem \ref{AR}}

\medskip

We shall use the language of dyadic cubes rather than tiles used in \cite{AHMTT}.

\subsection{Adapted martingale difference operators}Let $Q$ be a dyadic cube of the space of homogeneous type $X$. For every $f \in L^2(X)$ we define
\[ E_Q(f) = [f]_Q  1_Q \quad \mathrm{and} \quad \Delta_Q (f) = \sum_{Q' \in \widetilde{Q}}{E_{Q'}(f) - E_Q(f)}. \] 
If $Q$ has only one child, then $\Delta_Q = 0$. It is easy to check that if $\mu(X) = +\infty$, for any $f \in L^2(X)$ we have the representation formula
\[f = \sum_{Q \subset X}{\Delta_Q f},\]
with $\|f\|^2_2 = \sum_{Q \subset X}{\|\Delta_Q f\|^2_2}.$ If $\mu(X) < +\infty$, replace $f$ by $f - [f]_X$ on the left hand sides. We omit the details and refer to \cite{AHMTT} for further explanation.

Let $\mathbf{P}$ be a collection of cubes in $X$. We will say that a locally integrable function $b$ is pseudo-accretive on $\mathbf{P}$ if $| [b]_Q | \gtrsim 1$ for all $Q \in \mathbf{P}$, and is strongly pseudo-accretive on $\mathbf{P}$ if we have $|[b]_Q| \gtrsim 1$ and $ | [b]_{Q'} | \gtrsim 1$ for all $Q \in \mathbf{P}$, $Q' \in \widetilde{Q}$. Conversely, we say that such cubes $Q$ are respectively pseudo-accretive (\emph{pa}), strongly pseudo-accretive (\emph{spa}) with respect to $b$. A cube that is pa but not spa with respect to $b$ is called degenerate pseudo-accretive (\emph{dpa}) with respect to $b$. Given a function $b$ strongly pseudo-accretive on $Q$ (that is, on $\mathbf{P} = \{Q\}$), we define
\[ E^b_Q (f) = \frac{[f]_Q}{[b]_Q}  1_Q, \]
\[ \Delta_Q^b (f) = \sum_{Q' \in \widetilde{Q}}{E^b_{Q'}(f)} - E^b_Q (f). \] 
Again, if $Q$ has only one child then $\Delta_Q^b = 0$. A straightforward computation shows that 
\[ \int_Q{b \Delta_Q^b f d\mu} = 0. \]
Note that the operators $\Delta_Q^b$ introduced here are not the same ones than those defined by S. Hofmann in \cite{Hofmann}. However, if $D_Q^b$ is the operator defined in \cite{Hofmann}, we have $D_Q^b = b \Delta_Q^b$. It will be easier to work with the operators $\Delta_Q^b$ introduced here because we do not assume $b \in L^2(Q)$. The following proposition gives a wavelet representation of theses operators.

\begin{prop}\label{representation} 
Let $Q$ be a dyadic cube and $N_Q$ the number of children of $Q$. Assume $N_Q \geq 2$. If $b \in L^1(Q)$ is strongly pseudo-accretive on $Q$, then there exist $0<C<\infty$ and functions $\phi_Q^{b,s}, \widetilde{\phi}_Q^{b,s}$, $1\leq s \leq N_Q -1$ with the following properties
\begin{enumerate}
\item The functions $\phi_Q^{b,s}, \widetilde{\phi}_Q^{b,s}$ are supported on $Q$ and constant on each child of $Q$.
\medskip
\item $\int_Q {b\phi_Q^{b,s} d \mu} = \int_Q{\widetilde{\phi}_Q^{b,s} b d\mu} = 0.$
\medskip
\item $\|\phi_Q^{b,s} \|_2 + \| \widetilde{\phi}_Q^{b,s}\|_2 \leq C.$
\medskip
\item $\int_Q {\widetilde{\phi}_Q^{b,s}b \phi_Q^{b,s'} d\mu} = \delta_{s,s'}.$
\medskip
\item For all $f\in L^1(X)$ \[ \Delta_Q^b f = \sum_{s=1}^{N_Q -1}{\langle f, \phi_Q^{b,s}\rangle \widetilde{\phi}_Q^{b,s}  }. \] 

\item For all $f\in L^1(X)$ \[ \frac{1}{C} \sum_{s=1}^{N_Q -1}{|\langle f, \phi_Q^{b,s} \rangle |^2} \leq \|\Delta_Q^b f \|^2_{L^2(Q)} \leq C \sum_{s=1}^{N_Q -1}{|\langle f, \phi_Q^{b,s} \rangle |^2}.   \]
\end{enumerate}
Moreover, $C$ depends only on $C(b,Q)=\sup_{Q' \in \widetilde{Q}}([ |b|]_Q, |[b]_Q^{-1}|, |[b]_{Q'}^{-1} | ).$
\end{prop}

The proof is postponed to appendix $A$. Let us comment. If $N_Q = 1$, then there is nothing to do as $\Delta_Q^b f = 0$ anyway. On the other hand, $N_Q$ is uniformly bounded as mentioned earlier. Recall here that $\langle f,g \rangle = \int_X {fg d\mu}.$ Next, these functions are adapted Haar functions as in \cite{AHMTT} with a biorthogonality property. Note that $(1)$ and $(2)$ imply that $\|\phi_Q^{b,s} \|_1 + \| \widetilde{\phi}_Q^{b,s}\|_1 \leq C \mu(Q)^{\frac{1}{2}}$ and $\|\phi_Q^{b,s} \|_{\infty} + \| \widetilde{\phi}_Q^{b,s}\|_{\infty} \leq C_X^{\frac{1}{2}} C \mu(Q)^{- \frac{1}{2}}$, with 
\begin{equation}\label{CX}
C_X = \sup_{Q \subset X, Q' \in \widetilde{Q}}{\frac{\mu(Q)}{\mu(Q')}}.
\end{equation} 
As long as $C(b,Q)$ remain uniform as $Q$ and $b$ vary, we obtain uniform estimates. As this shall be the case throughout, and as $s$ will play no role in our analysis, we shall make the standing assumption that, unless it is $0$, $\Delta_Q^b f$ rewrites as 
\[ \Delta_Q^b f = \langle f, \phi_Q^b \rangle \phi_Q^b. \]
Given a dyadic cube $Q \subset X$, set
\[ R_Q = \{  Q' \, \mathrm{dyadic} \, ; \, Q' \subset Q    \}. \]
Now, if $\{P_i\}$ is a collection of non overlapping dyadic subcubes of a cube $Q_0$, set 
\[ \Omega(Q_0,\{P_i\}) = R_{Q_0} \setminus \bigcup_i{R_{P_i}},\]
and assume that  $\Omega(Q_0,\{P_i\}) \subset \{ Q \subset Q_0  \, : \, Q \, \mathrm{spa} \, b \}.$ Then we have the following results.
\begin{lem}\label{Mart}
Let $1<q < +\infty$, let $C_{\{P_i\}}(b) = \sup_i{[|b|^q]^{1/q}_{P_i}} $. We have, with $\Omega:= \Omega(Q_0,\{P_i\}),$
\[ \sum_{Q \in \Omega}{| \langle f, \phi_{Q}^b \rangle |^2} \lesssim   C_{\{P_i\}}(b)^2 \, \|f\|^2_{L^2(Q_0)}.\]
\end{lem}

\begin{proof}
We refer to Lemma $6.7$ of \cite{AHMTT} for the detail. The only difference is that we use here the language of cubes rather than tiles.
\end{proof}

\begin{lem} \label{Mart2}
Suppose that $C_{\{P_i\}}(b) < +\infty$, $b 1_{Q_0 \backslash \cup {P_i}} \in L^{\infty}(Q_0)$, and $b \in L^q(Q_0)$ for some $1<q < +\infty$. Let $c_Q$ be a set of uniformly bounded coefficients, and define, with $\Omega:= \Omega(Q_0,\{P_i\}),$ the bounded bilinear form
\[ \langle Lf , g \rangle = \sum_{Q \in \Omega}{c_Q \langle f, \phi_{Q}^b \rangle \langle  \phi^b_Q  , g  \rangle}, \quad \mathrm{for} \, f,g \in L^2(Q_0).  \]
Then $L$ is bounded on $L^{\nu}(Q_0)$ for all $1<\nu< +\infty$, with
\[ \|Lf\|_{L^{\nu}(Q_0)} \lesssim  \| \{c_Q\} \|_{l^{\infty}} \left (C_{\{P_i\}}(b) +1 \right ) \| f \|_{L^{\nu}(Q_0)}. \]
Furthermore, $\sum_{Q \in \Omega}{\langle f, \phi_{Q}^b \rangle b \phi_{Q}^b}$ unconditionally converges in $L^{\nu}(Q_0)$ if $1<\nu \leq q$ and $f\in L^{\nu}(Q_0)$, with
\[ \left \| \sum_{Q \in \Omega}{c_{Q}\langle f, \phi_{Q}^b \rangle b \phi_{Q}^b } \right \|_{L^{\nu}(Q_0)} \leq \| L \|_{L^{\nu}\rightarrow L^{\nu}} \left (C_{\{P_i\}}(b) + \|b 1_{Q_0 \backslash \cup{P_i}}\|_{\infty}\right ) \|f\|_{L^{\nu}(Q_0)}  . \]
The implicit constants in the $\lesssim$ depend only on the doubling constant $C_D$ and $\nu$.
\end{lem}

\begin{proof}
Let $F = Q_0 \backslash \cup {P_i}$. First, $L$ is bounded on $L^2(Q_0)$ by Lemma \ref{Mart} and the Cauchy-Schwarz inequality. We now prove that $L$ is of weak type $(1,1)$. Indeed, let $f \in L^1(Q_0)\cap L^2(Q_0)$, $\lambda >0$, and write a classical Calder\'on-Zygmund decomposition of \mbox{$f$:} there exist functions $h$ and $(\beta_i)_i$ supported in mutually disjoint dyadic cubes $Q_i$ such that $f=h+\sum_i{\beta_i}$, with $h \in L^1(Q_0) \cap L^{\infty}(Q_0)$, $\|h\|_{L^1(Q_0)} \lesssim \| f\|_{L^1(Q_0)}$, $\|h\|_{L^{\infty}(Q_0)} \lesssim \lambda$, and for all $i$, $\|\beta_i\|_{L^1(Q_0)\cap L^2(Q_0)} < +\infty$, $\int_{Q_i}{\beta_i d\mu} = 0.$\\ 
Observe that $L\beta = \sum_i \sum_{Q \in \Omega}{c_Q \langle \beta_i, \phi_Q^b \rangle \phi^b_Q  }$, so that if $Q$ strictly contains $Q_i$ then $\langle \beta_i, \phi_Q^b \rangle =0$ since $ \int_{Q_i}{\beta_i}=0$ and $\phi_Q^b$ constant on the children of $Q$. It is thus clear that $L\beta$ is supported in $\cup_i{\overline{Q_i}}$, where $\overline{Q_i}$ denotes the closure of $Q_i$. Hence $Lf=Lh$ on $Q_0 \backslash \cup_i{\overline{Q_i}}$. Now, write
\[ \mu(\{ x \in Q_0: |Lf(x)|> \lambda \}) \leq \mu(\{ x\in Q_0 \backslash \cup_i{\overline{Q_i}}: |Lh(x)| > \lambda \}) + \sum_i{\mu(\overline{Q_i})}, \]
and note that $\mu(\overline{Q_i}) = \mu(Q_i)$ by $(6)$ of Lemma \ref{cubes}. Since $h \in L^1(Q_0) \cap L^{\infty}(Q_0)$, $h\in L^2(Q_0)$ with $\|h\|^2_{L^2(Q_0)} \lesssim \lambda \|f\|_{L^1(Q_0)}$, so that, using the boundedness of $L$ on $L^2(Q_0)$,
\[ \mu(\{ x\in Q_0 \backslash \cup_i{\overline{Q_i}}: |Lh(x)| >\lambda \}) \leq \frac{1}{\lambda^2}\int_{Q_0}{|Lh|^2d\mu} \lesssim \frac{\|L\|^2_{L^2 \rightarrow L^2}}{\lambda}\|f\|_{L^1(Q_0)}.  \]
Furthermore, we have by the maximal theorem for dyadic averages
\[   \sum_i{\mu(Q_i)}\leq \frac{1}{\lambda}\|f\|_{L^1(Q_0)}. \]
This proves that $L$ is of weak type $(1,1)$. By duality and interpolation, $L$ is thus bounded on $L^{\nu}(Q_0)$ for all $1<\nu< +\infty$, with the required control on its norm. \\
The second part of Lemma \ref{Mart2} amounts to proving that for all $1<\nu \leq q$, and $f \in L^{\nu}(Q_0), g \in L^{\nu'}(Q_0),$
\[ |\langle Lf , bg \rangle| \leq  \| L \|_{L^{\nu} \rightarrow L^{\nu}} \left (C_{\{P_i\}}(b) + \|b 1_{Q_0 \backslash \cup{P_i}}\|_{\infty}\right )  \|f\|_{L^{\nu}(Q_0)}\|g\|_{L^{\nu'}(Q_0)}.\]
Write
\[ \langle Lf, b g \rangle = \sum_{Q\in \Omega}{c_Q  \langle f, \phi_{Q}^b \rangle \langle  \phi^b_Q , b g 1_F  \rangle} + \sum_{Q \in \Omega}{c_Q  \langle f, \phi_{Q}^b \rangle \langle  \phi^b_Q , b g 1_{\bigcup{P_i}}  \rangle} = \, (\mathrm{i})  +  (\mathrm{ii}).  \]
First, as a consequence of what precedes, we have
\[ |(\mathrm{i}) | \leq  \| L \|_{L^{\nu} \rightarrow L^{\nu}} \, \| f \|_{L^{\nu}(Q_0)} \|b 1_F  g\|_{L^{\nu'}(Q_0)} \leq \| L \|_{L^{\nu} \rightarrow L^{\nu}} \, \|b 1_F\|_{\infty}  \|f\|_{L^{\nu}(Q_0)}\|g\|_{L^{\nu'}(Q_0)}. \]
For the second term, write
\[ (\mathrm{ii}) = \sum_{Q \in \Omega}\sum_{i}{c_Q  \langle f, \phi_{Q}^b \rangle \langle  \phi^b_Q , 1_{P_i} \rangle [b g]_{P_i}} + \sum_{Q \in \Omega}\sum_{i}{c_Q \langle f, \phi_{Q}^b \rangle \langle  \phi^b_Q , (b g - [b g]_{P_i}) 1_{P_i}  \rangle}.  \]
Observe that for any $Q \in \Omega$, either $Q$ and $P_i$ are disjoint, or $\phi^b_Q $ is constant over $P_i$, which yields in any case
\[  \langle  \phi^b_Q  , (b g - [b g]_{P_i}) 1_{P_i}  \rangle = [\phi^b_Q ]_{P_i} \int_{P_i}{(b g - [b g]_{P_i}) d\mu} = 0.  \]
Therefore,
\[ (\mathrm{ii}) = \sum_{Q \in \Omega}\sum_{i}{c_Q  \langle f, \phi_{Q}^b \rangle \langle  \phi^b_Q , 1_{P_i} \rangle [b g]_{P_i}}, \]
and
\begin{align*}
| (\mathrm{ii}) | & \leq \| L \|_{L^{\nu} \rightarrow L^{\nu}} \, \|f\|_{L^\nu(Q_0)} \, \left \|   \sum_{i}{[b g ]_{P_i} 1_{P_i}} \right  \|_{L^{\nu'}(Q_0)}  \!\!\!\!\!\!\!\leq  \| L \|_{L^{\nu} \rightarrow L^{\nu}} \, \|f\|_{L^\nu(Q_0)} \,  \left ( \sum_{i}{ \mu(P_i) |[b g ]_{P_i} |^{\nu'}} \right )^{1/{\nu'}}.\\
\end{align*}
Observe that for all $i$, applying the Hölder inequality, we have
\begin{align*}
|[bg]_{P_i}| & \leq \frac{1}{\mu(P_i)} \int_{P_i}{|bg|d\mu} \leq \left ( \frac{1}{\mu(P_i)} \int_{P_i}{|b|^{\nu}d\mu}    \right )^{1/{\nu}} \left ( \frac{1}{\mu(P_i)} \int_{P_i}{|g|^{\nu'} d\mu} \right )^{1/{\nu'}}\\
& \leq C_{\{P_i\}}(b) \left ( \frac{1}{\mu(P_i)} \int_{P_i}{|g|^{\nu'} d\mu} \right )^{1/{\nu'}},\\
\end{align*}
because $\nu \leq q$. Consequently, we get
\begin{align*}
| (\mathrm{ii}) | & \leq \| L \|_{L^{\nu} \rightarrow L^{\nu}} \,  C_{\{P_i\}}(b) \, \|f\|_{L^\nu(Q_0)} \,  \left ( \sum_{i}{ \int_{P_i}{|g|^{\nu'} d\mu}} \right )^{1/{\nu'}}\\
& \leq \| L \|_{L^{\nu} \rightarrow L^{\nu}} \,C_{\{P_i\}}(b)  \|f\|_{L^\nu(Q_0)} \, \|g\|_{L^{\nu'}(Q_0)},\\
\end{align*}
where the last inequality comes from the fact that the $P_i$ are non overlapping cubes. As this argument is valid for any set of uniformly bounded coefficients $(c_Q)_Q$, this proves the unconditional convergence of the sum $\sum_{Q \in \Omega}{\langle f, \phi_{Q}^b \rangle b \phi_{Q}^b}$ in $L^{\nu}(Q_0)$, and we are done. 
\end{proof}

\begin{remark}
The operator $L$  is a sample of a perfect dyadic singular integral operator as in \cite{AHMTT}. See Section $8.1$.
\end{remark}

\smallskip

\subsection{A stopping time decomposition}We now state a stopping time lemma, already seen in \cite{AHMTT} and \cite{Hofmann}, and which is central in our argument. 

\begin{lem} \label{Hofmann}
Let $ 1< p,q < + \infty$ with dual exponents $p', q'$. Assume that $T$ is a singular integral operator with locally bounded kernel. Suppose that there exists a $(p,q)$ dyadic pseudo-accretive system $(\{ b^1_Q \},\{b^2_Q\})$ adapted to $T$. Then there exist $0< \varepsilon  \ll 1$ and $C<+\infty$ depending only on the doubling constant $C_D$, the implicit constants in \eqref{standard1} and \eqref{standard2}, and the constant $C_A$ from \eqref{accretive2} and \eqref{accretive3bis}, such that for each fixed dyadic cube $Q_0$, $R_{Q_0}$ has a partition
\[ R_{Q_0} = \Omega^1 \cup \Omega^1_{buffer} \cup (\cup R_{P_n^1}), \]
adapted to $b=b^1_{Q_0}$ as follows
\begin{enumerate}
\item The tops $\{ P_n^1 \}$ are non overlapping dyadic subcubes of $Q_0$ with
\begin{equation} \label{packing}
\sum_n{\mu(P_n^1)} \leq (1-\varepsilon)\mu(Q_0).
\end{equation}
We say that they realize a $(1 - \varepsilon)$-packing of $Q_0$.
\smallskip
\item b is strongly pseudo-accretive on $ \Omega^1$, and degenerate pseudo-accretive on $ \Omega^1_{buffer}$.
\smallskip
\item We have for all $Q \in \Omega^1 \cup \Omega^1_{buffer} \cup \{P_n\}_n$ the mean bounds

\begin{equation} \label{Hofmann3} 
\int_{Q}{\left ( {|b|}^p + {|T(b)|}^{q'} \right ) d\mu} \leq C \mu(Q).
\end{equation}

\item $\Omega^1_{buffer} = \{ Q \in R_{Q_0} \backslash \cup R_{P_n^1} \ \ | \ \  \exists n \,:  \, P_n^1 \in \widetilde{Q}     \}$, that is $\Omega^1_{buffer}$ is composed of the dpa cubes with respect to $b$, and we have, with $C_X$ as in \eqref{CX}, the $C_X$-packing property

\begin{equation}  \label{packing2}
\sum_{Q \in \Omega^1_{buffer}}{\mu(Q)} \leq C_X \mu(Q_0).
\end{equation}

\item We have the decomposition 
\begin{equation} \label{Hofmann5}
f = [ f ]_{Q_0} b + \sum_{Q \in \Omega^1}{b \Delta^{b}_{Q}f  } +  \sum_{Q \in \Omega^1_{buffer}}{\xi^b_Q f} +  \sum_{n}{(f 1_{P_n^1} - [f]_{P_n^1} b^1_{P_n^1})},
\end{equation}

where the "buffer functions" $\xi^b_Q$ are supported on $Q$, have mean zero, and take the form

\[ \xi^b_Q f = \sum_{Q' \in \widetilde{Q} \atop Q' \in \Omega^1 \cup \Omega^1_{buffer}}{a_{Q'}b 1_{Q'}} + \sum_{P_n^1 \in \widetilde{Q} }{(a'_{P_n^1}b^1_{P_n^1} - a''_{P_n^1} b 1_{P_n^1}) }, \]

where the coefficients $a_Q, a'_Q, a''_Q$ depend on $f$ and $b$, and obey the bounds

\[ \sum_{Q' \in \widetilde{Q}}{\left ( |a_{Q'}| + |a'_{Q'}| + |a''_{Q'}| \right )} \lesssim  \|f\|_{L^{\infty}(Q_0)}. \]

\end{enumerate}
A similar statement holds with $b^1_Q$, $Tb^1_Q$ and $p$, $q'$ replaced by $b^2_Q$, $T^{\ast}b^2_Q$ and $q$, $p'$.
\end{lem}
%
\begin{proof}
Only minor modifications have to be made to adapt the proof to the homogeneous setting. The stopping time is simply on whether
\[ (i) \quad [|b|]_Q < \delta \] or \[(ii) \quad [ |b|^p ]_Q + [|T(b)|^{q'}]_Q >C\] for dyadic subcubes $Q$ of $Q_0$ for appropriately chosen $\delta>0$ and $C<+\infty$. We refer to \cite{Hofmann} for a proof that can be followed line by line (we do not need the control with the maximal function used there). The fact that \eqref{Hofmann3} holds also for $Q=P_n$ comes from the doubling property. The decomposition formula \eqref{Hofmann5} is identical to Lemma $6.11$ of \cite{AHMTT}.
\end{proof}
Observe that this lemma does not require \eqref{WBP1} nor \eqref{WBP2} to be satisfied. Remark as well that with our previous notation, we have $ \Omega^i \cup \Omega^i_{buffer} = \Omega(Q_0, \{P^i_n\}).$
We now introduce the stopping-time projections which are fundamental to our analysis. For $f \in L^{1}(Q_0)$, and with the notation of Lemma \ref{Hofmann}, for $i \in \{ 1,2 \}$, set
\begin{equation} \label{Pieq}
\Pi^{b^i_{Q_0}} f = f - \sum_n{(f 1_{P_n^i} - [f]_{P_n^i} b^i_{P_n^i})} = f 1_{Q_0 \backslash \cup P_n^i} + \sum_n{[f]_{P_n^i}b^i_{P^i_n}}.
\end{equation}
This operator not only depends on $b^i_{Q_0}$ but also on the stopping-time cubes $P^i_n$ and the corresponding functions $b^i_{P_n^i}$. Thus, the notation is not completely accurate but simplifies the understanding.
\begin{lem} \label{Projec}
The stopping-time projections $\Pi^{b^i_{Q_0}}$ have the following properties: for all $f \in L^{\infty}(Q_0)$, we have
\begin{equation} \label{Projec1}
 \int_{Q_0}{|\Pi^{b^i_{Q_0}}f |^{p_i} d\mu} \leq (1 + C_A) \| f\|^{p_i}_{L^{\infty}(Q_0)} \mu(Q_0), 
 \end{equation}
with $C_A$ the constant in \eqref{accretive2} and $(p_1,p_2) = (p,q)$.
\begin{equation} \label{Projec2}
[\Pi^{b^i_{Q_0}}f ]_Q = [f]_Q, 
\end{equation}
for all $Q \in \Omega^i \cup \Omega^i_{buffer} \cup \{P^i_n\}$ .
\begin{equation} \label{Projec3}
\Pi^{b^i_{Q_0}}(\Pi^{b^i_{Q_0}}f) = \Pi^{b^i_{Q_0}} f.
\end{equation}
\begin{equation}\label{Projec4}
(\Pi^{b^i_{Q_0}}f) 1_Q = \Pi^{b^i_{Q_0}}(f1_Q),
\end{equation}
for all $Q \in \Omega^i \cup \Omega^i_{buffer} \cup \{P^i_n\}$ .
\begin{equation} \label{Projec5}
 \int_{Q}{|\Pi^{b^i_{Q_0}}f |^{p_i} d\mu} \leq (1 + C_A) \| f\|^{p_i}_{L^{\infty}(Q)} \mu(Q), 
\end{equation}
for all $Q \in \Omega^i \cup \Omega^i_{buffer} \cup \{P^i_n\}$ .
\end{lem}

\begin{proof}
We prove first \eqref{Projec1} and \eqref{Projec5}. By \eqref{Pieq} we have $|\Pi^{b^i_{Q_0}}f |^{p_i} = |f|^{p_i} 1_{Q_0 \backslash \cup P_n^i} + \sum_n{|[f]_{P_n^i}b^i_{P^i_n}|^{p_i}}$. By \eqref{accretive2} for $b^i_{P^i_n}$ and \eqref{packing}, we have
\begin{align*} 
 \int_{Q_0}{|\Pi^{b^i_{Q_0}}f |^{p_i} d\mu} & \leq \| f\|^{p_i}_{L^{\infty}(Q_0)} \mu(Q_0) +  \| f\|^{p_i}_{L^{\infty}(Q_0)} \sum_n {\int_{Q_0}{|b^i_{P_n^i}|^{p_i}d\mu}} \\
 & \leq \| f\|^{p_i}_{L^{\infty}(Q_0)} \mu(Q_0) +  \| f\|^{p_i}_{L^{\infty}(Q_0)}C_A \sum_n{\mu(P^i_n)} \leq (1 + C_A) \| f\|^{p_i}_{L^{\infty}(Q_0)} \mu(Q_0).
\end{align*}
If $Q \in \Omega^i \cup \Omega^i_{buffer} \cup \{P^i_n\}$, then $Q$ cannot be contained in any $P^i_n$. Thus, applying again \eqref{accretive2} for $b^i_{P^i_n}$ and the disjointness of the cubes $P^i_n$, we have
\begin{align*} 
 \int_{Q}{|\Pi^{b^i_{Q_0}}f |^{p_i} d\mu} & \leq \| f\|^{p_i}_{L^{\infty}(Q_0)} \mu(Q \backslash \cup P^i_n ) +  \| f\|^{p_i}_{L^{\infty}(Q_0)} C_A \!\!\!\sum_{n; P^i_n \subset Q} {\mu(P^i_n)}  \leq (1 + C_A) \| f\|^{p_i}_{L^{\infty}(Q_0)} \mu(Q).
\end{align*}
For \eqref{Projec2}, observe that because of \eqref{accretive1} for $b^i_{P^i_n}$,
\[ \int_Q {\Pi^{b^i_{Q_0}}f  d\mu} = \int_{Q \backslash \cup P^i_n}{f d\mu} + \sum_{n; P^i_n \subset Q}{[f]_{P^i_n} \mu(P^i_n)  } = \int_Q{f d\mu},\]
and \eqref{Projec2} follows. Note that property \eqref{Projec3} is a direct consequence of \eqref{Projec2}. Only the property \eqref{Projec4} thus remains to be proved. Observe that for a cube $Q \in \Omega^i \cup \Omega^i_{buffer} \cup \{P^i_n\}$, and for any $n$, $[f1_Q]_{P^i_n}=0$ if $Q\cap P^i_n = \varnothing$ and  $[f1_Q]_{P^i_n}= [f]_{P^i_n}$ if $P^i_n \subset Q$. As we have already remarked, there is no other possibility, so that
\[ \Pi^{b^i_{Q_0}}(f1_Q) = f 1_{Q\backslash \cup P^i_n} + \sum_{n;P^i_n \subset Q}{[f]_{P^i_n} b^i_{P^i_n}} = (\Pi^{b^i_{Q_0}}f )1_Q.  \]
\end{proof}

\subsection{Sketch of the proof of Theorem \ref{AR}}
To prove Theorem \ref{AR}, we will begin by showing that we can reduce to proving that for every $f,g$ supported on a reference cube $Q_0$ and uniformly bounded by $1$, we have $|\langle \Pi^{b^2_{Q_0}}f, T( \Pi^{b^1_{Q_0}} g) \rangle| \leq C \mu(Q_0).$ This is done in Section $6$ and involves estimating the error terms; we will see that stopping-time Lemma \ref{Hofmann} plays a central part in this, and particularly \eqref{packing} and \eqref{Hofmann3}. To prove the above inequality, we apply \eqref{Hofmann5} of Lemma \ref{Hofmann} to decompose the stopping-time projectors $\Pi$ on our adapted Haar wavelet basis given by Proposition \ref{representation}. This gives us a number of matrix coefficients to estimate, which, to put it briefly, is done using the decay of the kernel of $T$ and hypotheses \eqref{accretive2}, \eqref{accretive3}, \eqref{WBP1} and \eqref{WBP2} of Theorem \ref{AR}.
\medskip

\section{Reductions}

\medskip

We now proceed to the proof of Theorem \ref{AR}. Let $T$ be a singular integral operator with locally bounded kernel on $X$, and assume that there exists a $(p,q)$ dyadic pseudo-accretive system $(\{b^1_Q\}, \{ b^2_Q \})$ adapted to $T$. It is of classical knowledge that to prove the $L^2$ boundedness of $T$, it suffices to show that there exists a constant $C<+\infty$ such that for every dyadic cube $Q$,
\begin{equation} \label{T1}
\| T 1_Q \|_{L^1(Q)} \leq C \mu(Q) \quad \mathrm{and} \quad  \| T^{\ast} 1_Q  \|_{L^1(Q)} \leq C \mu(Q).
\end{equation}
By symmetry of our hypotheses, it is enough to prove the second assertion. We give ourselves a reference cube $Q_0$, and we do this for all the cubes $Q$ contained in $Q_0$. Since $Q_0$ is arbitrary, the general case follows, as long as our constants are independent of $Q_0$. As a matter of fact, let us introduce the equivalence relation on the dyadic cubes "$Q_1 \sim Q_2 \Leftrightarrow \exists  \, Q$ dyadic cube such that $Q_1\cup Q_2 \subset Q$". It is easy to see that in a space of homogeneous type, there must be a finite number of equivalence classes for this relation. Thus, the argument which follows will give us a constant $C_i$ for each one of those equivalent classes, by letting $Q_0$ grow into one such equivalent class. But since they come in a finite number, \eqref{T1} follows.


Let us now proceed to our argument. For a dyadic cube $Q$, we denote by $E_Q$ the set of all functions supported on $Q$ and uniformly bounded by $1$. By duality, it suffices to prove that
 \[ A = \sup_{Q\subset Q_0 \atop f \in E_Q }{\left |  \frac{1}{\mu(Q)} \langle  f, T 1_Q  \rangle \right |} \lesssim 1. \]
Note that $A <+\infty$ with our a priori assumption $C = \|K\|_{L^{\infty}(Q_0 \times Q_0)} < +\infty$ and $\mu(Q_0)< +\infty$. Indeed, $|\langle f, T1_Q \rangle| \leq C \|f\|_{\infty} \mu(Q)^2$ so $A \leq C \mu(Q_0) < +\infty.$ Of course, this is not the bound we are after.

\subsection{First reduction}Let
\[ B = \sup_{Q\subset Q_0 \atop f \in E_Q }{\left |  \frac{1}{\mu(Q)} \langle  \Pi^{b^2_Q}f, T 1_Q  \rangle \right |}.  \]
\begin{lem}\label{red1}
We have $A \lesssim B + 1.$
\end{lem}

\begin{proof}
Let us fix a cube $Q$ in $X$, a function $f \in E_Q$, and for the sake of simplicity, denote $b^2_Q$ by $b^2$  and $\Pi^{b^2_Q}$ by $\Pi^2$. We write 
\[ f = \Pi^2 f + \sum_{j}{f_j}, \quad \mathrm{with} \, \, f_j =  f 1_{P^2_j} - [f]_{P_j^2} b^2_{P_j^2}.  \]
We thus have $\langle f,T1_Q \rangle = \langle \Pi^2 f, T1_Q \rangle + \Sigma$, with
\[ \Sigma = \sum_{j}{\langle f_j , T 1_Q \rangle},  \]
which we decompose in two parts
\[ \Sigma_1 = \sum_{j}{\langle f_j , T 1_{P_j^2}  \rangle}  \quad \mathrm{and} \quad \Sigma_2 = \sum_j{  \langle  f_j , T 1_{Q \backslash P_j^2}  \rangle}. \]
By definition of the quantity $A$, the fact that $f1_{P^2_j} \in E_{P_j^2}$, and \eqref{packing}, we have
\[ \sum_j{| \langle  f 1_{P^2_j}, T 1_{P_j^2}  \rangle |} \leq A \sum_j{\mu(P_j^2)} \leq A (1 - \varepsilon) \mu(Q). \]
By \eqref{accretive3} for $T^{\ast}(b^2_{P^2_j})$, we also have
\[ \sum_j{| \langle   [f]_{P_j^2} b^2_{P_j^2}, T 1_{P_j^2}  \rangle |} \lesssim \sum_j{\mu(P_j^2)} \lesssim \mu(Q), \]
which takes care of the sum $\Sigma_1$.
For the sum $\Sigma_2$, we write
\[ \Sigma_2 =   \sum_j{  \langle  f_j , T 1_{Q \backslash \widehat{P_j^2}} \rangle} +  \sum_j{  \langle  f_j , T 1_{\widehat{P_j^2} \backslash P_j^2}  \rangle}. \]
As the functions $f_j$ have mean zero, a standard computation using \eqref{standard estimate dyadic} allows us to bound the first sum by $\sum {\mu(P_j^2)} \lesssim \mu(Q).$ For the second sum, \eqref{Hardy1} applied to $f_j \in L^q(Q)$ by \eqref{accretive2} for $b^2_{P^2_j}$, and $1_Q \in L^{q'}(Q)$ gives us the same bound. We have obtained
\[ | \langle f,T1_Q \rangle | \leq B\mu(Q) +A(1-\varepsilon)\mu(Q) + C\mu(Q), \]
and Lemma \ref{red1} follows using $A<+\infty$.
\end{proof}
\medskip

\subsection{Second reduction}This is where our argument departs from the ones in \cite{AHMTT} or \cite{Hofmann}. Let Q be a fixed dyadic subcube of $Q_0$. Let us consider
\[  A_Q = \sup{\left |  \frac{1}{\mu(Q')} \langle  \Pi^{b^2_Q}  f, T1_{Q'} \rangle \right |},\]
\[ B_Q =  \sup{\left |  \frac{1}{\mu(Q')} \langle  \Pi^{b^2_Q}  f, T  \Pi^{b^1_{Q'}} 1_{Q'}  \rangle \right |}, \]
where the suprema are taken over all the $b^2_Q$ pseudo-accretive subcubes $Q'$ of $Q$, and all the functions $f \in E_{Q'}$. As $Q$ is itself $b^2_Q$ pseudo-accretive, it is clear that $B \leq \sup_Q{A_Q}$. Again $A_Q<+\infty$ for each $Q$ by our qualitative assumptions.

\begin{lem}\label{red2}
We have $A_Q \lesssim B_Q + 1.$
\end{lem}

\begin{proof}
As before, let us fix a pseudo-accretive subcube $Q'$ of $Q$, a function $f \in E_{Q'}$, and let us forget $Q$ and $Q'$ in our notations for the sake of simplicity. We write
\[ 1_{Q'} - \Pi^11_{Q'} = \sum_i{g_i}, \quad \mathrm{with} \, \, g_i = 1_{P^1_i} - b^1_{P^1_i}, \]
where we recall that the $P^1_i$ are given by the decomposition of $Q'$ with respect to $b^1_{Q'}$ as in Lemma \ref{Hofmann} and $\Pi^1= \Pi^{b^1_{Q'}}$.
We thus have $\langle \Pi^2 f, T1_{Q'} \rangle = \langle \Pi^2 f, T(\Pi^1 1_{Q'}) \rangle + \Sigma,$ with
\[ \Sigma = \sum_i{\langle  \Pi^2  f, T g_i  \rangle}. \]
We first consider the sum $\Sigma_1$ running over the indices $i$ such that the cubes $P_i^1$ are pa $2$ (here, $P^1_i$ pa $2$ means that $P^1_i \in \Omega^2 \cup \Omega^2_{buffer}$), and we write
\[ \Pi^2 f = \Pi^2 (f 1_{P_i^1}) + \Pi^2 (f 1_{Q' \backslash P^1_i}).  \]
We then have
\[ \langle \Pi^2 (f 1_{P_i^1}), Tg_i  \rangle =  \langle \Pi^2 (f 1_{P_i^1}), T 1_{P^1_i}   \rangle -  \langle \Pi^2 (f 1_{P_i^1}), T b^1_{P^1_i}  \rangle  \]
By definition of $A_Q$, the fact that $f 1_{P^1_i} \in E_{P_i^1} $, and \eqref{packing}, we have
\[ \sum_{P^1_i \, pa \, 2}{|  \langle \Pi^2 (f 1_{P_i^1}), T 1_{P^1_i}   \rangle |} \leq A_Q \sum_{P^1_i \, pa \, 2}{\mu(P^1_i)} \leq A_Q (1-\varepsilon) \mu(Q'). \] 
On the other hand, by \eqref{Projec4}, \eqref{Projec5}, and \eqref{accretive3} for $T(b^1_{P^1_i})$ on $P^1_i$, we have
\[  \sum_{P^1_i \, pa \, 2}{| \langle \Pi^2 (f 1_{P_i^1}), T b^1_{P^1_i}  \rangle |}  \lesssim \sum_{P^1_i \, pa \, 2}{\| \Pi^2 f \|_{L^q(P^1_i)}  \| T (b^1_{P^1_i}) \|_{L^{q'}(P^1_i)}} \lesssim \sum_{P^1_i \, pa \, 2}{\mu(P^1_i)} \lesssim \mu(Q').   \]
Finally, again for $P^1_i$ pa $2$, we have by Lemma \ref{Projec},
\[ \Pi^2 (f 1_{Q' \backslash P^1_i}) = f 1_{F \cap (Q' \backslash P^1_i)} + \sum_{P^2_j \cap P^1_i = \varnothing}{[f]_{P^2_j} b^2_{P^2_j}} = f 1_{Q' \backslash P^1_i} -  \sum_{P^2_j \cap P^1_i = \varnothing}{f_j}, \]
with $f_j = f 1_{P^2_j} - [f]_{P^2_j} b^2_{P^2_j}$ as before. The sum $\sum_{P^1_i \, pa \, 2}{|\langle  f 1_{Q' \backslash P^1_i} , T g_i  \rangle |}$ can be estimated in the same way we treated the sum $\Sigma_2$ in Lemma \ref{red1}. It thus remains to estimate the term
\[ \Sigma'_1 = - \sum_{P^1_i \, pa \, 2}{\sum_{P^2_j \cap P^1_i = \varnothing}{\langle  f_j, T g_i  \rangle}}, \]
which we defer for now.
We now estimate the sum $\Sigma_2$ running over those indices $i$ such that $P^1_i$ is not pa $2$, which means that there exists an index $k$ (unique), with $P^1_i \subset P^2_k$. For a fixed index $k$, denote 
\[G_k =  \sum_{P^1_i \subset P^2_k}{g_i}. \]
Observe that $G_k$ is supported on $P^2_k$, with mean zero and that $\int_{P^2_k}{|G_k|^p d\mu} \lesssim \mu(P^2_k)$ by \eqref{accretive3} for the $b^1_{P^1_i}$ on $P^1_i$ and $P^1_i \subset P^2_k$. Write
\[ \Pi^2 f = f 1_{Q' \backslash P^2_k} + [f]_{P^2_k}b^2_{P^2_k} - \sum_{j\neq k }{f_j}, \]
which allows us to decompose $\Sigma_2$ into three parts: $\Sigma_2 = \Sigma'_2 + \Sigma''_2 + \Sigma'''_2.$
Once more, we estimate the sum $\Sigma'_2 = \sum_k{\langle  f 1_{Q' \backslash P^2_k}, T G_k \rangle}$ as we estimated the sum $\Sigma_2$ in Lemma \ref{red1}. The sum $\Sigma''_2 = \sum_k{[f]_{P^2_k} \langle  T^{\ast}(b^2_{P^2_k}), G_k \rangle}$ is also easy to manage using the hypothesis \eqref{accretive3} for $T^{\ast}(b^2_{P^2_k})$ on $P^2_k$ and the support and size properties of the functions $G_k$. Overall, we are left to estimate the term
\begin{align*}
 \Sigma_3 = \Sigma'_1 + \Sigma'''_2 & = - \sum_{P^1_i \, pa \, 2}{\sum_{P^2_j \cap P^1_i = \varnothing}{\langle  f_j, T g_i  \rangle}} - \sum_{P^1_i \, non \, pa \, 2}{\sum_{P^2_j \cap P^1_i = \varnothing}{\langle  f_j, T g_i  \rangle}} \\
 & = - \sum_{(i,j): P^2_j \cap P^1_i = \varnothing}{\langle  f_j, T g_i  \rangle}. \\
 \end{align*}
We split the sum into two parts, depending on whether $(i,j)$ is such that $l(P^2_j) \leq l(P^1_i)$ or not. Let us consider for example the sum for the indices $(i,j)$ satisfying that condition, the other sum can be treated in a symmetric way. We want to estimate
 \[\Sigma'_3 = - \sum_{(i,j): P^2_j \cap P^1_i = \varnothing \atop l(P^2_j) \leq l(P^1_i)}{\langle  f_j, T g_i  \rangle}. \]
Fix the index $i$. Then either $P^2_j \subset \widehat{P^1_i} \backslash P^1_i$ or $P^2_j \cap  \widehat{P^1_i}  = \varnothing$ because of the relative sizes of those cubes. Set $F_i = \sum{f_j}$ where the sum runs over all the indices $j$ such that  $P^2_j \subset \widehat{P^1_i} \backslash P^1_i$ . Lemma \ref{Hardy} and the use of hypothesis \eqref{accretive3} for $T(b^1_{P^1_i})$ on $\widehat{P^1_i}$ assure that
\[ | \langle F_i , Tg_i \rangle | \leq | \langle F_i , T 1_{P^1_i} \rangle | + | \langle F_i , T b^1_{P^1_i} \rangle | \lesssim \mu(P^1_i). \]
For the indices $j$ such that $P^2_j \cap \widehat{P^1_i} = \varnothing$ we have to work a little bit more. Using that $g_i$ is of mean $0$ and applying \eqref{standard estimate dyadic}, we obtain for each concerned couple $(i,j)$
\begin{align*}
 | \langle f_j, T g_i \rangle | &  \lesssim \int_{P^1_i}{|g_i(y)|} \int_{P_j^2}{|f_j(x)| \left ( \frac{l(P^1_i)}{\rho (x,y)} \right ) ^{\alpha} \frac{1}{\lambda (x,z_{P^1_i})} d\mu(x) d\mu(y)}    \\
 & \lesssim  \int_{P^1_i}{|g_i(y)|} \sum_{k \geq 1} { \int_{P_j^2 \cap C_k (P^1_i)}{|f_j(x)| \left ( \frac{l(P^1_i)}{\delta^{-k + 1} l(P^1_i)} \right ) ^{\alpha} \frac{1}{\mu (B(z_{P^1_i}, \delta^{-k + 1} l(P^1_i)))} d\mu(x) d\mu(y)}}   \\
 & \lesssim \mu(P^1_i) \sum_{k \geq 1} { \delta^{\alpha k} \mu (B(z_{P^1_i}, \delta^{-k + 1} l(P^1_i)))^{-1}  \int_{P^2_j \cap C_k (P^1_i)}{|f_j| d\mu} }     ,\\ 
 \end{align*}
where $z_{P^1_i}$ denotes the center of the cube $P^1_i$ and 
\[ C_k (P^1_i) = \{ x \in X | \delta^{-k + 1} l(P^1_i) \leq \rho (x,z_{P^1_i}) < \delta^{-k} l(P^1_i) \}. \]
Now, remark that since $l(P^2_j) \leq l(P^1_i)$, for each $j$ there are at most two sets $C_k(P^1_i)$ non disjoint with $P^2_j$. Thus we can sum over $j$ with $k$ fixed, and as the functions $f_j$ are supported on the cubes $P^2_j$ with $\int{\!  |f_j|} \lesssim \mu (P^2_j),$ we get
\[\sum_j{\int_{P^2_j \cap C_k (P^1_i)}{|f_j| d\mu}} \lesssim \sum_{P^2_j \cap C_k(P^1_i) \neq \varnothing}{\mu (P^2_j)} \lesssim \mu(C_k(P^1_i)). \]
Then, summing over $k$, we have
\begin{align*}
\sum_{j: P^2_j \cap \widehat{P^1_i} = \varnothing}{|\langle f_j, T g_i \rangle |}&  \lesssim \sum_{k \geq 1}{\delta^{\alpha k} \mu (B(z_{P^1_i}, \delta^{-k + 1} l(P^1_i)))^{-1} \mu(C_k(P^1_i)) \mu(P^1_i)} \\
& \lesssim  \sum_{k \geq 1}{\delta^{\alpha k} \mu(P^1_i) } \lesssim \mu(P^1_i).\\
\end{align*}
It remains only to sum over $i$ to get $|\Sigma_3| \lesssim \mu(Q')$. To summarize, we have obtained
\[ | \langle \Pi^2 f , T1_{Q'}  \rangle | \leq B_Q\mu(Q') + A_Q(1- \varepsilon)\mu(Q') +C\mu(Q')   ,\]
and Lemma \ref{red2} follows using $A_Q <\infty$.
\end{proof}
\begin{remark}
Observe that this argument does not require the properties \eqref{WBP1} and \eqref{WBP2} to be satisfied. We used \eqref{accretive3} (and not \eqref{accretive3bis}) once. No further conditions on $p,q$ are required.
\end{remark}

\medskip

As a consequence of what precedes, we will be done if we prove there exists a constant $C<+\infty$ such that
\[ |\langle \Pi^{b^2_Q}f , T(\Pi^{b^1_{Q'}}g)  \rangle | \leq C \mu(Q') \]
for all dyadic subcube $Q$ of $Q_0$, all $b^2_Q$ pseudo-accretive subcube $Q'$ of $Q$, and all $f,g \in E_{Q'}$. By \eqref{Projec4} of Lemma \ref{Projec}, if $Q'$ is a $b^2_Q$ pseudo-accretive subcube of $Q$ and $f,g \in E_{Q'}$, we have
\[ \Pi^{b^2_Q}f = \Pi^{b^2_Q}(f1_{Q'}) = (\Pi^{b^2_Q}f)1_{Q'} = f 1_{Q' \backslash (\cup{P^2_j})} + \sum_{P^2_j \subset Q'}{[f]_{P^2_j} b^2_{P^2_j}}, \]
and
\[ \Pi^{b^1_{Q'}}g = g 1_{Q' \backslash (\cup{P^1_i})} + \sum_{P^1_i \subset Q'}{[g]_{P^1_i} b^1_{P^1_i}}.    \]
Thus, the expressions reduce to partitions of $Q'$ by dyadic subcubes and possibly the complement of their union. From this point on, we no longer use the $(1-\varepsilon)$-packing property \eqref{packing} in Lemma \ref{Hofmann} and the $1$-packing property suffices, that is we do not care if $Q' \backslash (\cup P_j^2) = \varnothing$, which is a possibility. This means that $Q'$ can become our reference cube, and $b^2_Q$ could be as well replaced by any function $b^2$ for which the non pseudo-accretive cubes are the $P^2_j \subset Q'$ and Lemma \ref{Hofmann} holds with $(1-\varepsilon)$ replaced by $1$ in \eqref{packing}. To simplify notation, we shall do this and set $Q'=Q_0$, and even further assume it is of generation $0$. We have to prove
\begin{equation} \label{reduction}
|\langle \Pi^2 f, T(\Pi^1g) \rangle | \leq C \mu(Q_0)
\end{equation}
for any $f,g \in E_{Q_0}$ where $\Pi^i = \Pi^{b^i_{Q_0}},$ and Lemma \ref{Hofmann} holds for both with the $(1-\varepsilon)$ of \eqref{packing} replaced by $1$. For simplicity, we denote $b^i_{Q_0}$ by $b^i$. We shall also say $Q$ spa $i$ (resp. dpa $i$) if $Q \in \Omega^i$ (resp. $Q \in \Omega^i_{buffer}$). We also set $\Delta_Q^i = \Delta^{b^i}_{Q}$ and $\xi^i_Q = \xi^{b^i}_Q$.

\medskip

\section{BCR algorithm and end of the proof}

\medskip

We intend to prove \eqref{reduction} with the simplification of notation.

\subsection{Representation of projections}

By the decomposition formula \eqref{Hofmann5} of Lemma \ref{Hofmann}, we have
\[ \Pi^{i} f = [ f ]_{Q_0} b^i + \sum_{Q \, \mathrm{spa} \, i}{b^i \Delta^{i}_{Q}f} + \sum_{Q \, \mathrm{dpa} \, i}{\xi^{i}_Q f}. \]
Set $E_0^{i} f =  [ f ]_{Q_0} b^i$, and $\forall j \geq 0$,
\[ D^{i}_j f =  \sum_{Q \, \mathrm{spa} \, i \atop l(Q) = \delta^j}{b^i \Delta^{i}_{Q}f} + \sum_{Q \, \mathrm{dpa} \, i \atop l(Q) = \delta^j}{\xi^{i}_Q f} =  \sum_{Q \, \mathrm{spa} \, i \atop l(Q) = \delta^j}{\langle f, \phi_Q^i \rangle b^i \phi_Q^i} + \sum_{Q \, \mathrm{dpa} \, i \atop l(Q) = \delta^j}{\xi^{i}_Q f} \]
\[ E_j^{i} f = E_0^{i} f + \sum_{l=0}^{j-1}{D_l^{i} f}. \]
\begin{lem} \label{Pi}
With the previous notation, we have 
\begin{equation} \label{Ej}
E_j^{i} f = \sum_{Q \, \mathrm{pa} \, i \atop l(Q) = \delta^j}{\frac{[f]_Q}{[b^i]_Q}b^i1_{Q}} + \sum_{P_n^i \atop l(P^i_n) \geq \delta^j}{[f]_{P^i_n}b^i_{P^i_n}} \, .
\end{equation}
Furthermore, we have when $f\in L^{\infty}(Q_0)$
\begin{equation} \label{L1}
\Pi^{i} f =  \lim_{j \rightarrow + \infty}{E_j^{i}f},
\end{equation}
with the convergence taking place in $L^1(Q_0)$.
\end{lem}

\begin{proof} 
The first part of Lemma \ref{Pi} is a linear algebra computation, we will omit the detail. The second assertion is an application of the dominated convergence theorem in $L^1(Q_0)$. Observe first that, by \eqref{Ej}, we have
\[ |E_j^{i}f | \leq C^{-1} |b^i| + \sum_{n}{|b^i_{P^i_n}|}, \]
where $C = \inf{|[b]_Q|}$, the infimum taken over the pa $i$ cubes $Q$. To see that $E_j^i f$ converges $\mu$-a.e. on $Q_0$ towards $\Pi^i$, use again \eqref{Ej} and observe that for every $n$, $E_j^i f$ is constant on $P^i_n$ for $j$ large enough (depending on $n$), and equal to $[f]_{P^i_n}b^i_{P^i_n}$. If $x \in F^i = Q_0 \backslash \cup{P^i_n}$, then by the Lebesgue differentiation theorem $\frac{[f]_Q}{[b]_Q}$ tends $\mu$-a.e. towards $\frac{f(x)}{b(x)}$ when $Q$ tends towards $x$. Thus $E^i_j f$ tends $\mu$-a.e. towards $f$ on $F^i$, and, by \eqref{Pieq}, $E^i_j f$ tends $\mu$-a.e. towards $\Pi^i f$ on $Q_0$. It remains only to apply the dominated convergence theorem in $L^1(Q_0)$ to get \eqref{L1}.

\end{proof}


Now, because of \eqref{L1}, and because we assumed the kernel of the operator $T$ to be locally bounded, we can take the limit and write
\begin{align*}
&  \langle \Pi^2 f, T \Pi^1 g \rangle  = \lim_{j \rightarrow +\infty}{ \langle E_j^2 f  , TE_j^1 g  \rangle} \\
 & = \langle E_0^2 f, TE_0^1 g \rangle + \sum_{j \geq 0}{\left ( \langle D_j^2 f, TD_j^1 g \rangle + \langle D_j^2 f, TE_j^1 g \rangle + \langle E_j^2 f, TD_j^1 g \rangle \right )}\\
 & = \langle E_0^2 f, TE_0^1 g \rangle + \langle f,Ug \rangle  +  \langle f,Vg \rangle + \langle f,Wg \rangle.
\end{align*}
This is the so called BCR algorithm, introduced in \cite{BCR} for the classical martingale differences. Here we use the adapted martingale differences. Let us remark that the first term is easy to estimate because $f,g \in E_{Q_0}$, and because of \eqref{accretive2}, \eqref{accretive3}. Remark also that the two last terms $V,W$ can be treated in the same way by duality. They lead to paraproduct type operators. The diagonal term $U$ is relatively easy to treat. With this decomposition and the wavelet representation, we will obtain matrix coefficients of the form $\langle b^2 \theta_Q , T(b^1 \theta_R) \rangle$, where the functions $\theta_Q, \theta_R$ are supported on the cubes $Q, R$ respectively. The idea is mostly to use the kernel decay and the standard Calder\'on-Zygmund estimates \eqref{standard estimate dyadic} when the cubes $Q,R$ are far away, and to use the weak boundedness properties \eqref{WBP1}, \eqref{WBP2} when they are close. \\

Before tackling those estimations, let us first introduce some more notation. For $Q$ and $R$ two dyadic cubes of the space of homogeneous type $X$, set
\[ \mu(Q,R) = \mu(R,Q) = \inf_{x \in Q, y \in R}\left \{ \mu(B(x,\rho(x,y))), \mu(B(y,\rho(x,y))) \right \}, \]
and
\[ \alpha_{Q,R} = \alpha_{R,Q} = \begin{cases}
 \mu(Q)^{\frac{1}{2}} \mu(R)^{\frac{1}{2}} \left ( {\frac{\inf(l(Q),l(R))}{\rho(Q,R)}} \right )^{\alpha} \frac{1}{\mu(Q,R)}   & \mathrm{if} \ \ \rho(Q,R) \geq \sup(l(Q), l(R)) \\
1  & \mathrm{if} \ \ l(Q) = l(R), \ \ \mathrm{and} \ \ \rho(Q,R) < l(Q).
 \end{cases}\\ \]
Of course $\alpha_{Q,R}$ is not defined on all pairs of cubes $(Q,R)$, but that does not matter as we will only use this notation when we are in one of the above cases. We will frequently use some coefficient estimates to bound the terms of the form $\langle b^2 \theta_Q , T(b^1 \theta_R) \rangle$ we evoked earlier, and we will constantly refer to Appendix $B$ for the detail of these estimates. 

\medskip

\subsection{Analysis of $U$}First, let us decompose
\[ \langle f,Ug \rangle =  \sum_{j \geq 0}{ \langle D_j^2 f, T D_j^1 g} \rangle =  \langle f,U_1 g \rangle + \langle f,U_2 g \rangle + \langle f,U_3 g \rangle + \langle f,U_4 g \rangle, \]
with 
\begin{align*}
\langle f,U_1 g \rangle & =  {\sum_{Q \, \mathrm{spa} \, 2 ; R \, \mathrm{spa} \, 1 \atop l(R) = l(Q)}{\langle f, \phi_Q^2 \rangle \langle b^2 \phi_Q^2 , T(b^1 \phi_R^1 ) \rangle \langle g, \phi_R^1 \rangle}},  \\
\medskip
\langle f,U_2 g \rangle & =   {\sum_{Q \, \mathrm{spa} \, 2 ; R \, \mathrm{dpa} \, 1 \atop l(R) = l(Q)}{\langle f, \phi_Q^2 \rangle  \langle b^2 \phi_Q^2 , T(\xi^1_R g)\rangle}},  \\
\medskip
\langle f,U_3 g \rangle & =   {\sum_{Q \, \mathrm{dpa} \, 2 ; R  \, \mathrm{spa} \, 1 \atop l(R) = l(Q)}{\langle \xi^2_Q f, T(b^1  \phi_R^1 )\rangle \langle g, \phi_R^1 \rangle}}, \\
\medskip
\langle f,U_4 g \rangle & =   {\sum_{Q  \, \mathrm{dpa} \, 2 ; R  \, \mathrm{dpa} \, 1 \atop l(R) = l(Q)}{\langle \xi^2_Q f, T(\xi^1_R g)\rangle}}.\\
\end{align*}
Observe that the sums with respect to $j$ have disappeared once we notice they force the cubes $Q$ and $R$ to have equal lengths. 

\subsubsection*{\textbf{Estimate of $\langle f,U_1 g \rangle $}}We refer to Appendix $B$ for the detail, but, applying  \eqref{eqn:1} of Lemma \ref{estimates}, we have
\[ | \langle f,U_1 g \rangle | \lesssim {\sum_{Q \, \mathrm{spa} \, 2 ; R \, \mathrm{spa} \, 1 \atop l(R) = l(Q)}{ \alpha_{Q,R} | \langle f, \phi_Q^2 \rangle | \, |\langle g, \phi_R^1 \rangle |  }}. \]
Remember that by Lemma \ref{Mart}
\[ \sum_{Q \, \mathrm{spa} \, 2}{|\langle f, \phi_Q^2 \rangle |^2} \lesssim  \|f\|^2_2 \lesssim \mu(Q_0),  \]
and similarly for $\langle g , \phi_R^1 \rangle $. Therefore, we have by the Cauchy-Schwarz inequality
\begin{align*} 
| \langle f,U_1 g \rangle | & \lesssim {\sum_{Q \, \mathrm{spa} \, 2 ; R \, \mathrm{spa} \, 1 \atop l(R) = l(Q) }{ \alpha_{Q,R} \, | \langle f, \phi_Q^2 \rangle | \, |\langle g, \phi_R^1 \rangle |  }} \\
& \lesssim   { \left ( \sum_{Q \, \mathrm{spa} \, 2 ; R \, \mathrm{spa} \, 1 \atop l(R) = l(Q) }{ \!\!\!\!\!\!\!\!\! \alpha_{Q,R} \left( \frac{\mu(R)}{\mu(Q)} \right)^{1/2} | \langle f, \phi_Q^2 \rangle |^2  }  \right )^{1/2} \!\! \left ( \sum_{Q \, \mathrm{spa} \, 2 ; R \, \mathrm{spa} \, 1 \atop l(R) = l(Q) }{\!\!\!\!\!\!\!\!\!  \alpha_{Q,R} \left( \frac{\mu(Q)}{\mu(R)} \right)^{1/2} |\langle g, \phi_R^1 \rangle |^2   } \right )^{1/2}  }\\
\end{align*}
Let us state a lemma that will handle those sums.
\begin{lem}\label{Calcul1}
Let $R$ be a fixed dyadic cube of the space of homogeneous type $X$. We have the following summing property
\[ \sum_{Q ; l(Q)=l(R) \atop \rho(Q,R) \geq l(R)}{\mu(Q) \left ( \frac{l(Q)}{\rho(Q,R)} \right )^{\alpha} \frac{1}{\mu(Q,R)} } \lesssim 1.  \]
Furthermore, if $p \in \N$, we have the stronger summing property
\[ \sum_{Q ; l(Q)= \delta^p l(R) \atop \rho(Q,R) \geq l(R)}{\mu(Q) \left ( \frac{l(Q)}{\rho(Q,R)} \right )^{\alpha} \frac{1}{\mu(Q,R)} } \lesssim \delta^{p \alpha},  \]
which, by summing over $p \in \N$, yields
\[ \sum_{Q ; l(Q) \leq l(R) \atop \rho(Q,R) \geq l(R)}{\mu(Q) \left ( \frac{l(Q)}{\rho(Q,R)} \right )^{\alpha} \frac{1}{\mu(Q,R)} } \lesssim 1.  \]
The exponent $\alpha$ can be replaced by any $\alpha' >0$. 
\end{lem}

\begin{proof}
It is enough to prove the second inequality. The idea is to split the sum for $\delta^{-m}l(R) \leq \rho(Q,R) < \delta^{-m-1}l(R), m \in \N$ (here and subsequently, we then write that $\rho(Q,R) \sim \delta^{-m}l(R)$):  write
\begin{align*}
\sum_{Q ; l(Q)= \delta^p l(R) \atop \rho(Q,R) \geq l(R)}{\mu(Q) \left ( \frac{l(Q)}{\rho(Q,R)} \right )^{\alpha} \frac{1}{\mu(Q,R)} } & = \delta^{p\alpha} \sum_{m \geq 0} \!\!\!\!  \sum_{Q ; l(Q)= \delta^p l(R) \atop \rho(Q,R) \sim \delta^{-m} l(R)}{\!\!\!\!\!\!\! \mu(Q) \left ( \frac{l(R)}{\rho(Q,R)} \right )^{\alpha} \frac{1}{\mu(Q,R)}} \\
& \lesssim  \delta^{p\alpha} \sum_{m \geq 0}{\delta^{m\alpha} } \!\!\!  \!\!\!\!   \sum_{Q ; l(Q)= \delta^p l(R) \atop \rho(Q,R) \sim \delta^{-m} l(R)}{\!\!\frac{\mu(Q)}{\mu(Q,R)}}\\
& \lesssim  \delta^{p\alpha} \sum_{m \geq 0}{\delta^{m\alpha} } \!\!\!  \!\!\!\!   \sum_{Q ; l(Q)= \delta^p l(R) \atop \rho(Q,R) \sim \delta^{-m} l(R)}{\!\!\frac{\mu(Q)}{\mu(B(z_R, \delta^{-m} l(R)))}},\\
\end{align*}
where $z_R$ denotes the center of the cube $R$. All those cubes $Q$ are non overlapping as they are of the same generation, and they are all contained in a ball of comparable measure to the ball $B(z_R, \delta^{-m} l(R))$. We therefore have
\[ \sum_{Q ; l(Q)= \delta^p l(R) \atop \rho(Q,R) \sim \delta^{-m} l(R)}{\!\frac{\mu(Q)}{\mu(B(z_R, \delta^{-m} l(R)))}} \lesssim 1, \]
and summing then over $m \in \N$, we get the desired estimate. Moreover, it is clear that $\alpha$ can be replaced by any $\alpha' >0$.
\end{proof}

\noindent Now let us go back to our argument. For a fixed cube $R$, we write
\[ \sum_{Q \, \mathrm{spa} \, 2  \atop l(Q) = l(R) }{ \!\!\! \alpha_{Q,R} \left( \frac{\mu(Q)}{\mu(R)} \right)^{1/2}   }  =  \sum_{Q \, \mathrm{spa} \, 2 ; l(Q)=l(R) \atop \rho(Q,R) \geq l(R) }{ \!\!\! \mu(Q) \left ( \frac{l(Q)}{\rho(Q,R)} \right )^{\alpha} \frac{1}{\mu(Q,R)}} \, \, + \!\!\!\!\!\! \sum_{Q \, \mathrm{spa} \, 2 ; l(Q)=l(R) \atop \rho(Q,R) < l(R) }{ \!\!\!  \left( \frac{\mu(Q)}{\mu(R)} \right)^{1/2}   }.\]
Lemma \ref{Calcul1} insures that the first sum is uniformly bounded, and the second sum is bounded as well, because such cubes $Q$ and $R$ have comparable measure (they are neighbors), and any given cube $R$ has a uniformly bounded number of neighbors. 
We thus have for any given cube $R$,
\[ \sum_{Q \, \mathrm{spa} \, 2  \atop l(Q) = l(R) }{ \!\!\! \alpha_{Q,R} \left( \frac{\mu(Q)}{\mu(R)} \right)^{1/2}   } \lesssim 1, \]
and similarly, for any given cube $Q$, 
\[ \sum_{R \, \mathrm{spa} \, 1  \atop l(R) = l(Q) }{ \!\!\! \alpha_{Q,R} \left( \frac{\mu(R)}{\mu(Q)} \right)^{1/2}   } \lesssim 1. \]
Therefore,
\begin{align*}
| \langle f,U_1 g \rangle | & \lesssim  { \left ( \sum_{Q \, \mathrm{spa} \, 2  }{| \langle f, \phi_Q^2 \rangle |^2} \right )^{1/2} \left (  \sum_{R \, \mathrm{spa} \, 1 }{|\langle g, \phi_R^1 \rangle |^2}  \right )^{1/2}} \lesssim  \mu(Q_0).
\end{align*}

\subsubsection*{\textbf{Estimate of $\langle f, U_2 g \rangle$}}Remark that by duality, $\langle f, U_3 g \rangle$ is estimated in the same way. Refering to estimate \eqref{eqn:2} of Lemma \ref{estimates}, we have
\[ | \langle f, U_2 g \rangle | \lesssim  {\sum_{Q \, \mathrm{spa} \, 2 ; R \, \mathrm{dpa} \, 1 \atop l(R) = l(Q) }{\alpha_{Q,R} \, \mu(R)^{\frac{1}{2}} \, | \langle f, \phi_Q^2 \rangle | }}.  \]
We split this sum into two parts, depending on whether $\rho(Q,R) \geq l(Q)$ or not. By the Cauchy-Schwarz inequality and Lemma \ref{Calcul1}, we have
\begin{align*}
{\sum_{Q \, \mathrm{spa} \, 2 ; R \, \mathrm{dpa} \, 1 \atop{ l(R) = l(Q) \atop \rho(Q,R) \geq l(Q)}}{\!\!\!\!\!\alpha_{Q,R} \, \mu(R)^{\frac{1}{2}} \,  | \langle f, \phi_Q^2 \rangle |}} & \lesssim 
\left (\sum_{Q \, \mathrm{spa} \, 2 }{\!\!\! | \langle f, \phi_Q^2 \rangle |^2}\!\!\!\!\!\!\! \sum_{R \, \mathrm{dpa} \, 1; l(R)=l(Q)  \atop  \rho(Q,R) \geq l(Q) }{\!\!\!\!\!\!\!\mu(R) \left ( \frac{l(R)}{\rho(Q,R)}\right ) ^{\alpha} \frac{1}{\mu(Q,R)}} \right ) ^{1/2} \\
& \times \left (\sum_{R \, \mathrm{dpa} \, 1}{\!\!\mu(R)} \!\!\!\!\!\!\! \sum_{Q \, \mathrm{spa} \, 2 ; l(Q) = l(R) \atop  \rho(Q,R) \geq l(R) }{\!\!\!\!\!\!\!\mu(Q) \left ( \frac{l(Q)}{\rho(Q,R)}\right ) ^{\alpha} \frac{1}{\mu(Q,R)}} \right ) ^{1/2}\\
& \lesssim \left ( \sum_{Q \, \mathrm{spa} \, 2  }{ | \langle f, \phi_Q^2 \rangle |^2} \right ) ^{1/2} \times \left ( \sum_{R \, \mathrm{dpa} \, 1 }{\mu(R)} \right ) ^{1/2} \lesssim \mu(Q_0),
\end{align*}
the last inequality being a consequence of the fact that $\sum_{R \, \mathrm{dpa} \, 1  } {\mu(R)} \lesssim \mu(Q_0).$ The remaining sum is easy to estimate: since for a fixed $Q$ there is a uniformly bounded number of neighbors $R$, one can write
\begin{align*}
{\sum_{Q \, \mathrm{spa} \, 2 ; R \, \mathrm{dpa} \, 1 \atop{ l(R) = l(Q) \atop \rho(Q,R) < l(Q) }}{\!\!\!\!\! \mu(R)^{\frac{1}{2}} \,  | \langle f, \phi_Q^2 \rangle |}}  & \lesssim 
 \left ( \sum_{Q \, \mathrm{spa} \, 2} \sum_{R \, \mathrm{dpa} \, 1 \atop R \subset \widehat{Q}; l(R)=l(Q)}{ | \langle f, \phi_Q^2 \rangle |^2 } \right )^{1/2}  \left (  \sum_{R \, \mathrm{dpa} \, 1} \sum_{Q \, \mathrm{spa} \, 2 \atop Q \subset \widehat{R}; l(Q)=l(R) } {\mu(R)} \right )^{1/2}\\
&  \lesssim  \left (  \sum_{Q \, \mathrm{spa} \, 2 } { | \langle f, \phi_Q^2 \rangle |^2 } \right )^{1/2}  \left (  \sum_{R \, \mathrm{dpa} \, 1 } {\mu(R) } \right )^{1/2} \lesssim \mu(Q_0).
\end{align*}

\subsubsection*{\textbf{Estimate of $\langle f, U_4 g \rangle$}}This term is pretty easy to handle. Indeed, applying the coefficient estimate \eqref{eqn:3}, we have
\begin{align*}
| \langle f , U_4 g \rangle | & \lesssim {\sum_{Q  \, \mathrm{dpa} \, 2 ; R  \, \mathrm{dpa} \, 1 \atop l(R) = l(Q)}{ \alpha_{Q,R} \, \mu(Q)^{\frac{1}{2}} \, \mu(R)^{\frac{1}{2}}}}\\
& \lesssim \sum_{Q\, \mathrm{dpa} \, 2} \left ( \sum_{R  \, \mathrm{dpa} \, 1 ; l(R) = l(Q) \atop \rho(Q,R) < l(Q)}{ \!\!\!\!\!\!\!\mu(Q)^{\frac{1}{2}} \, \mu(R)^{\frac{1}{2}}} + \mu(Q) \!\!\!\!\! \sum_{R  \, \mathrm{dpa} \, 1 ; l(R) = l(Q) \atop \rho(Q,R) \geq l(Q)}{\!\!\!\!\!\!\!\mu(R) \left ( \frac{l(R)}{\rho(Q,R)}\right ) ^{\alpha} \frac{1}{\mu(Q,R)} }  \right )\\
& \lesssim \sum_{Q\, \mathrm{dpa} \, 2} {\mu(Q)} \lesssim \mu(Q_0),\\
\end{align*}
where we have used Lemma \ref{Calcul1} and once again the fact that any fixed cube $Q$ has a uniformly bounded number of neighbors.

\subsection{Analysis of $V$}As in the previous subsection, we write
\[ \langle f,Vg \rangle = \sum_{j \geq 0}{\langle D_j^2 f, T E_j^1 g \rangle} = \langle f,V_1 g \rangle + \langle f,V_2 g \rangle + \langle f,V_3 g \rangle + \langle f,V_4 g \rangle, \]
with
\begin{align*}
\langle f,V_1 g \rangle & =  {\sum_{Q  \, \mathrm{spa} \, 2 ; R  \, \mathrm{pa} \, 1 \atop l(R) = l(Q) }{  \langle f, \phi_Q^2 \rangle   \langle b^2 \phi_Q^2 , T(b^1 1_R)\rangle \frac{[g]_R}{[b^1]_R} }},\\
\medskip
\langle f,V_2 g \rangle & =   {\sum_{Q  \, \mathrm{spa} \, 2}{\sum_{l(P_i^1) \geq l(Q)}{  \langle f, \phi_Q^2 \rangle \langle b^2 \phi_Q^2 , T(b^1_{P^1_i})\rangle [g]_{P^1_i}}}},\\
\medskip
\langle f,V_3 g \rangle & =   {\sum_{Q  \, \mathrm{dpa} \, 2 ; R  \, \mathrm{pa} \, 1 \atop l(R) = l(Q)}{\langle \xi^2_Q f, T( b^1 1_R)\rangle \frac{[g]_R}{[b^1]_R} }},\\
\medskip
\langle f,V_4 g \rangle & =   {\sum_{Q  \, \mathrm{dpa} \, 2 }{\sum_{l(P^1_i) \geq l(Q)}{\langle \xi^2_Q f, T(b^1_{P^1_i})\rangle [g]_{P^1_i}}}}.\\\end{align*}
The term $V_1$ will be the most difficult term, it is where a parapoduct appears, which will require the use of hypothesis \eqref{WBP1} to be controlled. 

\subsubsection*{\textbf{Estimate of $\langle f,V_4 g \rangle$}}We split the sum into two parts, as before, depending on the distance between $P_i^1$ and $Q$. When $\rho(P_i^1,Q) \geq l(P_i^1)$, we have, by the coefficient estimate \eqref{eqn:6},
\begin{align*}
{\sum_{Q  \, \mathrm{dpa} \, 2}\!\!\!{\sum_{l(P^1_i) \geq l(Q) \atop \rho(P_i^1,Q) \geq l(P_i^1)}{\!\!\!\!\!\!\!\!\!  | \langle \xi^2_Q f, T(b^1_{P^1_i})\rangle [g]_{P^1_i}| }}} \, \, \,   & \lesssim  \, \sum_{P^1_i} \!\! \sum_{Q  \, \mathrm{dpa} \, 2 ; l(Q) \leq l(P^1_i) \atop  \rho(Q,P_i^1) \geq l(P_i^1)}{\!\!\!\!\!\!\mu(P^1_i)^{\frac{1}{2}} \mu(Q)^{\frac{1}{2}} \alpha_{Q, P^1_i} }\\
& \lesssim \,   \sum_{P^1_i}  {\mu(P^1_i)} \!\! \!\!\!\sum_{Q  \, \mathrm{dpa} \, 2 ; l(Q) \leq l(P^1_i) \atop  \rho(Q,P_i^1) \geq l(P_i^1)}{\!\!\!\!\!\!\!\!\!\!\! \mu(Q) \left ( \frac{l(Q)}{\rho(Q,R)} \right )^{\alpha} \frac{1}{\mu(Q,R)}  } \\
& \lesssim \sum_{P^1_i}{\mu(P^1_i)}  \lesssim \mu(Q_0),\\
\end{align*}
where we have once again used Lemma \ref{Calcul1} to obtain the third inequality, and the $1-$packing property of the $P^1_i$ to get the last one. It then remains to estimate the sum
\[\langle f, V_{4,1}g \rangle =  \sum_{P^1_i}  \sum_{Q  \, \mathrm{dpa} \, 2 ; l(Q) \leq l(P^1_i) \atop  \rho(Q,P_i^1) < l(P_i^1)}{\!\!\! \langle \xi^2_Q f, T(b^1_{P^1_i})\rangle [g]_{P^1_i}  } = \sum_{P^1_i} \sum_{P \subset \widehat{P^1_i} \atop l(P) = l(P^1_i)} \sum_{Q  \, \mathrm{dpa} \, 2 \atop Q \subset P}{ \langle \xi^2_Q f, T(b^1_{P^1_i})\rangle [g]_{P^1_i} },\]
which we will do below, when we take care of the term $V_{2,1}$. 

\subsubsection*{\textbf{Estimate of $\langle f, V_3 g \rangle$}}This term is easy to handle. Indeed, applying the coefficient estimate \eqref{eqn:4}, and the fact that the mean of $b^1$ is bounded below uniformly on the pa $1$ cubes $R$, we have
\begin{align*}
|\langle f,V_3 g \rangle | & \lesssim  \sum_{Q   \, \mathrm{dpa} \, 2 } \sum_{R  \, \mathrm{pa} \, 1 \atop l(R) = l(Q) }{|  \langle \xi^2_Q f, T( b^1 1_R)\rangle | \frac{|[g]_R|}{|[b^1]_R|}   }\\
& \lesssim  \sum_{Q   \, \mathrm{dpa} \, 2 } \sum_{R  \, \mathrm{pa} \, 1 \atop l(R) = l(Q) }{\alpha_{Q,R} \, \mu(Q)^{\frac{1}{2}} \, \mu(R)^{\frac{1}{2}}}\\
& \lesssim  \sum_{Q   \, \mathrm{dpa} \, 2 }{ \mu(Q)} \!\!\!\! \sum_{R  \, \mathrm{pa} \, 1 ; l(R) = l(Q) \atop \rho(R,Q) \geq l(R) }{\!\!\!\!\!\!\!\mu(R) \left ( \frac{l(R)}{\rho(Q,R)} \right )^{\alpha} \frac{1}{\mu(Q,R)} } + \!\! \sum_{Q   \, \mathrm{dpa} \, 2 }{ \mu(Q)^{\frac{1}{2}}} \!\!\!\!\!\! \sum_{R  \, \mathrm{pa} \, 1 ; l(R) = l(Q) \atop \rho(R,Q) < l(R) }{\!\!\!\!\!\!\mu(R)^{\frac{1}{2}}}\\
&\lesssim   \sum_{Q   \, \mathrm{dpa} \, 2 }{ \mu(Q)} +  \sum_{Q   \, \mathrm{dpa} \, 2 }{ \mu(Q)} \lesssim \mu(Q_0),\\
\end{align*}
where once more the last line is a consequence of Lemma \ref{Calcul1}, the fact that neighbor cubes have comparable measure, that any dyadic cube $Q$ has a uniformly bounded number of such neighbors, and the $C_X-$packing property of the dpa $2$ cubes \eqref{packing2}.

\subsubsection*{\textbf{Estimate of $\langle f, V_2 g \rangle$}}First, let us examine the part of the sum when the cubes $P_i^1$ and $Q$ are close, that is
\begin{align*} 
\langle f, V_{2,1}g \rangle  & =  \sum_{P^1_i} \sum_{Q  \, \mathrm{spa} \, 2 ; l(Q) \leq l(P^1_i) \atop \rho(Q,P^1_i) < l(P^1_i)}{\langle f, \phi_Q^2 \rangle  \langle b^2 \phi_Q^2 , T(b^1_{P^1_i})\rangle [g]_{P^1_i}} \\
& = \sum_{P^1_i} \!\!\! \sum_{P \subset \widehat{P^1_i} \atop l(P) = l(P^1_i)}\!\!\! \sum_{Q  \, \mathrm{spa} \, 2 \atop Q \subset P}{   \langle f, \phi_Q^2 \rangle \langle b^2 \phi_Q^2 , T(b^1_{P^1_i})\rangle [g]_{P^1_i}}. \\
\end{align*}
We put this sum together with the term left to estimate $\langle f, V_{4,1}g \rangle$ to get 
\[ \Sigma =  \sum_{P^1_i} {[g]_{P^1_i} }\!\!\!  \sum_{P \subset \widehat{P^1_i} \atop l(P) = l(P^1_i)}\!\!\!  \left ( \sum_{Q  \, \mathrm{spa} \, 2 \atop Q \subset P}{ \langle f, \phi_Q^2 \rangle \langle b^2 \phi_Q^2 , T(b^1_{P^1_i})\rangle } +   \sum_{Q  \, \mathrm{dpa} \, 2 \atop Q \subset P}{ \langle \xi^2_Q f, T(b^1_{P^1_i})\rangle  }   \right ) .\]
By definition, we have
\[ \sum_{Q  \, \mathrm{spa} \, 2 \atop Q \subset P}{ \langle f, \phi_Q^2 \rangle \langle b^2 \phi_Q^2 , T(b^1_{P^1_i})\rangle }  + \sum_{Q  \, \mathrm{dpa} \, 2 \atop Q \subset P}{\langle \xi^2_Q f, T(b^1_{P^1_i})\rangle  } = \langle \Pi^{2}(f1_P) , T(b^1_{P^1_i}) \rangle  -  [f]_P \langle b^2 1_P   ,  T(b^1_{P^1_i})  \rangle. \]
We can assume that $P$ is pa $2$ (else the above sum reduces to $0$), so that, using \eqref{accretive3} for $T(b^1_{P^1_i})$ on $P \subset \widehat{P^1_i}$, we get
\[ | \langle b^2 1_P   ,  T(b^1_{P^1_i})  \rangle| \lesssim \|b^2\|_{L^q(P)} \|  T(b^1_{P^1_i}) \|_{L^{q'}(P)} \lesssim \mu(P).   \]
Observe that since the $P^2_j$ contained in $P$ are non overlapping dyadic subcubes of $P$, 
\[ \| \Pi^{2} (f 1_P) \|_{L^q(P)}^q \lesssim \| f \|_{L^q(P)}^q + \sum_{P^2_j \subset P}{\|b^2_{P^2_j}\|_{L^q(P)}^q} \lesssim \mu(P) + \sum_{P^2_j \subset P}{\mu(P^2_j)} \lesssim \mu(P), \]
and thus, applying again \eqref{accretive3} for $T(b^1_{P^1_i})$ on $P$, we have $| \langle \Pi^{2} (f1_P) , T(b^1_{P^1_i}) \rangle |  \lesssim \mu(P)$, giving the expected bound on $\langle f, (V_{2,1} + V_{4,1})g  \rangle$.

\noindent We now estimate the sum in $\langle f,V_2 g \rangle$ running over the pairs of cubes $Q$ and $P_i^1$ which are far away from one another. Applying the coefficient estimate \eqref{eqn:7} and then the Cauchy-Schwarz inequality, one gets
\[ \sum_{P^1_i} \sum_{Q  \, \mathrm{spa} \, 2 ; l(Q) \leq l(P^1_i) \atop \rho(Q,P^1_i) \geq l(P^1_i)}{ | \langle f, \phi_Q^2 \rangle | \, | \langle b^2 \phi_Q^2 , T(b^1_{P^1_i})\rangle [g]_{P^1_i}|}  \lesssim \sum_{P^1_i} \sum_{Q  \, \mathrm{spa} \, 2 ; l(Q) \leq l(P^1_i) \atop \rho(Q,P^1_i) \geq l(P^1_i)}{ | \langle f, \phi_Q^2 \rangle | \, \mu(P^1_i)^{\frac{1}{2}} \, \alpha_{Q,P^1_i}} \]
\begin{align*}
&  \lesssim  \sum_{P^1_i} \!\!\! \sum_{Q  \, \mathrm{spa} \, 2 ; l(Q) \leq l(P^1_i) \atop \rho(Q,P^1_i) \geq l(P^1_i)}{\!\!\!\left \{   | \langle f, \phi_Q^2 \rangle | \frac{\mu(P^1_i)^{1/2}}{\mu(Q,P^1_i)^{1/2}}  \frac{l(Q)^{\alpha /2}}{\rho(Q,P^1_i)^{\alpha /2}}  \right \} \!\!  \left  \{\mu(P^1_i)^{\frac{1}{2}} \frac{\mu(Q)^{1/2}}{\mu(Q,P^1_i)^{1/2}}  \frac{l(Q)^{\alpha /2}}{\rho(Q,P^1_i)^{\alpha /2}}   \right \}}\\
& \lesssim S_1^{1/2} \, . \, S_2^{1/2},
\end{align*}
with
\begin{align*}
S_1 & =   \sum_{Q  \, \mathrm{spa} \, 2}{ | \langle f, \phi_Q^2 \rangle |^2} \!\!\!\!\! \sum_{P^1_i ; l(P^1_i)  \geq l(Q) \atop \rho(P^1_i,Q) \geq l(P^1_i)}{\!\!\frac{\mu(P^1_i)}{\mu(Q,P^1_i)}\left (\frac{l(Q)}{\rho(Q,P^1_i)}\right )^{\alpha}} ,\\
S_2 & =  \sum_{P^1_i}{\mu(P^1_i)} \!\!\!\!\! \sum_{Q  \, \mathrm{spa} \, 2 ; l(Q) \leq l(P^1_i) \atop  \rho(Q,P^1_i) \geq l(P^1_i)}{\!\frac{\mu(Q)}{\mu(Q,P^1_i)}\left (\frac{l(Q)}{\rho(Q,P^1_i)}\right )^{\alpha}} .   \\
\end{align*}
By Lemma \ref{Calcul1} and the $1-$packing property of the $P^1_i$, we have already $S_2 \lesssim \sum_i{\mu(P^1_i)} \lesssim \mu(Q_0).$ For the first sum, we have to work a bit more. The result is obtained applying another summing lemma: 
\begin{lem}\label{Calcul2}
Let $Q$ be a fixed dyadic cube of the space of homogeneous type $X$. Let $p\in \N$. We have the following summing property
\[ \sum_ {R ; l(R)  = \delta^{-p} l(Q) \atop \rho(R,Q) \geq l(R)}{\!\!\mu(R)\left (\frac{l(Q)}{\rho(Q,R)}\right )^{\alpha}\frac{1}{\mu(Q,R)}} \lesssim \delta^{p\alpha}, \]
which, by summing over $p \in \N$, immediately yields
\[  \sum_ {R ; l(R)  \geq l(Q) \atop \rho(R,Q) \geq l(R)}{\!\!\mu(R)\left (\frac{l(Q)}{\rho(Q,R)}\right )^{\alpha}\frac{1}{\mu(Q,R)}} \lesssim 1 .\]
\end{lem}

\begin{proof}
Observe that since $\rho(Q,R) \geq l(R) \geq l(Q)$, we have $\mu(Q,R) \approx \mu(B(z_Q,\rho(z_Q,y)))$ for all $y \in R$, with $z_Q$ denoting the center of the cube $Q$. In the same way, $\rho(Q,R) \approx \rho(z_Q,y)$ for all $y\in R$. Therefore, we can write
\begin{align*}
 \sum_ {R ; l(R)  = \delta^{-p} l(Q) \atop \rho(R,Q) \geq l(R)}{\!\!\!\left (\frac{l(Q)}{\rho(Q,R)}\right )^{\alpha}\!\!\!\frac{\mu(R)}{\mu(Q,R)}}  & \lesssim \sum_{R ; l(R)  = \delta^{-p} l(Q) \atop \rho(R,Q) \geq l(R)}{\!\! \int_R{\left (\frac{l(Q)}{\rho(z_Q,y)}\right )^{\alpha} \frac{d\mu(y)}{\mu(B(z_Q,\rho(z_Q,y)))} }}\\
 & \lesssim \int_{\rho(z_Q,y) > l(R)}{\left (\frac{l(Q)}{\rho(z_Q,y)}\right )^{\alpha} \frac{d\mu(y)}{\mu(B(z_Q,\rho(z_Q,y)))}}\\
 & \lesssim  \sum_{m > 0} \int_{\rho(z_Q,y) \sim \delta^{-m}l(R) }{\left (\frac{l(Q)}{\delta^{-m}l(R)}\right )^{\alpha} \frac{d\mu(y)}{\mu(B(z_Q,\delta^{-m}l(R)))}}\\
 & \lesssim \left ( \frac{l(Q)}{l(R)} \right )^{\alpha} \sum_{m  > 0}{\delta^{m\alpha} \frac{\mu(B(z_Q,\delta^{-m-1}l(R)))}{\mu(B(z_Q,\delta^{-m}l(R)))}} \lesssim \delta^{p\alpha}.\\
\end{align*}
The second inequality comes from the fact that the cubes $R$ are non overlapping as they are all of the same generation. Then we obtained the third inequality by splitting the integral over the coronae $\rho(z_Q,y) \sim \delta^{-m}l(R)$, and the last line follows easily.
\end{proof}
\noindent Going back to our argument, Lemma \ref{Calcul2} ensures that $S_1 \lesssim \sum_{Q  \, \mathrm{spa} \, 2}{ | \langle f, \phi_Q^2 \rangle |^2} \lesssim \mu(Q_0),$ and it thus concludes the estimation of the term $\langle f, V_2 g \rangle$.
\subsubsection*{\textbf{Estimate of $ \langle f, V_1 g \rangle$}}This is the most difficult term. Recall that
\[ \langle f, V_1 g \rangle =   {\sum_{Q  \, \mathrm{spa} \, 2 ; R  \, \mathrm{pa} \, 1 \atop l(R) = l(Q) }{  \langle f, \phi_Q^2 \rangle   \langle b^2 \phi_Q^2 , T(b^1 1_R)\rangle \frac{[g]_R}{[b^1]_R} }} .\]
We split it into three parts: write
\[ \langle f, V_1 g \rangle = \langle f, V_{1,1} g \rangle + \langle f, V_{1,2} g \rangle + \langle f, V_{1,3} g \rangle, \]
with
\begin{align*}
\langle f, V_{1,1} g \rangle & = \sum_{Q  \, \mathrm{spa} \, 2 \, \mathrm{and} \, \mathrm{non} \, \mathrm{pa} \, 1 \atop{ R  \, \mathrm{pa} \, 1 \atop l(R) = l(Q) }}{  \langle f, \phi_Q^2 \rangle \langle b^2 \phi_Q^2 , T(b^1 1_R)\rangle \frac{[g]_R}{[b^1]_R} },\\
\langle f, V_{1,2} g \rangle & = \sum_{Q  \, \mathrm{spa} \, 2 \, \mathrm{and} \, \mathrm{pa} \, 1  }{  \langle f, \phi_Q^2 \rangle \langle b^2 \phi_Q^2 , T(b^1 1_Q)\rangle \frac{[g]_Q}{[b^1]_Q} },\\
\langle f, V_{1,3} g \rangle & =  \!\!\!\! \sum_{Q  \, \mathrm{spa} \, 2 \atop \mathrm{and} \, \mathrm{pa} \, 1 } \!\! \sum_{ R  \, \mathrm{pa} \, 1 ; R \neq Q \atop l(R) = l(Q) }{\!\left ( \!\! \langle f, \phi_Q^2 \rangle  \langle b^2 \phi_Q^2 , T(b^1 1_R)\rangle \frac{[g]_R}{[b^1]_R} + \!  \langle f, \phi_Q^2 \rangle \langle b^2 \phi_Q^2 , T  (b^1 (1_Q \! - \!\!\!\!\!\!\!\! \sum_{R'  \, \mathrm{pa} \, 1 \atop l(R') = l(Q) }{\!\!\!\!\!1_{R'}} )  ) \rangle \frac{[g]_Q}{[b^1]_Q} \!\! \right ) }.
\end{align*}
\paragraph{\textbf{Estimate of $\langle f, V_{1,1} g \rangle$}}Remark first that $Q$ non pa $1$ means $Q \subset P^1_i$ for some $i$. Hence,
\[ \langle f, V_{1,1} g \rangle = \sum_{P^1_i} \sum_{Q   \, \mathrm{spa} \, 2 \atop Q \subset P^1_i} \sum_{R \, \mathrm{pa} \, 1 \atop l(R) = l(Q)} {  \langle f, \phi_Q^2 \rangle \langle b^2 \phi_Q^2 , T(b^1 1_R)\rangle \frac{[g]_R}{[b^1]_R}}. \]
Now, applying the coefficient estimate \eqref{eqn:5} and the fact that the mean of $b^1$ is uniformly bounded below on the pa $1$ cubes, we obtain
\[ | \langle f, V_{1,1} g \rangle | \lesssim   \sum_{P^1_i} \sum_{Q   \, \mathrm{spa} \, 2 \atop Q \subset P^1_i} \sum_{R \, \mathrm{pa} \, 1 \atop l(R) = l(Q)}{\alpha_{Q,R} \,  | \langle f, \phi_Q^2 \rangle | \, \mu(R)^{\frac{1}{2}} } . \]
Fix $P^1_i$, and observe that the cubes $R$ involved in this sum are necessarily outside of $P^1_i$. Split the sum for the cubes $R$ that are contained in $\widehat{P^1_i}$, and those that are outside of $\widehat{P^1_i}$. We first show that
\[  \mathrm{I}  =   \sum_{Q   \, \mathrm{spa} \, 2 \atop Q \subset P^1_i} \sum_{R \, \mathrm{pa} \, 1 ; l(R) = l(Q) \atop R \cap \widehat{P^1_i} = \varnothing }{\alpha_{Q,R} \,  | \langle f, \phi_Q^2 \rangle | \, \mu(R)^{\frac{1}{2}} }  \lesssim \mu(P^1_i). \]
For such pairs of cubes $(Q,R)$, as $Q \subset P^1_i$ we necessarily have $\rho(Q,R) \geq \rho(P^1_i,R) \geq l(P^1_i)$, and $\mu(Q,R) \gtrsim \mu(P^1_i,R),$ so that, applying the Cauchy-Schwarz inequality, one gets
\begin{align*}
\mathrm{I} & \lesssim  \sum_{Q   \, \mathrm{spa} \, 2 \atop Q \subset P^1_i} \sum_{R \, \mathrm{pa} \, 1 ; l(R) = l(Q) \atop \rho(R, P^1_i) \geq l(P^1_i)} { | \langle f, \phi_Q^2 \rangle | \, \mu(R)^{\frac{1}{2}}    \frac{l(R)^{\alpha}}{\rho(P^1_i,R)^{\alpha}}\frac{ \mu(Q)^{\frac{1}{2}} \mu(R)^{\frac{1}{2}}}{\mu(P^1_i,R)}  }  \\
& \lesssim \!\!  \left (\!\! \sum_{Q   \, \mathrm{spa} \, 2 \atop Q \subset P^1_i}\!\!\!\!\! \sum_{R \, \mathrm{pa} \, 1 \atop{ l(R) = l(Q) \atop \rho(R, P^1_i) \geq l(P^1_i)}}{\!\!\!\!\!\!\!\!\!\!\!\! | \langle f, \phi_Q^2 \rangle |^2   \frac{l(R)^{\alpha}}{\rho(P^1_i,R)^{\alpha}} \frac{\mu(R)}{\mu(P^1_i,R)} }\!\! \right )^{\!\!\!1/2} \!\!\!  \left ( \!\!  \sum_{Q   \, \mathrm{spa} \, 2 \atop Q \subset P^1_i} \!\!\!\!\!\sum_{R \, \mathrm{pa} \, 1 \atop{ l(R) = l(Q) \atop \rho(R, P^1_i) \geq l(P^1_i)}}{\!\!\!\!\!\!\!\!\!\!\!\!\mu(Q)   \frac{l(R)^{\alpha}}{\rho(P^1_i,R)^{\alpha}} \frac{\mu(R)}{\mu(P^1_i,R)} }\!\!  \right ) ^{\!\!\!1/2}\\
\end{align*}
By Lemma \ref{Calcul1}, we get
\[  \sum_{R \, \mathrm{pa} \, 1 ; l(R) = l(Q) \atop \rho(R, P^1_i) \geq l(P^1_i)}{  \left ( {\frac{l(R)}{\rho(P^1_i,R)}} \right )^{\alpha} \frac{\mu(R)}{\mu(P^1_i,R)} } \lesssim \left ( \frac{l(Q)}{l(P^1_i)} \right )^{\alpha} \lesssim 1  .\]
Therefore,
\begin{align*}
\mathrm{I} & \lesssim \left (\sum_{Q   \, \mathrm{spa} \, 2 \atop Q \subset P^1_i}{ | \langle f, \phi_Q^2 \rangle |^2 } \right )^{1/2}   \left ( \sum_{Q   \, \mathrm{spa} \, 2 \atop Q \subset P^1_i}{\mu(Q) \left ( \frac{l(Q)}{l(P^1_i)} \right )^{\alpha}  }   \right )^{1/2}\\
& \lesssim \|f\|_{L^2(P^1_i)} \left ( \sum_{m \geq 0} \delta^{m\alpha} \sum_{Q \subset P^1_i \atop l(Q)=\delta^m l(P^1_i)}{\mu(Q)} \right )^{1/2} \lesssim \mu(P^1_i).
\end{align*}
We are left to show that
\[ \mathrm{II} =  \sum_{Q   \, \mathrm{spa} \, 2 \atop Q \subset P^1_i} \sum_{R \, \mathrm{pa} \, 1 ; l(R) = l(Q) \atop R \subset \widehat{P^1_i}  }{\alpha_{Q,R} \,  | \langle f, \phi_Q^2 \rangle | \, \mu(R)^{\frac{1}{2}} }  \lesssim \mu(P^1_i).\]
We again split this sum for the cubes $Q$ that are at a distance less than $l(Q)$ to the complement of $P^1_i$, and for those that are not. In the first case, for a fixed such cube $Q$, there are the cubes $R$ which are neighbors of $Q$ - they come in a uniformly bounded number and they have a measure comparable to that of $Q$ - and there are the cubes $R$ at a distance greater than $l(Q)$ of Q. Write
\begin{align*}
\mathrm{II}_1 & = \sum_{Q   \, \mathrm{spa} \, 2 ; Q \subset P^1_i \atop \rho(Q, (P^1_i)^c) < l(Q)} \sum_{R \, \mathrm{pa} \, 1 ; l(R) = l(Q) \atop R \subset \widehat{P^1_i}  }{\alpha_{Q,R} \,  | \langle f, \phi_Q^2 \rangle | \, \mu(R)^{\frac{1}{2}} } \\
& \lesssim \sum_{Q   \, \mathrm{spa} \, 2 ; Q \subset P^1_i \atop \rho(Q, (P^1_i)^c) < l(Q)}{\!\!\!\!\!\!\!\!\! | \langle f, \phi_Q^2 \rangle |} \left \{ \mu(Q)^{\frac{1}{2}} +\!\!\!\!\!\!\! \sum_{R \, \mathrm{pa} \, 1 ; l(R) = l(Q) \atop R \subset \widehat{P^1_i} ; \rho(R,Q) \geq l(Q)  }{ \!\!\!\!\!\!\!\mu(Q)^{\frac{1}{2}} \frac{\mu(R)}{\mu(Q,R)}\left ( \frac{l(R)}{\rho(Q,R)}\right )^{\alpha} } \right \} \\
& \lesssim  \sum_{Q   \, \mathrm{spa} \, 2  ; Q \subset P^1_i \atop \rho(Q, (P^1_i)^c) < l(Q)}{\!\!\!\!\!\!\!\!\!\!\! | \langle f, \phi_Q^2 \rangle | \mu(Q)^{\frac{1}{2}}} \lesssim  \left \{ \sum_{Q\, \mathrm{spa} \, 2 ; Q \subset P^1_i  \atop \rho(Q,(P^1_i)^c) < l(Q)}{\!\!\!\!\!\!\!\!\!\!\!\mu(Q)}  \right \}^{1/2}  \left \{ \sum_{Q\, \mathrm{spa} \, 2 ; Q \subset P^1_i  \atop \rho(Q,(P^1_i)^c) < l(Q)}{\!\!\!\!\!\!\!\!\!\!{ | \langle f, \phi_Q^2 \rangle |}^2}  \right \} ^{1/2},\\
\end{align*}
applying Lemma \ref{Calcul1} to get the second inequality and then the Cauchy-Schwarz inequality to get the last one. By Lemma \ref{Mart}, $\sum_{Q \subset P^1_i}{{ | \langle f, \phi_Q^2 \rangle |}^2} \lesssim \| f \|^2_{L^2(P^1_i)} \lesssim \mu(P^1_i). $ Furthermore, 
\begin{align*}
\sum_{Q \subset P^1_i  \atop \rho(Q,(P^1_i)^c) < l(Q)}{\!\!\!\!\!\!\!\mu(Q)} & = \sum_{m\geq0} \sum_{Q \subset P^1_i ; l(Q)=\delta^m l(P^1_i) \atop \rho(Q,(P^1_i)^c) < \delta^m l(P^1_i)}{\!\!\!\!\!\!\!\mu(Q)}\\
& \lesssim \sum_{m \geq 0}{\mu \left ( \{ x \in P^1_i: \rho(x,(P^1_i)^c) \lesssim \delta^m l(P^1_i) \} \right ) }\\
& \lesssim \sum_{m \geq 0}{\delta^{m \eta} \mu(P^1_i)} \lesssim \mu(P^1_i),
\end{align*}
by the small boundary condition $(6)$ of Lemma \ref{cubes}, so that, as desired, we have $\mathrm{II}_1 \lesssim \mu(P^1_i).$ The sum left to estimate is
\[ \mathrm{II}_2 =  \sum_{Q   \, \mathrm{spa} \, 2  ; Q \subset P^1_i \atop \rho(Q, (P^1_i)^c) \geq l(Q)}  \sum_{R \, \mathrm{pa} \, 1 ; l(R) = l(Q) \atop R \subset \widehat{P^1_i} } { | \langle f, \phi_Q^2 \rangle | \mu(Q)^{\frac{1}{2}} \frac{\mu(R)}{\mu(Q,R)} \left ( \frac{l(Q)}{\rho(Q,R)} \right )^{\alpha}}.  \]
For such couples of cubes $(Q,R)$, we have necessarily $\rho(Q,R) \geq \rho(Q,(P^1_i)^c) \geq l(Q)$ as $R$ and $P^1_i$ are disjoint. For a fixed cube $Q$, we thus see that
\begin{align*}  
\sum_{R \subset \widehat{P^1_i} ; l(R) = l(Q) \atop \rho(R,Q) \geq   l(Q)} { \frac{\mu(R)}{\mu(Q,R)} \left ( \frac{l(Q)}{\rho(Q,R)} \right )^{\alpha}} & \lesssim \left ( \frac{l(Q)}{\rho(Q,(P^1_i)^c)} \right )^{\alpha/2} \!\!\!\!\!\!\!\!\!\!\!\sum_{R \subset \widehat{P^1_i} ; l(R) = l(Q) \atop \rho(R,Q) \geq   l(Q)} { \frac{\mu(R)}{\mu(Q,R)} \left ( \frac{l(R)}{\rho(Q,R)} \right )^{\alpha/2}}\\
& \lesssim  \left ( \frac{l(Q)}{\rho(Q,(P^1_i)^c)} \right )^{\alpha/2},
\end{align*}
by Lemma \ref{Calcul1}. Applying the Cauchy-Schwarz inequality, we thus have
\begin{align*}
\mathrm{II}_2 & \lesssim  \sum_{Q   \, \mathrm{spa} \, 2  ; Q \subset P^1_i \atop \rho(Q, (P^1_i)^c) \geq l(Q)}{ | \langle f, \phi_Q^2 \rangle | \mu(Q)^{\frac{1}{2}}  \left ( \frac{l(Q)}{\rho(Q,(P^1_i)^c)} \right )^{\frac{\alpha}{2}} }\\
& \lesssim  \left \{ \sum_{Q   \, \mathrm{spa} \, 2  ; Q \subset P^1_i \atop \rho(Q, (P^1_i)^c) \geq l(Q)}{{ | \langle f, \phi_Q^2 \rangle |}^2} \right \}^{1/2} \left \{ \sum_{Q   \, \mathrm{spa} \, 2  ; Q \subset P^1_i \atop \rho(Q, (P^1_i)^c) \geq l(Q)}{ \left ( \frac{l(Q)}{\rho(Q,(P^1_i)^c)} \right )^{ \alpha} \mu(Q) } \right \}^{1/2}\\
\end{align*}
Lemma \ref{Mart} ensures that the first term is bounded by $\mu(P^1_i)^{1/2}$. For the second term, observe that
\[   \sum_{Q \subset P^1_i \atop \rho(Q, (P^1_i)^c) \geq l(Q)}{ \left ( \frac{l(Q)}{\rho(Q,(P^1_i)^c)} \right )^{ \alpha} \mu(Q) } = \sum_{m \in \N ; j \in \Z \atop m\geq j > j_0} {\delta^{\alpha(m-j)}} \sum_{Q \in E_{m,j} }{\mu(Q)}, \]
where, for $m \in \N, j \in \Z$, with $m \geq j  > j_0$,
\[ E_{m,j} = \{ Q \subset P^1_i: \frac{l(Q)}{l(P^1_i)} = \delta^m \, \mathrm{and} \,  \delta^j l(P^1_i)  \leq \rho(Q,(P^1_i)^c) < \delta^{j-1}l(P^1_i) \}, \]
and $j_0$ is such that $\delta^{j_0} \geq C_1 > \delta^{j_0 +1}$, with $C_1$ the constant intervening in Lemma \ref{cubes}. Indeed, for the cubes $Q$ intervening in the above sum, we always have $\rho(Q,(P^1_i)^c) < \diam(P^1_i) \leq C_1 l(P^1_i), $ which forces $j > j_0$. Observe that if $x \in Q \subset E_{m,j}$, then $\rho(x,(P^1_i)^c) \leq \rho(Q,(P^1_i)^c) + l(Q) \leq \delta^{j-1}l(P^1_i) + \delta^m l(P^1_i) \lesssim \delta^j l(P^1_i)$ since $m\geq j$. Thus the cubes of $E_{m,j}$ pave a corona $\{ x \in P^1_i: \rho(x,(P^1_i)^c) \lesssim \delta^j l(P^1_i) \}$. But by the small boundary condition $(6)$ of Lemma \ref{cubes}, this corona has its measure bounded by $\delta^{j \eta} \mu(P^1_i)$. Therefore,
\[   \sum_{Q \subset P^1_i \atop \rho(Q, (P^1_i)^c) \geq l(Q)}{ \!\!\!\!\!\!\!\!\!\!\!\!\!\left ( \frac{l(Q)}{\rho(Q,(P^1_i)^c)} \right )^{ \alpha}\!\! \mu(Q) } \lesssim \sum_{m \in \N ; j \in \Z \atop m\geq j > j_0} {\delta^{\alpha m}\delta^{(\eta - \alpha)j} \mu(P^1_i)} \lesssim \sum_{j > j_0}{\delta^{(\eta - \alpha)j} \delta^{ \alpha j}\mu(P^1_i)} \lesssim \mu(P^1_i),  \]
which gives us the bound $\mathrm{II}_2 \lesssim \mu(P^1_i).$ Summing then all these terms in $i$, and applying the $1-$packing property of the $P^1_i$, we obtain the expected bound $|\langle f, V_{1,1} g \rangle| \lesssim \mu(Q_0)$.

\medskip

\paragraph{\textbf{Estimate of $\langle f, V_{1,2} g \rangle$}}This is where the paraproduct appears, and unfortunately, a non trivial error term. Write
\begin{align*}
\langle f, V_{1,2} g \rangle & =  \!\!\!\!\!\! \sum_{Q  \, \mathrm{spa} \, 2 \, \mathrm{and} \, \mathrm{pa} \, 1}{\!\!\!\!\!\!\!   \langle f, \phi_Q^2 \rangle  \langle b^2 \phi_Q^2 , T(b^1)\rangle \frac{[g]_Q}{[b^1]_Q} } - \!\!\!\!\!\!\!  \sum_{Q  \, \mathrm{spa} \, 2 \, \mathrm{and} \, \mathrm{pa} \, 1}{\!\!\!\!\!\!\!   \langle f, \phi_Q^2 \rangle\langle b^2 \phi_Q^2 , T(b^1 1_{\bigcup_{l(P^1_i) \geq l(Q)}{P^1_i}})\rangle \frac{[g]_Q}{[b^1]_Q} } \\
& = \mathrm{I} + \mathrm{II}. \\
\end{align*}
Observe that for the pa $1$ cubes $Q$, the set of coefficients formed by $\left ( \frac{[g]_Q}{[b^1]_Q} \right )_Q$ is uniformly bounded. It is easy to see that the hypotheses of Lemma \ref{Mart2} are verified (with $b$ = $b^2$). Indeed, as a consequence of \eqref{Hofmann3}, we have $C_{\{P_j^2\}}(b^2) \lesssim 1$. Moreover, on $F = Q_0 \backslash \cup{P^2_j}$, it is clear that $b^2$ is uniformly bounded as a consequence of Lebesgue differentiation theorem. We can thus apply Lemma \ref{Mart2} to the term $\mathrm{I}$, with $\nu = q$, and it gives us the bound
\[ |\mathrm{I}| \lesssim \|f\|_{L^q(Q_0)}\|T(b^1)\|_{L^{q'}(Q_0)} \lesssim \mu(Q_0). \]
Next, we can rewrite the second term involving a sum over the cubes $P^1_i$. We split it into two parts, depending as usual on whether or not the cubes $Q$ and $P^1_i$ are at a distance less than $l(P^1_i)$. Applying the coefficient estimate \eqref{eqn:8}, one gets
\[ \left | \sum_{P^1_i} \sum_{Q  \, \mathrm{spa} \, 2 \, \mathrm{and} \, \mathrm{pa} \, 1 \atop {l(Q) \leq l(P^1_i) ; \rho(Q,P^1_i) \geq l(P^1_i)}}{\!\!\!\!\!\!\!\!\!\!\!\!\!\!\!\!\!\!  \langle f, \phi_Q^2 \rangle  \langle b^2 \phi_Q^2 , T(b^1 1_{P^1_i }) \rangle \frac{[g]_Q}{[b^1]_Q} } \right |  \lesssim  \sum_{P^1_i} \sum_{Q  \, \mathrm{spa} \, 2 \, \mathrm{and} \, \mathrm{pa} \, 1 \atop {l(Q) \leq l(P^1_i) ; \rho(Q,P^1_i) \geq l(P^1_i)}}{\!\!\!\!\!\!\!\!\!\!\!\!\!  | \langle f, \phi_Q^2 \rangle | \, \mu(P^1_i)^{\frac{1}{2}} \, \alpha_{Q,P^1_i} }.  \]
Remark that we already estimated this sum when we were taking care of the term $\langle f, V_2 g \rangle$. We have seen that it is bounded by $\mu(Q_0)$. For the remaining sum, observe that since a pa $1$ cube $Q$ is disjoint with any $P^1_i$, such a cube which is at a distance less than $l(P^1_i)$ to $P^1_i$ is necessarily contained inside one of the neighbors of $P^1_i$. We thus have left to estimate the term
\[ \sum_{P^1_i} \sum_{P \subset \widehat{P^1_i}\backslash P^1_i \atop l(P) = l(P^1_i) }\sum_{Q  \, \mathrm{spa} \, 2 \, \mathrm{and} \, \mathrm{pa} \, 1 \atop Q \subset P}{\!\!\!\!\!\langle f, \phi_Q^2 \rangle \langle b^2 \phi_Q^2 , T(b^1 1_{P^1_i }) \rangle \frac{[g]_Q}{[b^1]_Q} }. \] 
Fix $P^1_i$ and $P$. Set $F= P \backslash \cup{P^2_j}$, and decomposing 
\[ b^2 = b^2 1_F + \sum_{P^2_j  \subset P: \rho(P^2_j, P^1_i) \geq l(P^2_j)} {b^2 1_{P^2_j}} + \sum_{P^2_j \subset P: \rho(P^2_j, P^1_i) < l(P^2_j)} {b^2 1_{P^2_j}},\] 
write
\[ \sum_{Q  \, \mathrm{spa} \, 2 \, \mathrm{and} \, \mathrm{pa} \, 1 \atop Q \subset P}{\!\!\!\!\langle f, \phi_Q^2 \rangle \langle b^2 \phi_Q^2 , T(b^1 1_{P^1_i }) \rangle \frac{[g]_Q}{[b^1]_Q} }= \Sigma_1 + \Sigma_2 + \Sigma_3.\]
As in Lemma \ref{Mart2}, denote by $L$ the operator appearing in these sums. For the first term, apply the boundedness of $L$ given by Lemma \ref{Mart2} to get
\begin{align*}
| \Sigma_1 |  & = \left | \langle Lf, b^2 1_F T(b^1 1_{P^1_i }) \rangle  \right |  \lesssim \| f \|_{L^{p'}(P)} \|b^2 1_F T(b^1 1_{P^1_i })\|_{L^p(P)}\\
& \lesssim \mu(P)^{1/{p'}} \| b^1  \|_{L^p(P^1_i)} \lesssim \mu(P^1_i),
\end{align*}
where the second inequality comes from the Hardy inequality \eqref{Hardy2}, and the fact that $b^2 1_F \in L^{\infty}(P)$, and the last inequality from the fact that by \eqref{Hofmann3}, $b^1 1_{P^1_i} \in L^p(P^1_i)$, and $P,P^1_i$ have comparable measure. For the second term $\Sigma_2$, remark that for any fixed $j$, as $P^2_j$ is strictly contained in any spa $2$ cube, $Lf$ is constant on $P^2_j$. Thus we can write
\begin{align*} 
| \Sigma_2 | & = \left | \sum_{P^2_j \subset P \atop \rho(P^2_j, P^1_i) \geq l(P^2_j)}{[Lf]_{P^2_j} \langle b^2 1_{P^2_j}, T(b^1 1_{P^1_i}) \rangle} \right | \\
& \lesssim \sum_{P^2_j \subset P \atop  \rho(P^2_j, P^1_i) \geq l(P^2_j)} {\int_{P^1_i} \int_{P^2_j} {\frac{|[Lf]_{P^2_j} b^2(x) b^1(y)| }{\lambda(x,y)}d\mu(x) d\mu(y)}}.
\end{align*}
For a fixed $j$ such that $\rho(P^2_j, P^1_i) \geq l(P^2_j)$, as $l(P^1_i) \geq l(P^2_j)$, observe that for a fixed $y \in P^1_i$, $\sup_{x \in P^2_j}{\lambda(x,y)} \approx \inf_{x \in P^2_j}{\lambda(x,y)}$ uniformly in $y$ and $j$. Therefore, we can integrate in $x$ to get
\[\int_{P^2_j}{|[Lf]_{P^2_j}| |b^2(x)| d\mu(x)} \leq |[Lf]_{P^2_j}| \mu(P^2_j) \leq \int_{P^2_j}{|Lf(x)|d\mu(x)},\]
and we obtain
\begin{align*}
| \Sigma_2 | & \lesssim \sum_{P^2_j \subset P} {\int_{P^2_j} \int_{P^1_i} {\frac{|[Lf(x)| | b^1(y)| }{\lambda(x,y)}d\mu(y) d\mu(x)}} \lesssim \int_{P} \int_{P^1_i} {\frac{|[Lf(x)| | b^1(y)| }{\lambda(x,y)}d\mu(y) d\mu(x)}\\
& \lesssim \| Lf \|_{L^{p'}(P)} \|b^1\|_{L^{p}(P^1_i)} \lesssim \mu(P^1_i),
\end{align*}
where the last line is obtained applying the Hardy inequality \eqref{2etoile}, the boundedness of $L$, and \eqref{Hofmann} for $b^1$ on $P^1_i$. Only the third term $\Sigma_3$ remains. This is the term for which we will need to use the stronger weak boundedness property \eqref{WBP1} (and it is the only time in the proof we use it!). We do not see how to estimate this term without this extra property. As for $\Sigma_2$, write
\[ \Sigma_3 = \!\!\!\!\!\! \sum_{P^2_j \subset P \atop \rho(P^2_j, P^1_i) < l(P^2_j)}{\!\!\!\!\![Lf]_{P^2_j} \langle b^2 1_{P^2_j}, T(b^1 1_{P^1_i}) \rangle} = \left \langle b^2 \left (  \sum_{P^2_j \subset P \atop \rho(P^2_j, P^1_i) < l(P^2_j)} {[Lf]_{P^2_j}1_{P^2_j}}  \right ), T(b^1 1_{P^1_i}) \right \rangle. \]
Apply \eqref{WBP1} to get 
\begin{align*}
| \Sigma_3 | & \lesssim  \left \|  \sum_{P^2_j \subset P} {[Lf]_{P^2_j}1_{P^2_j}}  \right \|_{\nu} \mu(P^1_i)^{\frac{1}{\nu'}} \lesssim \left ( \sum_{P^2_j \subset P}{\mu(P^2_j) [|Lf|^{\nu}]_{P^2_j}   }   \right ) ^{1/{\nu}}  \mu(P^1_i)^{\frac{1}{\nu'}}\\
& \lesssim \| Lf \|_{L^{\nu}(P)} \ \  \mu(P^1_i)^{\frac{1}{\nu'}}  \lesssim \mu(P^1_i),
\end{align*}
where once again we have used the boundedness of $L$ given by Lemma \ref{Mart2}. It remains only to sum over the neighbors $P$ of $P^1_i$, which come in a uniformly bounded number, and then over the $P^1_i$ (recall that they realize a $1$-packing of $Q_0$), to obtain as desired
\[ | \mathrm{II} |  \lesssim \sum_i{\mu(P^1_i)} \lesssim \mu(Q_0).\]

\paragraph{\textbf{Estimate of $\langle f, V_{1,3} g \rangle$}}This is the only term left to estimate to complete the proof. For $Q$ spa $2$ and pa $1$, and $R$ pa $1$ of the same generation, set
\[ \beta_{Q,R} = \begin{cases}
\langle f, \phi_Q^2 \rangle \langle b^2 \phi^2_Q , T(b^1 1_R) \rangle \quad & \mathrm{if} \ \ Q \neq R \\
 \langle f, \phi_Q^2 \rangle \langle b^2 \phi^2_Q , T(b^1 (1_Q - \sum_{R' \, \mathrm{pa} \, 1 \atop l(R') = l(Q)}{1_{R'}})) \rangle \quad &  \mathrm{if} \ \ Q = R,
 \end{cases}\\ \]
so that the definition of $V_{1,3}$ rewrites
\[ \langle f, V_{1,3} g \rangle = \sum_{Q  \, \mathrm{spa} \, 2 \atop  \mathrm{and} \, \mathrm{pa} \, 1 } \sum_{ R  \, \mathrm{pa} \, 1  \atop l(R) = l(Q)}{\beta_{Q,R} \ \ \frac{[g]_R}{[b^1]_R}}  .\]
Observe that for any fixed spa $2$ and pa $1$ cube $Q$, 
\[ \sum_{R  \, \mathrm{pa} \, 1  \atop l(R) = l(Q)}{\beta_{Q,R}} = 0.\] 
Recall that for any pa $1$ cube $R$, $[g]_R = [\Pi^1 g]_R$ by \eqref{Projec2} of Lemma \ref{Projec}, allowing us to replace $[g]_R$ by  $[\Pi^1 g]_R$ in the above sum. Recall also that
\[ \Pi^1 g = E_0^1 g + \sum_{l \geq 0}{D^1_l g}. \]
Fix a pa $1$ cube $R$. As $b^1 \Delta^1_S g $ and $\xi^1_S g$ have mean $0$, it is clear that $[D^1_l g]_R = 0$ as soon as $ l \leq l(R)$, which implies $[\Pi^1 g]_R = [E_0^1 g]_R + \sum_{l > l(R)}{[D^1_l g]_R}.$ Note that $[E_0^1 g]_R = [g]_{Q_0}[b^1]_R$. Furthermore, if $l > l(R)$, there is a unique cube $S_R$ of generation $l$ such that $R \subset S_R$. If $S_R$ is spa $1$, then $[D^1_l g]_R = [b^1 \Delta^1_{S_R} g ]_R = [b^1 ]_R \langle g, \phi^1_{S_R} \rangle [\phi^1_{S_R} ]_R$, since $\Delta^1_{S_R} g$ is constant on the children of $S_R$ and therefore on $R$ as well. Else, $S_R$ is dpa $1$ ($R$ being pa $1$ it cannot be contained in a $P^1_i$), and then $[D^1_l g]_R = [\xi^1_{S_R} g]_R$. Consequently, we can write
\begin{align*}
\langle f, V_{1,3} g \rangle & = \sum_{Q  \, \mathrm{spa} \, 2 \atop  \mathrm{and} \, \mathrm{pa} \, 1 } \sum_{ R  \, \mathrm{pa} \, 1  \atop l(R) = l(Q)}{\beta_{Q,R} \ \ \frac{[\Pi^1 g]_R}{[b^1]_R}} \\
& = \sum_{Q  \, \mathrm{spa} \, 2 \atop  \mathrm{and} \, \mathrm{pa} \, 1 } \sum_{ R  \, \mathrm{pa} \, 1  \atop l(R) = l(Q)}{\beta_{Q,R}  \frac{1}{[b^1]_R} \left \{  \sum_{S  \, \mathrm{spa} \, 1 ; l(S) > l(Q) \atop S \supset R, S \neq R}{[b^1 ]_R\langle g, \phi^1_S \rangle [\phi^1_S]_R}    \right \} }\\
& +  \sum_{Q  \, \mathrm{spa} \, 2 \atop  \mathrm{and} \, \mathrm{pa} \, 1 } \sum_{ R  \, \mathrm{pa} \, 1  \atop l(R) = l(Q)}{\beta_{Q,R}  \frac{1}{[b^1]_R} \left \{  \sum_{S  \, \mathrm{dpa} \, 1 ; l(S) > l(Q) \atop S \supset R, S \neq R}{[\xi^1_{S} g ]_R}    \right \} } 
+  \sum_{Q  \, \mathrm{spa} \, 2 \atop  \mathrm{and} \, \mathrm{pa} \, 1 } {[g]_{Q_0 } }\sum_{ R  \, \mathrm{pa} \, 1  \atop l(R) = l(Q)}{\beta_{Q,R}  }.\\
\end{align*}
The last sum is equal to zero because of the summing property of the $\beta_{Q,R}$ we stated above. We have thus decomposed our term into two parts: $\langle f, V_{1,3} g \rangle = \langle f, V_{1,3,1} g \rangle + \langle f, V_{1,3,2} g \rangle.$

\medskip
\subparagraph{\textit{Estimate of $\langle f, V_{1,3,1} g \rangle$}}Inverting the sums over $R$ and $S$, one gets
\[ \langle f, V_{1,3,1} g \rangle =   \sum_{Q  \, \mathrm{spa} \, 2 \atop  \mathrm{and} \, \mathrm{pa} \, 1 } \sum_{ S  \, \mathrm{spa} \, 1  \atop l(S) > l(Q)} \sum_{R  \, \mathrm{pa} \, 1 ; l(R) = l(Q) \atop R \subset S, R \neq S}{\beta_{Q,R} \langle g, \phi^1_S \rangle [\phi^1_S]_R  }. \]
Next, we want to use the pseudo-annular decomposition used in \cite{AY}. To achieve this, let us introduce some more notation. For $N \in \N$, $Q$ spa $2$, $S$ spa $1$ with $l(S) > l(Q)$, and $R$ pa $1$ with $l(R) = l(Q)$ and $R \subset S, R \neq S$, set
\[  \beta^N_{Q,S,R} = \begin{cases}
\beta_{Q,R} \quad & \mathrm{if} \ \ \rho(Q,R) \sim \delta^{-N+1}l(R) \\
 - \sum_{R \, \mathrm{pa} \, 1 ; l(R) = l(Q) \atop {R \subset S ; R \neq S \atop \rho(R,Q) \sim \delta^{-N+1}l(R)}}{\beta_{Q,R}} \quad & \mathrm{if} \ \ Q = R \\
 0 \quad & \mathrm{else.}
 \end{cases}\\ \]
 where as before $\rho(R,Q) \sim \delta^{-N+1}l(R)$ means $ \delta^{-N+1}l(R) \leq \rho(R,Q) <  \delta^{-N}l(R)$. Obviously, for $N = 0$, the cubes $R$ such that $\rho(Q,R) \sim \delta l(R)$ simply are the cubes $R$ at distance less than $l(R)$ of $Q$. Remark that everything has been done so that for all $N \in \N$, for fixed $Q$ and $S$,
 \[ \sum_{R \, \mathrm{pa} \, 1 ; l(R) = l(Q) \atop R \subset S ; R \neq S }{\beta^N_{Q,S,R}} = 0. \]
Now, set
 \[ \gamma^N_{Q,S} =  \sum_{R \, \mathrm{pa} \, 1 ; l(R) = l(Q) \atop R \subset S ; R \neq S }{\beta^N_{Q,S,R} \langle g, \phi^1_S \rangle [\phi^1_S]_R }, \]
and
\[ \langle f, V^N_{1,3,1} g \rangle = \sum_{Q  \, \mathrm{spa} \, 2 \atop  \mathrm{and} \, \mathrm{pa} \, 1 } \sum_{ S  \, \mathrm{spa} \, 1  \atop l(S) > l(Q)}{ \gamma^N_{Q,S}}, \]
so that $V_{1,3,1} = \sum_{N \in \N}{V^N_{1,3,1}}.$
We will prove that for all $N \geq 0$
\[ | \langle f, V^N_{1,3,1} g \rangle | \lesssim (N+1) \delta^{N \alpha } \mu(Q_0),  \]
and it will give us the expected bound for $V_{1,3,1}$ by summing over $N \in \N$. To prove this result, we split the sum into two parts depending on the relative sizes of $Q$ and $S$:
\[  \langle f, V^N_{1,3,1} g \rangle = \sum_{Q  \, \mathrm{spa} \, 2 \atop  \mathrm{and} \, \mathrm{pa} \, 1 } \sum_{ S  \, \mathrm{spa} \, 1  \atop \delta^{N+3} < \frac{l(Q)}{l(S)} \leq \delta } {\!\!\!\!\!\!\gamma^N_{Q,S}} +  \sum_{Q  \, \mathrm{spa} \, 2 \atop  \mathrm{and} \, \mathrm{pa} \, 1 } \sum_{ S  \, \mathrm{spa} \, 1  \atop \frac{l(Q)}{l(S)} \leq \delta^{N+3}}{\!\!\!\!\!\! \gamma^N_{Q,S}} = \mathrm{I} + \mathrm{II}. \]
For the first term, fix cubes $Q$ and $S$. Remark that for a cube $R \neq Q$, $\beta^N_{Q,S,R} \neq 0 \Rightarrow \rho(Q,R) \geq \delta^{-N+1}l(R)$, so that applying the coefficient estimate \eqref{eqn:5}, one gets
\begin{align*}
| \beta^N_{Q,S,R} | \lesssim   | \langle f, \phi_Q^2 \rangle | \, \alpha_{Q,R} \,  \mu(R)^{\frac{1}{2}} &  \lesssim  | \langle f, \phi_Q^2 \rangle | \, \mu(Q)^{\frac{1}{2}} \frac{\mu(R)}{\mu(Q,R)} \left ( \frac{l(Q)}{\rho(Q,R)}  \right )^{\alpha} \\
& \lesssim  | \langle f, \phi_Q^2 \rangle | \, \mu(Q)^{\frac{1}{2}} \delta^{N \alpha} \frac{\mu(R)}{\mu(Q,R)}.\\
\end{align*}
Remark also that we have $ | \langle g, \phi^1_S \rangle [\phi^1_S]_R | \lesssim  | \langle g, \phi^1_S \rangle | {\mu(S)}^{-\frac{1}{2}}$ as a consequence of properties $(1)$ and $(2)$ of Lemma \ref{representation}. Finally, observe that $\gamma^N_{Q,S} \neq 0 \Rightarrow \rho(Q,S) \leq \delta^{-N}l(Q)$, because else all the $\beta^N_{Q,S,R}$ are equal to zero.
Now, combining all this, write
\begin{align*}
| \mathrm{I} | & \lesssim \delta^{N \alpha} \sum_{Q  \, \mathrm{spa} \, 2} \sum_{ S  \, \mathrm{spa} \, 1  \atop {\delta^{N+3} < \frac{l(Q)}{l(S)} \leq \delta \atop \rho(Q,S) \leq \delta^{-N}l(Q)}}{ | \langle f, \phi_Q^2 \rangle ||\langle g, \phi^1_S \rangle|\left ( \frac{\mu(Q)}{\mu(S)} \right )^{\frac{1}{2}} }  \sum_{R  \, \mathrm{pa} \, 1  \atop{ l(R) = l(Q) \atop { R \subset S \atop \rho(R,Q) \sim \delta^{-N}l(Q)}}} { \frac{\mu(R)}{\mu(Q,R)}  }\\
& \lesssim  \delta^{N \alpha} \, S_1^{1/2} \times S_2^{1/2},\\
\end{align*}
where 
\[ S_1 =  \sum_{Q  \, \mathrm{spa} \, 2} \sum_{ S  \, \mathrm{spa} \, 1  \atop {\delta^{N+3} < \frac{l(Q)}{l(S)} \leq \delta \atop \rho(Q,S) \leq \delta^{-N}l(Q)}}{{ | \langle f, \phi_Q^2 \rangle |}^2 } \sum_{R  \, \mathrm{pa} \, 1  \atop{ l(R) = l(Q) \atop{ R \subset S \atop \rho(R,Q) \sim \delta^{-N}l(Q)}}} {  \frac{\mu(R)}{\mu(Q,R)}  }  \]
\[  S_2 =  \sum_{Q  \, \mathrm{spa} \, 2} \sum_{ S  \, \mathrm{spa} \, 1  \atop {\delta^{N+3} < \frac{l(Q)}{l(S)} \leq \delta \atop \rho(Q,S) \leq \delta^{-N}l(Q)}}{{|\langle g, \phi^1_S \rangle|}^2 \, \frac{\mu(Q)}{\mu(S)} }  \sum_{R  \, \mathrm{pa} \, 1  \atop{ l(R) = l(Q) \atop { R \subset S \atop \rho(R,Q) \sim \delta^{-N}l(Q)}}} {  \frac{\mu(R)}{\mu(Q,R)}  } . \]     
Let us first consider $S_1$. Observe that for a fixed cube $Q$ of center $z_Q$, if $\rho(Q,R) \sim \delta^{-N}l(Q)$, then $\mu(Q,R) \approx \mu(B(z_Q, \delta^{-N}l(Q)))$. Therefore
\[  \sum_{R  \, \mathrm{pa} \, 1  \atop{ l(R) = l(Q) \atop { R \subset S \atop \rho(R,Q) \sim \delta^{-N}l(Q)}}} {\!\!\!\!\!\!\!\!\! \!\!\!\!\!  \frac{\mu(R)}{\mu(Q,R)}  } \lesssim \frac{1}{\mu(B(z_Q, \delta^{-N}l(Q)))} \sum_{R \subset S \atop l(R) = l(Q)} \mu(R) \lesssim \frac{\mu(S)}{\mu(B(z_Q, \delta^{-N}l(Q)))}. \]
Then, the cubes $S$ of a fixed generation, at a distance less than $\delta^{-N}l(Q)$ of $Q$, and of a diameter not exceeding $\delta^{-N-3}l(Q),$ pave a ball of a measure comparable to $\mu(B(z_Q, \delta^{-N}l(Q)))$. So, summing over those cubes $S$, and then summing over those $N+3$ generations, we get
\[ S_1 \lesssim (N+1)  \sum_{Q  \, \mathrm{spa} \, 2}{{ | \langle f, \phi_Q^2 \rangle |}^2 } \lesssim (N+1) \|f\|^2_{L^2(Q_0)} \lesssim (N+1) \mu(Q_0). \] 
For $S_2$, it is the same idea. Observe that for a fixed cube $S$ of center $z_S$, if  $\rho(Q,R) \sim \delta^{-N}l(Q)$, then $\mu(Q,R) \approx \mu(B(z_R, \delta^{-N}l(Q)))$, where $z_R$ is the center of the cube $R$. But since $S$ has a diameter not exceeding $\delta^{-N-3}l(Q),$, we even have $\mu(Q,R) \approx \mu(B(z_S, \delta^{-N}l(Q)))$. Therefore
\[ \frac{\mu(Q)}{\mu(S)} \sum_{R  \, \mathrm{pa} \, 1  \atop{ l(R) = l(Q) \atop { R \subset S \atop \rho(R,Q) \sim \delta^{-N}l(Q)}}} {\!\!\!\!\!\!\!\!\! \!\!\!\!\!  \frac{\mu(R)}{\mu(Q,R)}  } \lesssim \frac{\mu(Q)}{\mu(B(z_S, \delta^{-N}l(Q)))}.\]
Now, summing once more over cubes $Q$ of the same generation, then summing over those $N+3$ generations, and at last summing over the cubes $S$ spa $1$, we also get
 \[ S_2 \lesssim (N+1) \mu(Q_0), \]
which finally gives us the bound
\[ |\mathrm{I}| \lesssim (N +1)\delta^{N\alpha} \mu(Q_0),\]
for all $N \in \N.$
We now estimate the term
\[ \mathrm{II} = \sum_{Q  \, \mathrm{spa} \, 2 \atop  \mathrm{and} \, \mathrm{pa} \, 1 } \sum_{ S  \, \mathrm{spa} \, 1  \atop \frac{l(Q)}{l(S)} \leq \delta^{N+3}}{\!\!\!\!\!\! \gamma^N_{Q,S}}.  \] 
Fix such cubes $Q$ and $S$. For any cube $R$ strictly contained inside $S$, we know that $\phi^1_S $ is constant over $R$. Denote by $\fr(S)$ the union of the boundaries $\overline{S'}\backslash S'$ of all the children $S'$ of $S$. If $\rho(Q,\fr(S)) > \delta^{-N-1} l(Q),$ then either $Q \cap S = \varnothing$, which implies that all the $\beta^N_{Q,S,R}$ are equal to zero, and $\gamma^N_{Q,S}$ as well, or $Q \subset S$. In this case, if $\beta^N_{Q,S,R} \neq 0$, then $\rho(Q,R) \leq \delta^{-N} l(Q)$, and with the assumption we made it implies that $R$ is in the same child of $S$ as $Q$. Therefore, for all the $\beta^N_{Q,S,R}$ non equal to zero, $[ \phi^1_S ]_R$ has the same value, and $\gamma^N_{Q,S} = [ \phi^1_S  ]_Q \sum_R{\beta^N_{Q,S,R}} = 0.$ Thus $\gamma^N_{Q,S} \neq 0 \Rightarrow \rho(Q,\fr(S)) \leq \delta^{-N-1} l(Q). $
Now, if that condition is satisfied, applying the coefficient estimate \eqref{eqn:5}, we have
\[ |\beta^N_{Q,S,R}|  \lesssim  | \langle f, \phi_Q^2 \rangle | \, \mu(Q)^{\frac{1}{2}} \delta^{N \alpha} \frac{\mu(R)}{\mu(Q,R)}, \]
and 
\begin{align*}
\sum_{R \subset S ; l(R) = l(Q) \atop \rho(R,Q) \sim \delta^{-N}l(Q)}{\!\!\! \! |\beta^N_{Q,S,R}| |[ \Delta^1_S g ]_R|} & \lesssim  | \langle f, \phi_Q^2 \rangle | \, |\langle g, \phi^1_S \rangle| \, \left ( \frac{\mu(Q)}{\mu(S)} \right )^{\frac{1}{2}} \delta^{N \alpha}\!\!\!\!\!\! \sum_{l(R) = l(Q) \atop \rho(R,Q) \sim \delta^{-N}l(Q) }{ \!\!\!\!\!\frac{\mu(R)}{\mu(B(z_Q,\delta^{-N}l(Q)))}}\\
& \lesssim  \delta^{N \alpha}  | \langle f, \phi_Q^2 \rangle | \, |\langle g, \phi^1_S \rangle| \, \left ( \frac{\mu(Q)}{\mu(S)} \right )^{\frac{1}{2}}.\\
\end{align*}
Therefore, if we introduce the set $E^N_{t,S} = \bigcup_{P \in \theta^N_{t,S}} {P}$, where 
\[ \theta^N_{t,S} = \{ P \subset Q_0 ;  l(P) = t l(S) \quad \mathrm{and} \quad \rho(P,\fr(S)) \leq (\delta^{-N-1} t )l(S) \}, \] 
we have
\[ | \gamma^N_{Q,S}| \lesssim  \delta^{N \alpha}  | \langle f, \phi_Q^2 \rangle | \ \ |\langle g, \phi^1_S \rangle| \ \ \left ( \frac{\mu(Q)}{\mu(S)} \right )^{\frac{1}{2}} 1_{Q \subset E^N_{\frac{l(Q)}{l(S)},S}}. \]
Remark that if $x \in P$, with $P \in \theta^N_{t,S}$, then $\rho(x,\fr(S)) \lesssim \rho(P,\fr(S)) + l(P) \lesssim  (\delta^{-N-1} t )l(S) + t l(S)  \lesssim (\delta^{-N-1} t)l(S)$ as $N \geq 0$. Thus, by the small boundary condition $(6)$ of Lemma \ref{cubes}, $\mu(E^N_{t,S} ) \lesssim (\delta^{-N-1}t)^{\eta} \mu(S). $
Let $\beta > 0$ be such that $\eta - 2 \beta >0$, and let $\lambda_{\frac{l(Q)}{l(S)}} = (\delta^{-N-1 +j})^{\beta}$ whenever $\frac{l(Q)}{l(S)}= \delta^j$. Now, by the Cauchy-Schwarz inequality, write
\[ |\mathrm{II}| \lesssim \delta^{N\alpha}\left \{ \!\! \sum_{Q  \, \mathrm{spa} \, 2 \atop  \mathrm{and} \, \mathrm{pa} \, 1 }\!\! \sum_{ S  \, \mathrm{spa} \, 1  \atop \frac{l(Q)}{l(S)} \leq \delta^{N+3}}{\!\!\!\!\!\!\!\!{ | \langle f, \phi_Q^2 \rangle |}^2 \lambda^2_{\frac{l(Q)}{l(S)}} 1_{Q \subset E^N_{\frac{l(Q)}{l(S)},S}}    } \! \right  \}^{\!\! \frac{1}{2}} \!\! \left  \{ \!\!\sum_{Q  \, \mathrm{spa} \, 2 \atop  \mathrm{and} \, \mathrm{pa} \, 1 }\!\! \sum_{ S  \, \mathrm{spa} \, 1  \atop \frac{l(Q)}{l(S)} \leq \delta^{N+3}}  {\!\!\!\!\!\!\!\!{|\langle g, \phi^1_S \rangle|}^2 \lambda^{-2}_{\frac{l(Q)}{l(S)}} \frac{\mu(Q)}{\mu(S)} 1_{Q \subset E^N_{\frac{l(Q)}{l(S)},S}} }\! \right \}^{\!\! \frac{1}{2}} \]
 Observe that for a fixed cube $Q$ and a fixed value of $l(S)$, the number of cubes $S$ such that $Q \subset E^N_{\frac{l(Q)}{l(S)},S}$ is uniformly bounded. Therefore
 \begin{align*}
 \sum_{Q  \, \mathrm{spa} \, 2 \atop  \mathrm{and} \, \mathrm{pa} \, 1 } \sum_{ S  \, \mathrm{spa} \, 1  \atop \frac{l(Q)}{l(S)} \leq \delta^{N+3}}{\!\!\!\!\!\!\!\!{ | \langle f, \phi_Q^2 \rangle |}^2 \lambda^2_{\frac{l(Q)}{l(S)}} 1_{Q \subset E^N_{\frac{l(Q)}{l(S)},S}}    } & \lesssim \sum_{Q  \, \mathrm{spa} \, 2}{{ | \langle f, \phi_Q^2 \rangle |}^2 } \sum_{j \geq N+3}{(\delta^{-N-1+j})^{2 \beta}  } \sum_{S  \, \mathrm{spa} \, 1  \atop \frac{l(Q)}{l(S)} = \delta^j}{1_{Q \subset E^N_{\frac{l(Q)}{l(S)},S}}} \\
 & \lesssim  \sum_{Q  \, \mathrm{spa} \, 2}{{ | \langle f, \phi_Q^2 \rangle |}^2 } \lesssim \mu(Q_0).\\
 \end{align*}
 We can similarly bound the second sum: we have
 \begin{align*}
 \sum_{Q  \, \mathrm{spa} \, 2 \atop  \mathrm{and} \, \mathrm{pa} \, 1 } \sum_{ S  \, \mathrm{spa} \, 1  \atop \frac{l(Q)}{l(S)} \leq \delta^{N+3}}  {\!\!\!\!\!\!\!\!{|\langle g, \phi^1_S \rangle|}^2 \lambda^{-2}_{\frac{l(Q)}{l(S)}} \frac{\mu(Q)}{\mu(S)} 1_{Q \subset E^N_{\frac{l(Q)}{l(S)},S}} } & \lesssim \sum_{S  \, \mathrm{spa} \, 1}{{|\langle g, \phi^1_S \rangle|}^2} \sum_{j \geq N+3}{(\delta^{-N-1+j})^{-2 \beta}}\ \!\!\!\!\! \sum_{Q  \, \mathrm{spa} \, 2 \atop {l(Q)= \delta^j l(S) \atop Q \subset E^N_{\delta^j,S}}} {\!\!\!\frac{\mu(Q)}{\mu(S)}}\\
 & \lesssim \sum_{S  \, \mathrm{spa} \, 1}{{|\langle g, \phi^1_S \rangle|}^2} \sum_{j \geq N+3}{(\delta^{-N-1+j})^{-2 \beta} \ \ \frac{\mu(E^N_{\delta^j,S})}{\mu(S)}}\\
 & \lesssim \sum_{S  \, \mathrm{spa} \, 1}{{|\langle g, \phi^1_S \rangle|}^2} \sum_{j \geq N+3}{(\delta^{-N-1+j})^{\eta -2 \beta}}\\
 & \lesssim  \sum_{S  \, \mathrm{spa} \, 1}{{|\langle g, \phi^1_S \rangle|}^2} \lesssim \mu(Q_0).\\ 
 \end{align*}
We thus proved that $|\mathrm{II}| \lesssim \delta^{N\alpha} \mu(Q_0)$, and therefore for all $N\in \N$, we have as expected $| \langle f, V^N_{1,3,1} g \rangle | \lesssim (N+1) \delta^{N \alpha } \mu(Q_0).$ Summing over $N \in \N$, this handles the term $\langle f, V_{1,3,1} g \rangle$.
 
 \bigskip
 
\subparagraph{\textit{Estimate of $\langle f, V_{1,3,2} g \rangle$}}Remember that by definition of the "buffer" functions $\xi^1_S g$, if $R \subset S'$, $S' \in \widetilde{S}$, we have $[\xi^1_S g]_R = a_{S'} [b^1]_R$, where $a_{S'}$ denotes the coefficient uniformly bounded by $\|g\|_{\infty} \leq 1$ introduced in Lemma \ref{Hofmann}. Therefore
\begin{align*}
\langle f, V_{1,3,2} g \rangle & = \sum_{Q  \, \mathrm{spa} \, 2 \atop  \mathrm{and} \, \mathrm{pa} \, 1 } \sum_{ S  \, \mathrm{dpa} \, 1  \atop l(S) > l(Q)} \sum_{S' \in \widetilde{S}}\sum_{R  \, \mathrm{pa} \, 1 ; l(R) = l(Q) \atop R \subset S'}{\beta_{Q,R} \ \ a_{S'}  } \\
 & = \sum_{Q  \, \mathrm{spa} \, 2 \atop  \mathrm{and} \, \mathrm{pa} \, 1 } \sum_{ S  \, \mathrm{dpa} \, 1  \atop l(S) > l(Q)} \sum_{S' \in \widetilde{S}}{a_{S'}\ \ \mu(S)^{\frac{1}{2}} }\sum_{R  \, \mathrm{pa} \, 1 ; l(R) = l(Q) \atop R \subset S'}{\beta_{Q,R} \ \ \mu(S)^{-\frac{1}{2}} }.\\
 \end{align*}
Remark that by the $C_X-$packing property \eqref{packing2} of the dpa $1$ cubes,
 \[  \sum_{ S  \, \mathrm{dpa} \, 1 } \sum_{S' \in \widetilde{S}}{\left | a_{S'}\ \ \mu(S)^{\frac{1}{2}} \right |^2} \lesssim \sum_{ S  \, \mathrm{dpa} \, 1 }{\mu(S)} \lesssim \mu(Q_0). \]
That is all we need to apply exactly the same argument we applied to $\langle f, V_{1,3,1} g \rangle$, \emph{i.e.} applying the pseudo-annular decomposition of $V_{1,3,2}$ into operators $V^N_{1,3,2}$, estimating those operators, and then summing over $N \in \N$. Indeed, the crucial point was that we could bound the coefficients $ |\langle g, \phi^1_S \rangle [ \phi^1_S  ]_R|$ by $|\langle g, \phi^1_S \rangle| \, {\mu(S)}^{-\frac{1}{2}}$, and then that the $|\langle g, \phi^1_S \rangle|$ were square summable over the spa $1$ cubes $S$. We have the same square summability property here, and everything works out in the same way, so that though we will omit the detail, we obtain the expected estimate $| \langle f, V_{1,3,2} g \rangle | \lesssim \mu(Q_0).$
This completes the proof of Theorem \ref{AR}.
 
 \medskip
 
 \section{Two particular cases}
 
 \medskip
 
 \subsection{Perfect dyadic operators}
 
The notion of perfect dyadic operators was introduced in \cite{AHMTT}. We recall the definition.
 \begin{definition}{Perfect dyadic singular integral operators (pdsio).\\}
We say that $T$ is a perfect dyadic singular integral operator on $X$ if it satisfies the following properties:
\begin{enumerate}
\item $T$ is a linear continuous operator from $D_{\alpha}$ to $D'_{\alpha}$.
\item $T$ has a kernel $K$ satisfying the size condition \eqref{standard1}.
\item $\langle g , Tf \rangle$ is well defined for pairs of functions $(f,g) \in L^p_c(X) \times L^{p'}_c(X)$ for $1<p<+\infty$ and if, furthermore, they are integrable with support on disjoint dyadic cubes (up to a set of measure $0$)
\[ \langle g,Tf \rangle = \int \!\! \int_{X\times X}{g(x)K(x,y)f(y)d\mu(x)d\mu(y)}. \]
\item For all $(f,g)$ as above, if $f$ has support in a dyadic cube $Q$ and mean $0$, then $\langle g,Tf \rangle =0$ whenever the support of $g$ does not meet $Q$ (up to a set of measure $0$).
\end{enumerate}
\end{definition}

A pdsio has in a sense its singularity adapted to the dyadic grid. As was proven in \cite{AHMTT} in the setting of the real line (but there is no difficulty adapting it to a space of homogeneous type), pdsio satisfy Theorem \ref{AR} without having to suppose conditions \eqref{WBP1}, \eqref{WBP2}, and with \eqref{accretive3bis} instead of \eqref{accretive3}. Looking at our proof, it is easy to recover this result. Indeed, most of the terms we had to estimate simply vanish if one considers pdsio, and only the diagonal terms remain, which makes the proof a lot more straightforward. Here is how one estimates those diagonal terms. Consider for example the term
\[   \langle  b^2 \phi_Q^{b^2} , T(b^1  \phi_Q^{b^1} )    \rangle  = \sum_{Q' \in \widetilde{Q},R' \in \widetilde{Q}}{[ \phi_Q^{b^2}   ]_{Q'} \langle b^2 1_{Q'}  , T(b^1 1_{R'})  \rangle [\phi_Q^{b^1}  ]_{R'}  },\]
when $Q$ is a spa $1$ and spa $2$ cube. It remains to compute $\langle b^2 1_{Q'}  , T(b^1 1_{R'})  \rangle$. If $Q' \neq R'$, one cannot use the Hardy type inequality \eqref{Hardy1} directly because the exponents $p,q$ can happen to be uncompatible if $1/p+1/q >1$, but one can circumvent this as follows: observe that $\int{(b^1 1_{R'} - [b^1]_{R'}1_{R'})d\mu } = \int{(b^21_{Q'} - [b^2]_{Q'} 1_{Q'})d\mu} = 0$, so that
\begin{equation} \label{pdsio}
\langle b^2 1_{Q'} , T(b^1 1_{R'}) \rangle = \langle b^2 1_{Q'} , T(1_{R'}) \rangle [b^1]_{R'} = [b^2]_{Q'} \langle 1_{Q'} , T(1_{R'}) \rangle [b^1]_{R'}.
\end{equation}
Now, we can apply \eqref{Hardy1} because $ 1_{Q'},1_{R'} \in L^{\infty}(X)$. If $Q' =R'$, using the fact that $T$ is perfect dyadic, write
 \begin{align*} 
 \langle b^2 1_{Q'} , T(b^1 1_{Q'}) \rangle  & =  \ \  \langle b^2 1_{Q'} - [b^2]_{Q'}b^2_{Q'}, T(b^1 1_{Q'}) \rangle + [b^2]_{Q'} \langle  b^2_{Q'}, T(b^1 1_{Q'} - [b^1]_{Q'}b^1_{Q'})  \rangle\\
 & \ \ \ \ + [b^2]_{Q'} [b^1]_{Q'} \langle  b^2_{Q'}, T(b^1_{Q'})  \rangle \\
 &  = \ \  \langle b^2 1_{Q'} - [b^2]_{Q'}b^2_{Q'}, T(b^1 ) \rangle +  [b^2]_{Q'} \langle  T^{\ast}(b^2_{Q'}), b^1 1_{Q'} - [b^1]_{Q'}b^1_{Q'}  \rangle\\
 & \ \ \ \ + [b^2]_{Q'} [b^1]_{Q'} \langle  b^2_{Q'}, T(b^1_{Q'})  \rangle,\\
 \end{align*}
 and all those terms are easily estimated applying \eqref{accretive2}, \eqref{accretive3} (on same $Q'$) and the fact that $Q'$ is pa $1$ and pa $2$. We thus always have
 \[ | \langle b^2 1_{Q'} , T(b^1 1_{R'}) \rangle | \lesssim \mu(Q')^{1/2} \mu(R')^{1/2} \lesssim \mu(Q'), \]
because these cubes have comparable measure as they are children of the same cube. Consequently, using the $L^{\infty}$ estimates of the functions $\phi_Q^{b^2}, \phi_Q^{b^1}$ seen in Section $5.1$, we still have the estimate
 \[ \left |    \langle  b^2  \phi_Q^{b^2} , T(b^1  \phi_Q^{b^1} )    \rangle  \right |  \lesssim  1, \]
and all this argument remains valid when $1/p + 1/q >1$. All the other terms can be treated in the same way: to put it roughly, when considering pdsio and estimating terms where one had to use \eqref{WBPC}, one can always reduce to considering $L^{\infty}$ functions by substracting the mean value of the functions involved, and then apply the Hardy inequality \eqref{Hardy1} to those $L^{\infty}$ functions. Similarly, one needs not suppose an integrability of $T(b^i_Q)$ over $\widehat{Q}$, an integrability condition over $Q$ suffices. Putting all this together, we recover Theorem $6.8$ of \cite{AHMTT}, with a different proof.

\begin{remark} 
The identity \eqref{pdsio} shows that, whenever $T$ is a pdsio, if $Q \neq R$ are neighbors and $\supp{f} \subset Q, \supp{g} \subset R$, then for all $1\leq p,q \leq +\infty$, we have
\[ |\langle Tf,g \rangle | \leq | \langle T1_{Q}, 1_{R} \rangle | \,  |[f]_{Q}| \, |[g]_{R}| \leq C \mu(Q)^{1-\frac{1}{p} - \frac{1}{q}} \|f\|_p \|g\|_q . \]
Such an inequality is wrong when $1/p +1/q>1$ for standard sio. This indicates that pdsio do not approximate well standard sio in this range of exponents. 
\end{remark}

\subsection{The case $1/p + 1/q \leq 1$}
We have stated in Section $3$ that when $1/p +1/q \leq 1$, Theorem \ref{AR} is valid under the hypotheses \eqref{accretive1}, \eqref{accretive2} and \eqref{accretive3bis}. This is Theorem $3.4$ of \cite{AY}. Let us prove this statement. We already pointed out that in that particular case, hypotheses \eqref{accretive3} and \eqref{WBP1} come as a consequence of \eqref{accretive1}, \eqref{accretive2} and \eqref{accretive3bis} (see Proposition \ref{particular}). However we cannot directly show \eqref{WBP2}, which we obviously need in order to estimate all the diagonal terms. The proof thus has to be adapted slightly. More precisely, we need to include a stronger control in stopping time Lemma \ref{Hofmann}: we need to control also the maximal function (non dyadic) of $b^1$, as we want the means $[|Mb^1|^p]_{\widehat{Q}}$ to be uniformly bounded on the pa $1$ cubes $Q$. To achieve this, the stopping time has to be on whether
\[ (i) \quad [|b^1|]_Q < \delta \] or \[(ii) \quad  [|T(b^1)|^{q'}]_Q + \sup_{E}{[|Mb^1|^p]_{E}} >C,\] 
for dyadic subcubes $Q$ of $Q_0$ for appropriately chosen $\delta>0$ and $C<+\infty$, and where the supremum runs over the unions $E$ of dyadic cubes such that $Q \subset E  \subset \widehat{Q}$. This can easily be done and we refer to \cite{Hofmann} for the detail. Once having this additional control, the diagonal terms of \eqref{WBP2} can be estimated as follows. If $R$ is a pa $1$ and pa $2$ dyadic cube, write
\[ \langle b^2 1_{R}, T(b^1 1_{R}) \rangle = [b^2]_{R} \langle T^{\ast} (b^2_{R}) , b^1 1_{R} \rangle  + \langle h, T(b^1) \rangle - \langle h, T(b^1 1_{\widehat{R} \backslash R}) \rangle - \langle h, T(b^1 1_{{\widehat{R}}^c}) \rangle, \]
where $h= b^2 1_{R} - [b^2]_{R} b^2_{R}.$ Remark that by the controls given by the stopping time, we have $\|h\|^q_q \lesssim \mu(R).$
For the first term, use \eqref{accretive3} for $T^{\ast}(b^2_{R})$ on $R$, \eqref{Hofmann3} for $b^1$ and $b^2$ on $R$, and $p'\leq q$:
\[ | [b^2]_{R} \langle T^{\ast} (b^2_{R}) , b^1 1_{R} \rangle | \lesssim \| T^{\ast} (b^2_{R})  \|_{L^{p'}(R)} \| b^1 \|_{L^p(R)} \lesssim \mu(R).  \]
For the same reasons, the second term is also bounded:
\[ | \langle h, T(b^1) \rangle  | \lesssim \|h\|_q \|T(b^1) \|_{L^{q'}(R)} \lesssim \mu(R). \]
Apply \eqref{Hardy1} to the third term, $q' \leq p$, and the control on $Mb^1$ given by the stopping-time to get
\begin{align*} 
|  \langle h, T(b^1 1_{\widehat{R} \backslash R}) \rangle | & \lesssim \|h\|_q \| T(b^1 1_{\widehat{R} \backslash R}) \|_{L^{q'}(R)}  \lesssim \|h\|_q \| b^1 \|_{L^{q'}(\widehat{R})} \\
& \lesssim   \|h\|_q \|Mb^1\|_{L^{q'}(\widehat{R})}\lesssim  \mu(R).
\end{align*} 
Finally, for the last term, since $h$ has mean $0$, we can apply \eqref{standard estimate dyadic} to get
\begin{align*}  
|\langle h, T(b^1 1_{{\widehat{R}}^c}) \rangle | & \lesssim \int_{x \in R}{|h(x)|} \int_{y \in \widehat{R}^c}{|b^1(y)| \left ( \frac{l(R)}{\rho(x,y)} \right ) ^{\alpha} \frac{1}{\lambda(x,y)} d\mu(y) d\mu(x) }\\  
& \lesssim  \int_{x \in R}{|h(x)|}  \sum_{j\geq 0}{2^{-j \alpha}} \int_{2^j l(R) \leq \rho(x,y) < 2^{j+1}l(R)}{|b^1(y)| \frac{1}{\lambda(x,y)} d\mu(y) d\mu(x) }\\
& \lesssim \int_{x \in R}{|h(x)|}  \sum_{j\geq 0}{2^{-j \alpha}} \frac{1}{\mu(B(x, 2^{j+1} l(R)))}  \int_{B(x, 2^{j+1} l(R))}{|b^1(y)| d\mu(y) d\mu(x) }\\
& \lesssim \int_{x \in R}{|h(x)| \, |Mb^1(x)| d\mu(x)} \lesssim \| h \|_{L^{q}(R)} \| Mb^1 \|_{L^{q'}(R)} \lesssim \mu(R),
\end{align*}
where once again we have used the control on $Mb^1$ given by the stopping time and $q' \leq p$. Summing up, we obtain as desired
\[ | \langle b^2 1_{R}, T(b^1 1_{R}) \rangle | \lesssim \mu(R),\]
for the pa $1$ and pa $2$ cubes $R$. This allows us to estimate all the diagonal terms appearing in our argument, which we previously took care of by applying hypothesis \eqref{WBP2}. It means that when $1/p + 1/q \leq 1$, we do not need to assume \eqref{WBP2}, nor \eqref{WBP1}, and this proves Theorem \ref{AR} under \eqref{accretive1}, \eqref{accretive2} and \eqref{accretive3bis}.

 \medskip

\section{Hardy type inequalities on spaces of homogeneous type}
\medskip

\subsection{A closer look at the Hardy inequalities}We introduced in the previous sections some Hardy type inequalities which we will review in detail now. First, let us give the proof of Lemma \ref{Hardydyadic}.

\begin{proof}[Proof of Lemma \ref{Hardydyadic}]\footnote{We thank J.-M. Martell for the suggestion of using Whitney coverings that led to an improvement of our earlier argument.}
Let $I$ be the integral in \eqref{2etoile}. We prove that for all $1 < r,s < \infty, $ we have
\[ I \lesssim   \left \langle \left (M_\mu (|f|^r) \right )^{1/r} , |g| \right \rangle + \left \langle |f| , \left (M_\mu (|g|^s) \right )^{1/s} \right \rangle  ,\]
and the Hardy-Littlewood maximal theorem then gives the desired result, choosing $1 < r < \nu$ and $1 < s < \nu'$. Without loss of generality, we can assume $f,g \geq 0$. We have
\[ I  =  \int_{x \in Q' \atop \rho (x,Q)>0}{\!\!\! g(x) \! \int_{y \in Q  \atop \rho (y,Q') \leq \rho (x,Q) }{\! \frac{f(y)}{\lambda(x,y)}d\mu(y)d\mu(x)}} 
+ \int_{y \in Q \atop \rho (y,Q')>0}{\!\!\!\! f(y) \! \int_{x \in Q' \atop \rho (y,Q') > \rho (x,Q)}{\!  \frac{g(x)}{\lambda(x,y)}d\mu(x)d\mu(y)}} \]
Indeed, $\mu \left ( \{x \in Q' \, | \, \rho (x,Q) = 0  \} \right ) = \mu \left ( \{y \in Q \, | \, \rho (y,Q') = 0  \} \right ) = 0$, because of the property $(6)$ of Lemma \ref{cubes}: we have for example
\[ \{x \in Q' \, | \, \rho (x,Q) = 0  \} \subset \bigcap_{j \in \N} \{x \in Q' \, |  \, \rho (x,{Q'}^c) \leq 2^{-j} \delta ^k  \}, \]
and all those sets have their measure bounded by $\mu(Q') 2^{-j \eta} \underset{j \to \infty}{\longrightarrow} 0$.
By symmetry, it is enough to estimate the first integral for example, which we call $I_1$. For $x \in Q'$, let $E_x = \left \{ y \in Q \, | \, \rho (y,Q^c) \leq \rho (x,Q) \right \}$. For a fixed $x$ in $Q'$, we prove that for all $1< r < \infty$,
\[ I_1(x) =  \int_{y \in E_x}{K(x,y)f(y)d\mu(y)}   \lesssim M_{\mu}(f^r)(x) ^{1/r}.   \]
Consider the dyadic subcubes $Q'$ of $Q$, which are maximal for the relation $l(Q') \leq \rho (Q',x)$. Call them $Q^l_{\alpha}(x).$ They realize a partition of the cube $Q$, as for every $y \in Q$, there exists a sufficiently small cube $Q_y$ containing $y$ such that $l(Q_y) \leq \rho(x,Q) \leq \rho(x,Q_y),$ and $Q_y$ is then included in one of those maximal $Q^l_{\alpha}(x)$. 
The case where there is a unique cube $Q^l_{\alpha}(x)=Q$ means that $l(Q) \leq \rho(x,Q)$. Then $\lambda(x,y) \geq \mu(B(x,l(Q)))$ while $Q \subset B(x,Cl(Q))$ for some dimensional constant $C$. Hence $I_1(x) \lesssim M_{\mu}f(x)$.
The second case is when $Q$ is not a maximal cube. By maximality, if $Q^{l-1}_{\beta} \subset Q$ is the unique parent of a $Q^l_{\alpha}(x)$, then we have $\rho(x,Q^{l-1}_{\beta}) < l(Q^{l-1}_{\beta}) $, and thus $\rho(x,Q^l_{\alpha}(x)) \lesssim \rho(x,Q^{l-1}_{\beta}) + l(Q^{l-1}_{\beta}) \lesssim l(Q^{l-1}_{\beta}) = \delta^{-1} l(Q^{l}_{\alpha}(x))$. Note that this implies $\rho(x,Q) \leq \rho(x,Q^l_{\alpha}(x)) \lesssim \delta^l$ with implicit constant independent of $x$ and $l$. For a fixed $l$, let 
\[C^l (x) = \bigcup_{\alpha: Q^l_{\alpha}(x) \cap E_x \neq \varnothing}{Q^l_{\alpha}(x)}.\]
Observe that if $y \in Q^l_{\alpha}(x) \cap E_x$, then $\rho(y, (Q^l_{\alpha}(x))^c) \leq \rho(y,Q^c) \leq \rho(x,Q)$. Thus, $\mu(Q^l_{\alpha}(x) \cap E_x) \lesssim \left ( \frac{\rho(x,Q)}{\delta^l} \right )^{\eta} \mu(Q^l_{\alpha}(x))$ by the small boundary property $(6)$ of Christ's dyadic cubes stated in Lemma \ref{cubes}. Summing over the cubes in $C^l(x)$, we have
\[ \mu(C^l (x) \cap E_x) \lesssim \left ( \frac{\rho(x,Q)}{\delta^l} \right )^{\eta} \mu(C^l(x)) \lesssim \left ( \frac{\rho(x,Q)}{\delta^l} \right )^{\eta} \mu(B(x,\delta^l)), \]
the last inequality being a consequence of doubling and the fact that each of the cubes $Q^l_{\alpha}(x)$ in $C^l(x)$ is of length $\delta^l$ and at a distance comparable to $\delta^l$ from $x$, so that $C^l(x) \subset B(x,C\delta^l)$ for some $C$ independent of $l$ and $x$. Now, write
\begin{align*} 
I_1(x) & \lesssim \sum_{l}{\int_{C^l(x) \cap E_x}{\frac{|f(y)|}{\mu(B(x,\rho(x,y)))}d\mu(y) }} \\ 
& \lesssim \sum_{l}{ \frac{1}{\mu(B(x,\delta^l))}  \int_{B(x,C\delta^l)}{|f| 1_{C^l(x) \cap E_x} d\mu }}\\
& \lesssim \sum_{l}{ \left ( \frac{1}{\mu(B(x,\delta^l))} \int_{B(x,C\delta^l)}{|f|^r d\mu}\right )^{1/r} \left ( \frac{\mu(C^l(x)\cap E_x)}{\mu(B(x,\delta^l))} \right )^{\frac{1}{r'}}      }  \\
& \lesssim  M_{\mu}(|f|^r)(x) ^{1/r} \sum_{l}{\left ( \frac{\rho(x,Q)}{\delta^l} \right )^{\frac{\eta}{r'}} } \lesssim M_{\mu}(|f|^r)(x) ^{1/r},\\
\end{align*}
where the third inequality is obtained by applying the Hölder inequality with $r>1$, and the last one comes from the fact that $\rho(x,Q) \lesssim \delta^l$ for the integers $l$ involved in the summation. This concludes our proof.
\end{proof}

%
%
%

Now let us study more precisely the Hardy property \eqref{3etoile} we introduced in Section $4$. It is not completely clear when \eqref{3etoile} is true in a general space of homogeneous type. It obviously depends on how the sets $B$ and $B^c$ see each other in $X$. By analogy with Christ's dyadic cubes, we shall assume that the outer and inner layers $\{ x \in B | \rho(x,B^c) \leq \varepsilon \}$ and $\{ y \in B^c  | \rho(y,B) \leq \varepsilon  \}$ tend to zero in measure as $\varepsilon \rightarrow 0$, and in a scale invariant way.

\begin{definition}{Layer decay and relative layer decay properties.\\}
Let $(X,\rho,\mu)$ be a space of homogeneous type. For a ball $B$ in $X$, denote $B_{\varepsilon} = \{ x \in B  | \rho(x,B^c) \leq \varepsilon  \} \cup \{ y \in B^c  | \rho(y,B) \leq \varepsilon  \}$ the union of the inner and outer layers.
\begin{itemize}
\item We say that $X$ has the layer decay property if there exist constants $\eta>0$, $C<+\infty$ such that for every ball $B=B(z,r)$ in $X$ and every $\varepsilon >0$, we have
\begin{equation} \label{OLD}
\mu(B_{\varepsilon}) \leq C \left (\frac{\varepsilon}{r} \right )^{\eta} \mu(B(z,r)).
\end{equation}  
\item We say that $X$ has the relative layer decay property if there exist constants $\eta>0$, $C<+\infty$ such that for every ball $B=B(z,r)$ in $X$, every ball $B(w,R)$ not containing $z$, and every $\varepsilon >0$, we have
\begin{equation} \label{LUOLD}
\mu\left (B_{\varepsilon}\cap B(w,R) \right ) \leq C \left (\frac{\varepsilon}{R} \right )^{\eta} \mu(B(w,R)).
\end{equation}  
\end{itemize}
\end{definition}
The layer decay property already appeared in \cite{Buck2} (with only $\mu\left (\{ x \in B  | \rho(x,B^c) \leq \varepsilon  \}  \right )$ in the left hand side of \eqref{OLD}). Remark that \eqref{LUOLD} implies \eqref{OLD}. As a matter of fact, let $z \in X, r>0,  0<\varepsilon<r $, denote $B=B(z,r)$ and suppose that $\varepsilon<r/2$ (else the inequality is trivial). By the Vitali covering lemma, if $B_{\varepsilon}$ is non empty, there exist points $w_k \in B_{\varepsilon}$ such that the balls $B(w_k,r/10)$ are mutually disjoint and $B_{\varepsilon}$ is covered by the union $\cup_k B(w_k,r/2)$. Now, write
\begin{align*}
\mu(B_{\varepsilon}) & = \mu(B_{\varepsilon} \cap (\cup_k {B(w_k,r/2)}))  \leq \sum_k{\mu(B_{\varepsilon} \cap B(w_k,r/2))}\\
& \lesssim \sum_k{\left ( \frac{\varepsilon}{r} \right )^{\eta} \mu(B(w_k,r/2)) } \lesssim \left ( \frac{\varepsilon}{r} \right )^{\eta} \sum_k{\mu(B(w_k,r/10))}\\
& \lesssim \left ( \frac{\varepsilon}{r} \right )^{\eta} \mu(B),\\ 
\end{align*}
where the second inequality comes from \eqref{LUOLD}, the third from the doubling property of $\mu$, and the last from the disjointness of the balls $B(w_k,r/10)$ and the fact that they are all inside a ball of measure comparable to $\mu(B)$. 

Remark also that the condition $z \notin B(w,R)$ cannot be dropped in the relative layer decay property (take a very small $B(z,r)$ contained in a large $B(w,R)$). If $z \in B(w,R)$, the only general conclusion could be a bound like $C({\varepsilon}/r)^{\eta} \mu(B(z,r))$.

It turns out that the relative layer decay property constitutes a sufficient condition for the Hardy property \eqref{3etoile}, as is shown by the following proposition.

\begin{prop}\label{LUOLDP}
Let $(X,\rho,\mu)$ be a space of homogeneous type, and suppose that $X$ has the relative layer decay property. Then the Hardy property \eqref{3etoile} is satisfied on $X$.
\end{prop}

\begin{proof}
The proof is quite similar to the proof of Lemma \ref{Hardy}, with modifications owing to the fact one cannot use exact coverings with balls as for dyadic cubes. Let again $I$ denote the integral in \eqref{3etoile}. Fix a ball $B$ of center $z$ and radius $r>0$, and functions $f,g$ respectively supported on $2B\backslash B$ and $B$ with $f \in L^{\nu}$, $g \in L^{\nu'}$. As before, we can assume that $f,g\geq 0$ and it suffices to prove that for all $1 < \sigma, s < \infty, $ we have
\begin{equation} \label{inthardy}
I \lesssim   \left \langle \left (M_\mu (f^{\sigma}) \right )^{1/\sigma} , g \right \rangle + \left \langle f , \left (M_\mu (g^s) \right )^{1/s} \right \rangle.
\end{equation}
The hypotheses imply $\mu(\overline{B}\backslash B) = 0$ and allow us to write
\[ I  =  \int_{x \in B \atop \rho (x,B^c)>0}{\!\!\!\! g(x) \! \int_{y \in 2B \backslash B  \atop \rho (y,B) \leq \rho (x,B^c) }{\!\! \frac{f(y)}{\lambda(x,y)}d\mu(y)d\mu(x)}} + \int_{y \in 2B \backslash B \atop \rho (y,B)>0}{\!\! f(y) \! \int_{x \in B \atop \rho (y,B) > \rho (x,B^c)}{\!\! \frac{g(x)}{\lambda(x,y)}d\mu(x)d\mu(y)}}. \]
Let us begin by estimating the first term. Fix $x \in B$. As before, let $E_x = \left \{ y \in 2B \backslash B \, | \, \rho (y,B) \leq \rho (x,B^c) \right \}$. We decompose the integral in $y$ over coronae at distance $2^j \rho(x,B^c)$ from $x$:
\begin{align*}
 I_1(x) & =  \int_{y \in E_x }{\frac{f(y)}{\lambda(x,y)} d\mu(y)}  \\
 & \lesssim \ \  \ \ \sum_{j \geq 0} \ \  \ \ \int_{y \in E_x \atop 2^j \rho(x,B^c) \leq \rho(x,y) < 2^{j+1}\rho(x,B^c)}{\frac{f(y)}{\mu(B(x,\rho(x,y)))}d\mu(y)}\\
& \lesssim  \sum_{j\geq 0 \atop z \notin B(x, 2^{j+1}\rho(x,B^c))} M_{\mu}(f^{\sigma})(x) ^{1/\sigma} \left ( \frac{\mu(E_x \cap B(x,2^{j+1}\rho(x,B^c)))}{\mu( B(x,2^{j+1}\rho(x,B^c)))} \right )^{\frac{1}{\sigma'}}\\ 
& + \sum_{j \geq 0 \atop z \in B(x, 2^{j+1}\rho(x,B^c)))} \frac{1}{\mu(B(x,2^{j+1}\rho(x,B^c)))} \int_{y \in E_x \atop 2^j \rho(x,B^c) \leq \rho(x,y) < 2^{j+1}\rho(x,B^c) }{fd\mu},\\
\end{align*}
where the last inequality is obtained applying the Hölder inequality with $\sigma>1$. Observe that there are at most four integers $j \geq 0$ such that $z \in B(x, 2^{j+1}\rho(x,B^c)))$ and $E_x \cap \{ y \in B^c | 2^j \rho(x,B^c) \leq \rho(x,y) < 2^{j+1}\rho(x,B^c) \} \neq \varnothing$. Indeed, let $j_0$ be the first such integer, which implies $\rho(x,z) \leq 2^{j_0 +1} \rho(x,B^c)$, and let $j\geq j_0$ be another one. Let $y \in E_x \cap \{ y \in B^c | 2^j \rho(x,B^c) \leq \rho(x,y) < 2^{j+1}\rho(x,B^c) \}$. Using $y \in E_x$, we have $\rho(z,y) \leq r + \rho(y,B) \leq r + \rho(x,B^c)$. Also, $r \leq \rho(x,z)+\rho(x,B^c)$. Hence $\rho(z,y) \leq \rho(x,z) +2 \rho(x,B^c)$. Using $2^j \rho(x,B^c) \leq \rho(x,y)$, we obtain
\[ 2^j \rho(x,B^c) \leq \rho(x,y)\leq \rho(x,z) + \rho(z,y) \leq 2 \rho(x,z) +2 \rho(x,B^c) \leq  (2^{j_0 +2}+2) \rho(x,B^c), \]
hence $j \leq j_0 +3.$\\
Consequently, applying \eqref{LUOLD}, we have
\[ I_1(x) \lesssim  M_{\mu}(f^\sigma)(x) ^{1/\sigma} \sum_{j\geq 0} {\left ( \frac{\rho(x,B^c)}{2^{j+1}\rho(x,B^c)}\right )^{\frac{\eta}{\sigma'}}} + 4M_{\mu}f(x) \lesssim M_{\mu}(f^\sigma)(x) ^{1/\sigma}, \]
so the first integral is controlled by $\langle g,  (M_\mu (f^{\sigma})  )^{1/\sigma} \rangle$. The argument for the second integral is entirely similar using the symmetry of our assumptions (and $4$ above becomes $3$). This proves \eqref{inthardy}. Remark that this argument only uses \eqref{LUOLD} for $R \leq 2r$.
\end{proof}
\medskip
\begin{remark}
An interesting remark is that if $f$ and $g$ are taken in $L^{\nu_1}(X)$ and $L^{\nu_2}(X)$ with $1/{\nu_1}+1/{\nu_2} <1$ (and $g$ no longer in $L^{\nu_{1}'}$), then if $I$ denotes the integral in \eqref{3etoile}, the normalised inequality 
\[  \frac{1}{\mu(B)} I   \lesssim \|f\|_{L^{\nu_1}(2B \backslash B, \frac{d\mu}{\mu(B)})}  \|g\|_{L^{\nu_2}(B, \frac{d\mu}{\mu(B)})}   \]
is true in any space of homogeneous type only satisfying \eqref{OLD}. The proof is much easier: let $B=B(z,r)$ be a ball in $X$, $f$ be supported in $2B \backslash B$, with $f \in L^{\nu_1}(2B \backslash B)$, $g$ be supported in $B$ with $g \in L^{\nu_2}(B)$, denote $1/{\nu} = 1 - 1/{\nu_1} - 1/{\nu_2}$, $S_j = \{ x \in B: \rho(x,B^c) \leq 2^{-j} r \}$. Observe that by \eqref{OLD}, $\mu(S_j) \lesssim 2^{-j \eta} \mu(B)$, and splitting the integral in $y$ over coronae at distance $2^{-j}r$ from $x$, write
\begin{align*}
I & \lesssim \sum_{j\geq -2} \int_{x \in B}{|g(x)|}\int_{y \in 2B \backslash B \atop 2^{-j-1} r < \rho(x,y) \leq 2^{-j} r}{\frac{|f(y)|}{\mu(B(x,2^{-j}r ))}d\mu(y)d\mu(x)}\\
& \lesssim \sum_{j\geq -2} \int_{x \in S_j}{|g(x)|M_{\mu}{f}(x) d\mu(x)}  \lesssim  \sum_{j\geq -2} \int_{x \in X}{|g(x)|M_{\mu}{f}(x) 1_{S_j}(x) d\mu(x)} \\
& \lesssim  \| f \|_{\nu_1} \|M_{\mu}g\|_{\nu_2} \sum_{j\geq -2} \mu(S_j)^{1/{\nu}},
\end{align*}
where we have applied the Hölder inequality to get the last line. It remains to sum in $j$ to obtain the desired estimate. The same remark with identical proof holds if $B, 2B\backslash B$ are replaced by $Q, \widehat{Q} \backslash Q$ for $Q$ a dyadic cube.
\end{remark}

\medskip

\subsection{Geometric conditions ensuring the relative layer decay property}For obvious practical reasons, we now would like to find sufficient geometric conditions assuring that \eqref{LUOLD} is satisfied. Let us first give a couple of other definitions.

\begin{definition}{Annular decay and relative annular decay properties.\\}
Let  $(X,\rho,\mu)$ be a space of homogeneous type. For $z \in X$ and $r>0,0<s<r$, let $C_{r,r-s}(z) = B(z,r) \backslash B(z,r-s)$. 
\begin{itemize}
\item We say that $X$ has the annular decay property if there exist constants $\eta>0$ and $C<+\infty$ such that for every $z \in X, r>0,0<s<r$, we have
\begin{equation} \label{AD}
\mu(C_{r,r-s}(z)) \leq C \left (\frac{s}{r} \right )^{\eta} \mu(B(z,r)).
\end{equation}
\item We say that $X$ has the relative annular decay property if there exist constants $\eta>0$ and $C<+\infty$ such that for every $z \in X$, $r>0,0<s<r$, and every ball $B(w,R)$ not containing $z$, we have
\begin{equation} \label{LUAD}
\mu ( C_{r,r-s}(z) \cap B(w,R) ) \leq C \left (\frac{s}{R} \right )^{\eta} \mu(B(w,R)). 
\end{equation}
\end{itemize}
\end{definition}
It is interesting to note that this condition \eqref{AD} was made an assumption in \cite{DJS} for the first proof of the global $Tb$ theorem in a space of homogeneous type. Similarly as for layer decay properties, we have that \eqref{LUAD} implies \eqref{AD}. It is also clear that the relative annular decay property implies the relative layer decay property. In \cite{Buck}, Buckley introduces the notion of chain ball spaces and proves that under that condition, a doubling metric measure space has the annular decay property. Colding and Minicozzi II already had proved that this property was satisfied by doubling complete riemannian manifolds in \cite{CM}. Tessera introduced a notion of monotone geodesic property in \cite{Tess}, and proved that this property also entails the annular decay property (called there the Føllner property for balls) in a doubling metric measure space. Lin, Nakai and Yang recently showed in \cite{LNY} that chain ball and a slightly stronger scale invariant version of the monotone geodesic were equivalent. This motivates us to consider the monotone geodesic property of \cite{LNY} in the following. We recall this property \mbox{here:}
\begin{definition}
Let $(X,\rho)$ be a metric space. We say that $X$ has the monotone geodesic property if there exists a constant $0< C<+\infty$ such that for all $u>0$ and all $x,y \in X$ with $\rho(x,y) \geq u$, there exists a point $z \in X$ such that
\begin{equation} \label{monotone} 
\rho(z,y) \leq Cu \quad \mathrm{and} \quad \rho(z,x) \leq \rho(y,x) - u.
\end{equation}
\end{definition}
Remark that $C$ must satisfy $C\geq 1$. Remark also that iterating this property, one gets that for every $x,y \in X$ with $\rho(x,y) \geq u$, there exists a sequence of points $y_0=y, y_1,..., y_m=x$ such that for every $i \in \{0,...,m-1 \}$
\[ \rho(y_{i+1},y_i) \leq Cu \quad \mathrm{and} \quad \rho(y_{i+1},x) \leq \rho(y_i,x) - u.  \]
It appears that the monotone geodesic property not only yields the annular decay property, but also the stronger relative annular decay property as we next show.
\begin{prop}\label{Delta}
Let $(X,\rho,\mu)$ be a space of homogeneous type, and suppose that $X$ has the monotone geodesic property. Then $X$ has the relative annular decay property.
\end{prop}
It immediately yields the following corollary.
\begin{cor}
Let $(X,\rho,\mu)$ be a space of homogeneous type. If $X$ has the monotone geodesic property, then the Hardy property \ref{HardyBalls} is satisfied on $X$.
\end{cor}
\begin{proof}
It is a straightforward application of Proposition \ref{LUOLDP} and Proposition \ref{Delta}.
\end{proof}
\begin{proof}[Proof of Proposition \ref{Delta}]
Let $C_0\geq 1$ be the constant intervening in the monotone geodesic property of the homogeneous type space $X$. We already know that $X$ has the annular decay property (see \cite{Tess} for example), say with constants $C_1, \eta_1>0$. Note that \eqref{AD} is still valid for all $\eta \in [0,\eta_1]$, so that one can make $\eta_1$ as small as we wish. We prove \eqref{LUAD}. The argument adapts the one in \cite{Tess}.  Fix $z \in X$, $r>0,0<s<r$, and $w \in X$, $R>0$ such that $z \notin B(w,R)$. Observe that if $6C_0 s \geq R $, \eqref{LUAD} is trivial, so that we can freely assume $6C_0 s < R$ in the following. We will prove \eqref{LUAD} in two steps. We first show that there exist $\theta <1$ and $C<+\infty$ such that for all $\sigma >0$ with $6C_0 \sigma < R$, $2\sigma < r$ (and $R< \rho(w,z)$), we have
\begin{equation} \label{LUAD1}
\mu(C_{r,r-\sigma}(z) \cap B(w,R)) \leq \theta \mu(C_{r,r-2\sigma}(z) \cap B(w,R)) + C \left ( \frac{\sigma}{R}\right )^{\eta_1} \mu(B(w,R)).
\end{equation}
Assuming these conditions on $\sigma$ are met, we begin by establishing the existence of a dimensional constant $A>0$ such that
\[ \mu(C_{r,r-\sigma}(z) \cap B(w, R-6C_0 \sigma) \leq A\mu(C_{r-\sigma,r-2\sigma} \cap B(w,R)). \]
 If $C_{r,r-\sigma}(z) \cap B(w,R-6C_0 \sigma) = \varnothing$, there is nothing to do, so we will suppose that this intersection is not empty. Let $y \in C_{r,r-\sigma}(z) \cap B(w, R-6C_0 \sigma)$, and let $0<u < \min(\sigma/4,C_0(r-\sigma))$. Since $\rho(z,y) \geq r-\sigma >u/{C_0}$, by monotone geodesicity of $X$ there exist points $y_0 =y, y_1,..., y_m =z$ such that for every $i \in \{ 0,..., m-1 \}$,
\[ \rho(z,y_{i+1}) \leq \rho(z,y_i) - \frac{u}{C_0} \quad \mathrm{and} \quad \rho(y_i,y_{i+1}) \leq u. \]
Let $i_0$ be the first integer such that $y_{i_0} \in B(z, r-3\sigma/2)$. Note that $u<\sigma/4$ implies $i_0 >2.$ Write
\[ r - 3\sigma/2 \leq \rho(z, y_{i_0 -1}) \leq \rho(z,y) - \frac{(i_0 -1)u}{C_0} \leq r -  \frac{(i_0 -1)u}{C_0},  \]
so that $\frac{3\sigma}{2} \geq \frac{(i_0 - 1)u}{C_0}$, and
\[ i_0 u \leq \frac{i_0}{i_0 -1} \frac{3C_0 \sigma}{2} < 3 C_0 \sigma, \]
because $i_0 >2$. Thus $\rho(y, y_{i_0}) \leq i_0 u < 3 C_0 \sigma$ and it follows that $y_{i_0} \in B(w,R-3C_0 \sigma)$. Furthermore,
\[ \rho(z,y_{i_0}) \geq \rho(z,y_{i_0 -1}) - \rho(y_{i_0 -1}, y_{i_0}) \geq r - \frac{3\sigma}{2} - u \geq r - \frac{7\sigma}{4}.\]
It means that for every $y \in C_{r,r-\sigma}(z) \cap B(w, R-6C_0 \sigma)$, there exists $y' \in C_{r-3\sigma/2, r- 7\sigma/4}(z) \cap B(w,R-3C_0 \sigma)$ such that $\rho(y,y') < 3 C_0 \sigma.$ Then, 
\[ B_y = B(y', \sigma/4) \subset C_{r-\sigma,r-2\sigma}(z) \cap B(w,R). \]
Now if we consider the union of all those balls $B_y$ for $y \in C_{r,r-\sigma}(z) \cap B(w,R-6C_0 \sigma)$, by the Vitali covering lemma we can find points $y_k$ such that the balls $B_{y_k}$ are mutually disjoint and $\bigcup_{y}B_y \subset \bigcup_k 5B_{y_k}$. But since for every $k$, $\rho(y_k,y'_k) < 3 C_0 \sigma$, 
\[ C_{r,r-\sigma}(z) \cap B(w, R-6C_0 \sigma) \subset \bigcup_{k}{CB_{y_k}}, \]
with $C= 5 + 12C_0$. Applying the doubling measure $\mu$ and remembering that the $B_{y_k}$ are mutually disjoint sets contained in $C_{r-\sigma,r-2\sigma}(z) \cap B(w,R)$, we obtain
\[ \mu(C_{r,r-\sigma}(z) \cap B(w,R-6C_0 \sigma)) \leq \sum_k{\mu(CB_{y_k})} \leq A \sum_k{\mu(B_{y_k})} \leq A \mu(C_{r-\sigma,r-2\sigma}(z) \cap B(w,R)). \]
Applying the annular decay property inside the ball $B(w,R)$, write then
\begin{align*}
& \mu(C_{r,r-\sigma}(z) \cap B(w,R))  = \mu(C_{r,r-\sigma}(z) \cap B(w,R-6C_0 \sigma)) + \mu(C_{R,R-6C_0 \sigma}(w))\\
& \leq A \mu(C_{r-\sigma,r-2\sigma}(z) \cap B(w,R)) + C_1 \left ( \frac{6C_0 \sigma}{R}\right )^{\eta_1} \mu(B(w,R))\\
& \leq A ( \mu(C_{r,r-2\sigma}(z) \cap B(w,R) ) - \mu(C_{r,r-\sigma}(z) \cap B(w,R))  ) + C_1  \left ( \frac{6C_0 \sigma}{R}\right )^{\eta_1} \mu(B(w,R)),\\
\end{align*}
so that \eqref{LUAD1} holds for $\theta = \frac{A}{1+A} <1$ and $C=\frac{C_1 (6C_0)^{\eta_1}}{1+A} < +\infty.$

The second step in the proof now consists in iterating this inequality. Let  $d = \rho(w,B(z,r)) \geq 0$. Let us distinguish two cases. If $d \geq \left (1 -\left ( \frac{s}{R} \right )^{\frac{1}{2}} \right )R$, then $C_{r,r-s}(z) \cap B(w,R) \subset C_{R,R-(\frac{s}{R})^{1/2} R}(w)$ as one easily checks. Applying the annular decay property in the ball $B(w,R)$, we have  
\[ \mu(C_{r,r-s}(z) \cap B(w,R)) \leq  C_1 \left ( \frac{s}{R} \right )^{\frac{\eta_1}{2}} \mu(B(w,R)), \] 
so that there is no iteration to make. \\
Suppose now that $d< \left (1-\left ( \frac{s}{R} \right )^{\frac{1}{2}} \right )R$. Observe that $R \leq \rho(w,z) \leq r +d$, so that $R - d \leq r$. Thus, if $n$ is an integer such that $2^n s < \frac{R-d}{3C_0 }$, we also have $2^n s < \frac{r}{3C_0} < r$, so that we can apply \eqref{LUAD1} with $\sigma = 2^{k-1} s$ for $k=1,...,n$, and iterate to obtain
\[ \mu(C_{r,r-s}(z) \cap B(w,R)) \leq \theta^n \mu(C_{r,r-2^n s}(z) \cap B(w,R)) + \left ( \sum_{k=0}^{n-1}{\theta^k 2^{k \eta_1}} \right ) C \left ( \frac{s}{R}\right )^{\eta_1} \mu(B(w,R)). \]
Consequently, let $n_0 =    [ \log_2(\frac{ R-d}{3C_0 s})] - 1$, where $[x]$ denotes the integer part of $x$, note that $R-d > (\frac{s}{R})^{1/2} R$, and write
\begin{align*}
\mu(C_{r,r-s}(z) \cap B(w,R)) & \leq  \theta^{n_0} \mu(C_{r,r-2^{n_0} s}(z) \cap B(w,R)) + \left ( \sum_{k=0}^{n_0 -1}{\theta^k 2^{k \eta_1}}\right ) C \left ( \frac{s}{R}\right )^{\eta_1} \mu(B(w,R))\\
& \leq  \frac{1}{\theta^2} \left ( \frac{3C_0 s}{(\frac{s}{R})^{1/2}R} \right )^{-\log_2{\theta}} \mu(B(w,R)) +  \frac{C}{1-\theta 2^{\eta_1}} \left ( \frac{s}{R}\right )^{\eta_1} \mu(B(w,R))\\
& \lesssim  \left ( \frac{s}{R}\right )^{\frac{\eta_1}{2}}  \mu(B(w,R)),\\
\end{align*}
provided $\eta_1 < -\log_2{\theta}$, which we may assume by a previous remark, and because we have supposed $6C_0 s < R$.
\end{proof}

\begin{remark} \begin{enumerate}
\item We have proven the desired result with $\eta = \eta_1/2$ and $\eta_1 < -\log_2{\theta}$, but remark that we can clearly get $\eta= \eta_1 - \varepsilon$ for $\varepsilon>0$ as small as desired (provided $\eta_1< -\log_2{\theta}$). 
\item We have obtained a rather satisfying geometric condition ensuring that the Hardy property \ref{HardyBalls} is \mbox{satisfied:} this monotone geodesic property is obviously satisfied by complete riemannian manifolds, or by length spaces for example. We have also shown the stronger relative annular decay property in such cases.

\item Observe that conversely, the Hardy property does not imply the monotone geodesic property. Let us give two basic examples to illustrate this. First consider the space formed by the real line from which an arbitrary interval has been withdrawn, fitted with the usual euclidean distance and Lebesgue measure. This space obviously does not have the monotone geodesic property, as, to put it roughly, there is a hole in it. On the other hand, this space clearly has the Hardy property, as a consequence of the Hardy property being satisfied on the real line. The second example is a connected one: consider the space made of the three edges of an arbitrary triangle in the plane, again fitted with the euclidean distance and Lebesgue measure. This space has the Hardy property, once again as a straightforward consequence of the fact that the unit circle has it and easy change of variables. It easily follows from the fact that one of the angles must be less than $\pi/2$ that it does not have the monotone geodesic property: one of the pairs $(x,y)$ with $x$ a vertex and $y$ its orthogonal projection on the opposite side cannot meet condition \eqref{monotone}. In passing, it proves that this property is not stable by bi-Lipschitz mappings (see also \cite{Tess}).
\item The study of the first example we gave above motivates the following observation: to prove the Hardy inequality on a space of homogeneous type $X$, it is enough to prove it on a larger space $X_0$ of which $X$ is a subset, provided that $X_0$ hasn't got too much weight compared to $X$. Let us precise that. Let $(X_0,\rho,\mu)$ be a space of homogeneous type, and let $X \subset X_0$, fitted with the distance $\rho' = \rho |_X$ and the measure $\nu = \mu |_X$. Suppose furthermore that the following measure compatibility condition is satisfied: for all $x \in X$ and $r>0$, $\mu(B_{X_0}(x,r)) \lesssim \mu(B_{X_0}(x,r) \cap X) = \nu(B_X(x,r))$. Then $X$ inherits the Hardy property \ref{HardyBalls} from $X_0$. Indeed, under that condition, the kernels taken respectively on $X$ and $X_0$, which we denote $K_X$ and $K_{X_0}$, have comparable size: if $x,y \in X$,
\[ K_X (x,y) \approx \sup_{B_X \ni x,y}{\frac{1}{\nu(B_X)}} = \sup_{B_{X_0} \ni x,y \atop B \, \mathrm{centered} \, \mathrm{in} \, X}{\frac{1}{\mu(B_{X_0} \cap X)}} \lesssim \sup_{B_{X_0} \ni x,y}{\frac{1}{\mu(B_{X_0})}} \approx K_{X_0}(x,y). \]
It is clear then that the Hardy inequality is in that sense stable by restriction, provided one does not withdraw too much from the initial space. To check that Property \ref{HardyBalls} is satisfied by a space of homogeneous type $X$, it is thus enough to see if $X$ can be seen as a subset of a bigger space, satisfying the measure compatibility condition, on which the Hardy inequality is known to be true, or on which one can prove the monotone geodesic property for example.
\end{enumerate}
\end{remark}
\bigskip

 \clearpage

\appendix

\section{Wavelet representation of the adapted martingale difference operators}

\bigskip

We prove in this section Proposition \ref{representation}. 

%

\begin{proof}[Proof of Proposition \ref{representation}]
It follows essentially from results on pseudo-accretive sesquilinear forms developed in \cite{AT}. It is here simpler as everything is finite dimensional. Still, the uniform control of constants must be achieved. Let $V_Q$ be the space of complex-valued functions which are constant on each dyadic child of $Q$ (recall we assume there are at least two children), seen as a subspace of $L^2(X)$ and equipped with the complex $L^2$ inner product. On $V_Q$ consider the bilinear form 
\[ B(f,g) = \int_Q {f \, b \, g \, d\mu} = \sum_{Q' \in \widetilde{Q}}{ [f]_{Q'} [b]_{Q'} [g]_{Q'} \mu(Q') }.  \]
Clearly, since for $f \in V_Q$, $\|f\|_2^2 = \sum{[|f|^2]_{Q'} \mu(Q')}$ and $[|f|^2]_{Q'} = [f]_{Q'} \overline{[f]_{Q'}},$ we have
\[ |B(f,g)| \leq \sup_{Q' \in \widetilde{Q}} {|[b]_{Q'}| \ \ \|f\|_2 \ \ \|g\|_2}, \]
and 
\[ \inf_{\|g\|_2 = 1}{|B(f,g)|} \geq \inf_{Q' \in \widetilde{Q}} {|[b]_{Q'}| \ \ \|f\|_2} \]
by taking $g \in V_Q$ with 
\[ [g]_{Q'} = \frac{\overline{[f]_{Q'}}}{\|f\|_2}  \frac{\overline{[b]_{Q'}}}{|[b]_{Q'}|}. \]
Consider $W_Q \subset V_Q$ equal to the orthogonal complement of $\mathbb{C} 1_Q$ and $P_{W_Q}$ the orthogonal projection of $V_Q$ onto $W_Q$. Let also $X_Q = \{ g \in V_Q ; B(1_Q,g) = 0  \}.$ Then $\mathbb{C}1_Q \oplus X_Q = V_Q$ is a topological sum. Indeed, let $c_Q \in \mathbb{C}$ such that $c_Q^2 \frac{B(1_Q,1_Q)}{\mu(Q)} = 1$ which is possible because $|B(1_Q,1_Q)| = |[b]_Q \mu(Q)| \neq 0$. Then the splitting is given by $f = \frac{c_Q^2}{\mu(Q)} B(f,1_Q) 1_Q + g$, and 
\[ \left \| \frac{c_Q^2}{\mu(Q)} B(f,1_Q) 1_Q \right  \|_2  \leq |c_Q^2| \sup_{Q' \in \widetilde{Q}} {|[b]_{Q'}| \|f\|_2 \|g\|_2} \leq (1+|c_Q^2| \sup_{Q' \in \widetilde{Q}}{|[b]_{Q'}|}) \|f\|_2. \]
Let $P_{X_Q}$ be the projector on $X_Q$ associated to this splitting. Then it can be checked that $P_{X_Q}: W_Q \rightarrow X_Q$ is an isomorphism with inverse $P_{W_Q}: X_Q \rightarrow W_Q.$ Set $\phi_Q^0 = \frac{1_Q}{\mu(Q)^{1/2}}$ and complete it to an orthonormal basis of $W_Q$, $\phi_Q^1, ... , \phi_Q^{N_Q -1}$ for $\|.\|_2$. Define $\widetilde{\phi}_Q^{b,s}$ by $\widetilde{\phi}_Q^{b,0} = c_Q \phi_Q^0,$ $\widetilde{\phi}_Q^{b,s} = P_{X_Q}(\phi_Q^s)$ for $s\geq 1$. For $f\in X_Q \backslash \{0\}$, let $g \in V_Q$, $\|g\|_2 = 1$, such that $B(f,g) \geq \inf_{Q' \in \widetilde{Q}}{|[b]_{Q'}|} \|f\|_2$ as before. Let $h= \frac{P_{X_Q}g}{\|P_{X_Q}g\|_2},$ as $P_{X_Q}g \neq 0$ because otherwise $g \in \mathbb{C} 1_Q$ and $B(f,g)=0$. Thus $h \in X_Q$, $\|h\|_2 = 1$, and 
\[ B(f,h) = \frac{B(f,g)}{\|P_{X_Q}g\|_2} \geq \frac{\inf_{Q' \in \widetilde{Q}}{|[b]_{Q'}|}}{\|P_{X_Q}g\|_2} \|f\|_2.  \]
By the Riesz representation theorem, there exists $A \in \mathcal{L}(X_Q)$, $\|A\| \leq \sup_{Q' \in \widetilde{Q}}{|[b]_{Q'}|},$ such that for all $f,g \in X_Q$
\[ B(f,g) = \langle Af, \overline{g} \rangle = \int_Q{Af \, g \, d\mu} \]
and the above argument says that $A^{-1} \in \mathcal{L}(X_Q),$ and 
\[ \|A^{-1}\| \leq  \frac{\|P_{X_Q}g\|_2}{\inf_{Q' \in \widetilde{Q}}{|[b]_{Q'}|}}. \]
Set then $\phi_Q^{b,s} = {}^t A^{-1} \widetilde{\phi}_Q^{b,s} $ for $1\leq s \leq N_Q -1$ where ${}^t A $ is the real transpose of $A$. It is quite clear that $\phi_Q^{b,s}$ and $ \widetilde{\phi}_Q^{b,s}$ satisfy $(1),(2),(3),(4).$ It remains to prove $(5)$ and $(6)$. Remark that for $f \in X_Q$, we have obtained 
\[ f = \sum_{s=1}^{N_Q - 1} {B(f,\phi_Q^{b,s})  \widetilde{\phi}_Q^{b,s}} \]
and $\|f\|_2 \approx \sum{|B(f,\phi_Q^{b,s})|^2}$ (use $f = P_{X_Q}(\sum{B(f,\phi_Q^{b,s})\phi_Q^s })$ and the bounds for the operators $P_{X_Q}$ and $P_{W_Q}$). Now, let $f \in L^1(X)$ and observe that $\Delta_Q^b f \in V_Q$ and $\int_Q{b \Delta_Q^b f d\mu} = 0.$ Hence, $\Delta_Q^b f \in X_Q.$ It remains to show that $B(\Delta_Q^b f, \phi_Q^{b,s}) = \int_Q{f \phi_Q^{b,s} d\mu}$ for $s\geq 1$. Indeed, $B(E_Q^b f, \phi_Q^{b,s}) = 0$, so that
\[ B(\Delta_Q^b f, \phi_Q^{b,s}) = \sum_{Q' \in \widetilde{Q}}{B \left (\frac{[f]_{Q'}}{[b]_{Q'}} 1_{Q'}, \phi_Q^{b,s} \right )} = \sum_{Q' \in \widetilde{Q}}{[f]_{Q'} [\phi_Q^{b,s}]_{Q'} \mu(Q')  } = \int_Q{f \phi_Q^{b,s} d\mu}, \]
where we have used the fact that $\phi_Q^{b,s}$ is constant on each child $Q'$ of $Q$.
\end{proof}

\bigskip

\section{Coefficient estimates}

\medskip

Let $T$ be a singular integral operator. Let $Q$ and $R$ be two cubes of the space of homogeneous type $X$. Remember that we introduced
\[ \mu(Q,R) = \inf_{x \in Q, y \in R}\left \{ \mu(B(x,\rho(x,y))), \mu(B(y,\rho(x,y))) \right \}, \]
and
\[ \alpha_{Q,R} = \begin{cases}
 \mu(Q)^{\frac{1}{2}} \mu(R)^{\frac{1}{2}} \left ( {\frac{\inf(l(Q),l(R))}{\rho(Q,R)}} \right )^{\alpha} \frac{1}{\mu(Q,R)}  \quad & \mathrm{if} \ \ \rho(Q,R) \geq \sup(l(Q), l(R)) \\
1 \quad & \mathrm{if} \ \ l(Q) = l(R), \ \ \mathrm{and} \ \ \rho(Q,R) < l(Q).
 \end{cases}\\ \] 
We state the following lemma:

\begin{lem}\label{estimates}
Let us assume the hypotheses of Theorem \ref{AR}, and let $Q,R$ be two pseudo-accretive dyadic cubes in $X$, either strongly pseudo-accretive or degenerate pseudo-accretive with respect to $b^1,b^2$ respectively, and depending on the nature of the inequality below. Let also $P^1_i$ be one of the stopping cubes with respect to $b^1$. Then we have the following coefficient estimates: whenever $l(Q)=l(R)$,

\begin{equation}
\left |  \langle  b^2  \phi_Q^{b^2} , T(b^1  \phi_R^{b^1} )    \rangle \right | \lesssim \alpha_{Q,R}, 
\label{eqn:1}
\end{equation}
\begin{equation}
\left | \langle  b^2  \phi_Q^{b^2} , T(\xi^1_R g )    \rangle \right | \lesssim \alpha_{Q,R} \, \mu(R)^{\frac{1}{2}}  ,  \label{eqn:2}
\end{equation}
\begin{equation}
\left | \langle  \xi^2_Q f , T(\xi^1_R g )    \rangle \right | \lesssim \alpha_{Q,R} \, \mu(Q)^{\frac{1}{2}} \, \mu(R)^{\frac{1}{2}}  , \label{eqn:3}
\end{equation}
\begin{equation}
\left | \langle \xi^2_Q f, T( b^1 1_R)\rangle \right | \lesssim   \alpha_{Q,R} \, \mu(Q)^{\frac{1}{2}} \, \mu(R)^{\frac{1}{2}}, \label{eqn:4}
\end{equation}
\begin{equation}
\left | \langle b^2 \phi_Q^{b^2} , T(b^1 1_R)\rangle \right | \lesssim \alpha_{Q,R} \, \mu(R)^{\frac{1}{2}} , \label{eqn:5}
\end{equation}
and when $\rho(Q,P^1_i) \geq l(P^1_i) \geq l(Q),$ 

\begin{equation}
\left | \langle \xi^2_Q f, T(b^1_{P^1_i})\rangle \right | \lesssim \alpha_{Q,P^1_i} \, \mu(Q)^{\frac{1}{2}} \, \mu(P^1_i)^{\frac{1}{2}},\label{eqn:6}
\end{equation}
\begin{equation}
\left |\langle b^2 \phi_Q^{b^2} , T(b^1_{P^1_i})\rangle \right | \lesssim  \alpha_{Q,P^1_i} \, \mu(P^1_i)^{\frac{1}{2}} , \label{eqn:7}
\end{equation}
\begin{equation}
\left |\langle b^2 \phi_Q^{b^2} , T(b^1 1_{P^1_i})\rangle \right | \lesssim  \alpha_{Q,P^1_i} \, \mu(P^1_i)^{\frac{1}{2}} , \label{eqn:8}
\end{equation}
\end{lem}

\smallskip

\begin{proof}
We detail the proof of the first inequality. Suppose first that we have $\rho(Q,R) \geq l(Q)$. Since $ b^2 \phi_Q^{b^2} $ and $b^1 \phi_R^{b^1}$ have mean $0$, we can apply \eqref{standard estimate dyadic} to get the bound
\begin{align*}
 \left | \langle  b^2 \phi_Q^{b^2} , T(b^1 \phi_R^{b^1} )    \rangle \right | &  \lesssim \int_Q{|b^2 \phi_Q^{b^2}  |d\mu} \int_R{| b^1 \phi_R^{b^1}  |d\mu}   \left ( {\frac{\inf(l(Q),l(R))}{\rho(Q,R)}} \right )^{\alpha} \frac{1}{\mu(Q,R)}\\
 & \lesssim   \alpha_{Q,R} , \\
 \end{align*}
the last inequality being a consequence of the $L^{\infty}$ estimates of the functions $\phi_Q^{b^2}, \phi_R^{b^1}$. Indeed, Lemma \ref{representation} entails that $\|\phi_Q^{b^2}\|_{\infty} \lesssim \mu(Q)^{-1/2}, \|\phi_R^{b^1}\|_{\infty} \lesssim \mu(R)^{-1/2}$. Thus $\| b^2 \phi_Q^{b^2} \|_{L^1(Q)} \lesssim \mu(Q)^{-1/2} \|b^2\|_{L^1(Q)} \lesssim \mu(Q)^{1/2}$ since $Q$ is pa $2$, and similarly for $b^1 \phi_R^{b^1}$.\\
Now, assume that $\rho(Q,R) < l(Q) $. Write
\[ \langle  b^2 \phi_Q^{b^2} , T(b^1 \phi_R^{b^1} )    \rangle  = \sum_{Q' \in \widetilde{Q},R' \in \widetilde{R}}{[ \phi_Q^{b^2}   ]_{Q'} \langle b^2 1_{Q'}  , T(b^1 1_{R'})  \rangle [\phi_R^{b^1}  ]_{R'}  }. \]
If $Q'$ and $R'$ are neighbors, by the weak boundedness property \eqref{WBPC} (we do not need the stronger property \eqref{WBP1} here), we have
\[ |  \langle b^2 1_{Q'}  , T(b^1 1_{R'})  \rangle |  \lesssim \mu(R'). \]
If $\rho(Q',R') \geq l(R')$, then write
\begin{align*}
|\langle b^2 1_{Q'}, T(b^1 1_{R'}) \rangle | & \leq | \langle b^2 1_{Q'} , T(b^1 1_{R'} - [b^1]_{R'} 1_{R'}) \rangle| + |[b^1]_{R'}  | |\langle b^2 1_{Q'} , T(1_{R'}) \rangle | \\
& \lesssim \mu(Q') \mu(R') \left ( \frac{l(R')}{\rho(Q',R')} \right )^{\alpha} \frac{1}{\mu(Q',R')}  + \mu(R')\\
& \lesssim \mu(R').
\end{align*}
To get the second inequality, we have applied \eqref{standard estimate dyadic} to the first term (which is possible because $\int{(b^1 1_{R'} - [b^1]_{R'} 1_{R'})d\mu} = 0$), and the Hardy inequality \eqref{Hardy1} to the second term. Then the last inequality comes from the fact that $\rho(Q',R') \geq l(R')$, and $\mu(Q',R') \approx \mu(\widehat{Q}) \approx \mu(Q')$ as both $Q'$ and $R'$ are children of the neighbor cubes $Q,R$. Finally, using again the $L^{\infty}$ estimates of the functions $\phi$, and the fact that $Q',R'$ have comparable measure, we have
\[ \left |  \langle  b^2 \phi_Q^{b^2} , T(b^1 \phi_R^{b^1} )    \rangle \right |  \lesssim  \sum_{Q' \in \widetilde{Q}, R' \in \widetilde{R}}{\mu(Q')^{-\frac{1}{2}} \mu(R') \mu(R')^{-\frac{1}{2}} } \lesssim 1, \]
which was the desired estimate.\\
For the other estimates we asserted, it is exactly the same idea. Indeed, observe that we always have at least one of the functions in the scalar product which has mean zero, allowing us to conduct exactly the same argument as above when the cubes $Q$ and $R$ (or $P^1_i$) are far away. Remark that the only thing that changes is the normalization of those functions, but it is immediate to see that
\[  \int_R{| \xi^1_R g | d\mu }  \lesssim \mu(R), \quad \int_Q{| \xi^2_Q f | d\mu} \lesssim \mu(Q), \quad \int_{R}{|b^1 1_R|} \lesssim \mu(R), \quad  \int_{P^1_i}{|b^1_{P^1_i}|} \lesssim \mu(P^1_i),  \]
which yields all the desired estimates when the cubes are far away. Now when $Q$ and $R$ are neighbors, it is also the same idea as above, and one only has to apply the weak boundedness property \eqref{WBPC} to get the desired estimates. We therefore omit the detail here.
\end{proof}

\bigskip
\bigskip

\bibliographystyle{short}	
\bibliography{myrefs}

\end{document}